\documentclass{amsproc}
\usepackage{breqn,float,amssymb,pdflscape,rotating,hyperref,multirow,xcolor,setspace,graphicx}
\makeatletter
\@namedef{subjclassname@2020}{%
  \textup{2020} Mathematics Subject Classification}
\makeatother
\newtheorem{theorem}{Theorem}[section]

\newtheorem{proposition}[theorem]{Proposition}

\theoremstyle{definition}

\newtheorem{example}[theorem]{Example}

\theoremstyle{remark}

\allowdisplaybreaks
\begin{document}
\onehalfspacing
%
\title{Infinite integrals in terms of series}


\author{Robert Reynolds}
\address[Robert Reynolds]{Department of Mathematics and Statistics, York University, Toronto, ON, Canada, M3J1P3}
\email[Corresponding author]{milver73@gmail.com}
\thanks{}


\subjclass[2020]{Primary  30E20, 33-01, 33-03, 33-04}

\keywords{Cauchy contour integral, generalized infinite integral, Hurwitz-Lerch zeta function, logarithm function, singularity, Cauchy principal value, special functions, Stirling numbers of first and second kind, finite series, $q$-Pochhammer symbol, Pochhammer symbol, Ramanujan integral, Schr\"{o}der's integral formula, Gregory coefficients, Schr\"{o}der factor}

\date{}

\dedicatory{}

\begin{abstract}
In this work we derive and evaluate some infinite integrals involving the product of a generalized logarithm, polynomial, and special functions in the denominator. These integrals are expressed in terms of finite series involving the Hurwitz-Lerch zeta function $\Phi(z,s,a)$. We produce special cases of these integrals in terms of other special functions and fundamental constants.
\end{abstract}

\maketitle
\section{Significance Statement}
Definite integrals over the positive real axis with generalized logarithmic and rational and special functions are used by researchers in science and engineering. Some of these used cases are in the calculation of areas, probabilities and the evaluation of physical quantities. These integrals also aid in the study of singularities and the analysis of discontinuous functions over infinite intervals. These integrals are used in the modelling of probability distributions, electric fields and gravitational and nuclear forces. These types of integrals also play a key role in the analysis of convergence theory by assisting with the study of limits and series. Generalized integrals can also be used to derive transformations involving special functions with arguments involving basic arithmetic operations. Infinite integrals also form the basis of well known transforms such as; the Abel transform [Wolfram MathWorld, \href{https://mathworld.wolfram.com/AbelTransform.html}{(1)}], Aboodh transform [Aruldoss et al. \href{https://www.ripublication.com/gjpam20/gjpamv16n2_01.pdf}{(2)}], Bateman transform [Bateman, \href{https://doi.org/10.1112/plms/s2-1.1.451}{(1)}], Fourier transform [Wolfram MathWorld, \href{https://mathworld.wolfram.com/FourierTransform.html}{(1)}], Short-time Fourier transform [Hosseini Giv,\href{https://doi.org/10.1016/j.jmaa.2012.09.053}{(1.2)}], Gabor transform [Debnath et al.,\href{https://doi.org/10.1007/978-0-8176-8418-1_4}{pp 243-286}], Hankel transform [Wolfram MathWorld,\href{https://mathworld.wolfram.com/HankelTransform.html}{(2)}], Hartley transform [Wolfram MathWorld,\href{https://mathworld.wolfram.com/HartleyTransform.html}{(2)}], Hermite transform [,\href{https://doi.org/10.1112/jlms/s1-13.1.22}{pp 22-29}], Hilbert transform [,\href{https://doi.org/10.1007/BF01171098}{pp 218-224}], Jacobi transform [DLMF ,\href{https://dlmf.nist.gov/15.9.E12}{(15.9.12)}], Laguerre transform [McCully, \href{https://doi.org/10.1137/1002040}{(4)}], Laplace transform [Wolfram MathWorld, \href{https://mathworld.wolfram.com/LaplaceTransform.html}{(1)}], Mellin transform [Wolfram MathWorld, \href{https://mathworld.wolfram.com/MellinTransform.html}{(1)}], Radon transform [Wolfram MathWorld, \href{https://mathworld.wolfram.com/RadonTransform.html}{(1)}], Stieltjes transformation [Y\"{u}rekli,\href{https://doi.org/10.1016/0022-247X(92)90189-K}{(1)}], Wavelet transform [Nakayama, \href{https://doi.org/10.1007/b138494_14}{(1)}], and Weierstrass transform [Venugopal,\href{https://doi.org/10.1006/jmaa.1997.5836}{(1)}] to name a few. In the late 19th century, German mathematician Ernst Schr\"{o}der \cite{schroder_dipert}, contributed to the theory of special sequences through his work on integral representations.  He explored the connections between infinite series and definite integrals, particularly focusing on the coefficients that arise in the expansions of logarithmic functions. These coefficients, now known as Gregory coefficients [Wolfram MathWorld, \href{https://mathworld.wolfram.com/BernoulliNumberoftheSecondKind.html}{(2)}], had already appeared in numerical analysis and interpolation formulas introduced by James Gregory in the 17th century.  They serve as correction terms in quadrature rules and are closely related to Bernoulli numbers. Schr\"{o}der expressed these coefficients using an improper integral involving a logarithmic kernel. This approach connected discrete sequences to continuous analytic structures, reflecting the influence of complex analysis and contour integration techniques of the time. His formula offered a new perspective between series expansions, special functions, and integral transforms, which prove useful in analytic number theory and approximation methods. In this work we further previous research into generalized kernels and higher-order logarithmic terms. These new developments involve polylogarithms, Hurwitz-Lerch zeta functions, and Mellin transforms.
\section{Introduction}
In this work definite integrals over $\mathbb_{R_{+}}$ involving products of generalized logarithm and rational functions are studied. The logarithmic function has a complex number order and raised to a complex number $\log^k(ax)$. The rational functions have the form $\frac{1}{(ax^v+b)^{n+1}}$ and $\frac{1}{(ax^v+b)_{n+1}}$ where $a,b,v \in\mathbb{C}$ and $n$ is any positive integer. The definite integral is expressed in terms of the finite sum involving the Hurwitz-Lerch zeta function [DLMF, \href{https://dlmf.nist.gov/25.14}{25.14.1}]. This generalized type of definite integral makes its special case forms easy to study. Such special cases were previously studied and tabled by Bierens de Haan \cite{bdh}, Malmsten \cite{malmsten}, Gr\"{o}bner \cite{grobner}, Prudnikov \cite{prud1}, Erdeyli et al.\cite{erd_t1,erd_t2}, Oberhettinger \cite{ober_f,ober_m}, Brychkov et al. \cite{brychkov_m}, Moll \cite{moll_1,moll_2,moll_3}, Adamchik [Adamcik,\href{https://dl.acm.org/doi/pdf/10.1145/258726.258736}{(1)}], Blagouchine \cite{blagouchine}, Ripon et al \cite{ripon}, Reynolds and Stauffer \cite{reyn5} to name a few.\\\\
Infinite integrals over $\mathbb{R_{+}}$ of generalized logarithm and polynomial functions are  listed by Bierens de Haan in Table (135) in \cite{bdh}. In his Table (135), entries (7) to (13) involve the product of a logarithm function and a generalized quotient polynomial function expressed in terms of trigonometric functions and fundamental constants such as $\pi$ . In the book of Prudnikov \cite{prud1}, Table (2.6.4) entries (6) to (16) involve the Mellin transform of the product of a generalized logarithm function and a generalized polynomial function. These integrals are expressed in terms of the $n$-th derivative of the Beta function [Wolfram MathWorld, \href{https://mathworld.wolfram.com/BetaFunction.html}{(1)}], the polygamma function [Wolfram MathWorld, \href{https://mathworld.wolfram.com/PolygammaFunction.html}{(1)}], the finite sum of quotient linear functions, trigonometric functions, infinite series involving the Pochhammer symbol and reciprocal polynomials, Euler numbers, [Wolfram MathWorld, \href{https://mathworld.wolfram.com/EulerNumber.html}{(1)}]. In \cite{prud1}, Table (2.6.5) entries (14) to (24) involve the product a generalized logarithm function and a quotient linear and quadratic polynomial function expressed in terms the polylogarithm function [Wolfram MathWorld, \href{https://mathworld.wolfram.com/Polylogarithm.html}{(1)}], trigonometric functions, and fundamental constants. Table (2.6.8) entries (9) to (18) lists integrals involving the product of generalized logarithm function and quotient rational functions involving square roots of polynomial functions expressed in terms of Catalan's constant [Wolfram MathWorld, \href{https://mathworld.wolfram.com/CatalansConstant.html}{(1)}], polygamma function, complete Elliptic integral of the first kind, [Wolfram MathWorld, \href{https://mathworld.wolfram.com/CompleteEllipticIntegraloftheFirstKind.html}{(1)}], and the $n$-th partial derivative of the ratio of gamma functions [Wolfram MathWorld, \href{https://mathworld.wolfram.com/GammaFunction.html}{(1)}].
\section{Preliminaries}
This section provides a list of the formulae, special functions, and Cauchy contour integral representations used in this work. Cauchy’s integral formula is a fundamental theorem in complex analysis. It provides a method to find the value of a holomorphic (analytic) function at any point inside a simple closed curve, using only the function’s values on the curve’s boundary \cite{ahlfors}. Special functions are a class of mathematical functions with established names and notations because of their importance and frequent occurrence in fields like mathematical analysis, physics, engineering, and other applied sciences.  They often arise as solutions to differential equations or integrals of elementary functions [DLMF, \href{https://dlmf.nist.gov/}{DLMF}]. 

\begin{proposition}
The polylogarithm function [MathWorks, \href{https://www.mathworks.com/help/symbolic/sym.polylog.html}{1}];
\begin{multline}\label{eq:prelim_1}
\text{Li}_s(z)\\=(2 \pi )^{s-1} \Gamma (1-s) \left(i^{1-s} \zeta \left(1-s,\frac{1}{2}-\frac{i (\log (z)+i \pi )}{2 \pi }\right)\right. \\ \left.
+i^{s-1} \zeta \left(1-s,\frac{i (\log (z)+i \pi )}{2 \pi }+\frac{1}{2}\right)\right)
\end{multline}
where $Re(s)>1,0,z\leq 1$.\\\\
\end{proposition}
\begin{proposition}
The Hurwitz-Lerch zeta transcendent
\begin{equation}\label{eq:prelim_2}
\Phi (z,-s,0)=\text{Li}_{-s}(z)
\end{equation}
where $Re(s)>0$.
\end{proposition}
\begin{proposition}
The first partial derivative of the Hurwitz-Lerch zeta transcendent
\begin{equation}\label{eq:prelim_3}
\Phi'(-i,0,a)=\log \left(\frac{\Gamma \left(\frac{a}{4}\right)}{2 \Gamma
   \left(\frac{a+2}{4}\right)}\right)-i \log \left(\frac{\Gamma \left(\frac{a+1}{4}\right)}{2 \Gamma
   \left(\frac{a+3}{4}\right)}\right)
\end{equation}
where $a\in\mathbb{C}$,
\end{proposition}
\begin{proposition}
The first partial derivative of the Lerch transcendent
\begin{equation}\label{eq:prelim_4}
\Phi'(i,0,u)=\log \left(\frac{\Gamma \left(\frac{u}{4}\right)}{2 \Gamma
   \left(\frac{u+2}{4}\right)}\right)+i \log \left(\frac{\Gamma \left(\frac{u+1}{4}\right)}{2 \Gamma
   \left(\frac{u+3}{4}\right)}\right)
\end{equation}
where $a\in\mathbb{C}$.
\end{proposition}
\begin{proposition}
The Hurwitz zeta function [MathWorks, \href{https://www.mathworks.com/help/symbolic/sym.hurwitzzeta.html}{1}];
\begin{equation}
(2 \pi )^{-s} \Gamma (s) \left(e^{\frac{i \pi  s}{2}} \text{Li}_s\left(e^{-2 i a \pi
   }\right)+e^{-\frac{1}{2} i \pi  s} \text{Li}_s\left(e^{2 i a \pi }\right)\right)=\zeta (1-s,a)
\end{equation}
where $Re(s)>0, Im(a)>0$ or $Re(s)>1,Im(a)=0$.
\end{proposition}
\begin{proposition}
The Mellin transform of the the generalized reciprocal Pochhammer symbol;
\begin{equation}\label{eq:mellin_poch}
\int_0^{\infty } \frac{t^{-1+v+s v}}{\left(1+b t^v\right){}_{1+n}} \, dt=\frac{\pi  \csc (\pi  s)}{v n! b^{1+s}}\sum
   _{j=0}^n \binom{n}{j} (-1)^{j+1} (j+1)^s
\end{equation}
where $Re(v)>0$.
\end{proposition}
\subsection{The Hurwitz-Lerch zeta function}
We use equation (1.11.3) in \cite{erd} where $\Phi(z,s,v)$ is the Lerch function which is a generalization of the Hurwitz zeta $\zeta(s,v)$ and Polylogarithm functions $Li_{n}(z)$. In number theory and complex analysis, the Lerch function is a mathematical function that appears in many branches of mathematics and physics. It is named after Czech mathematician Mathias Lerch, who published a paper about the function in 1887. Numerous areas of mathematics, including number theory (especially in the investigation of the Riemann zeta function and its generalizations), complex analysis, and theoretical physics, all have uses for it. It can be used to express a variety of complex functions and series and is involved in numerous mathematical identities. The Lerch function has a series representation given by
\begin{equation}\label{knuth:lerch}
\Phi(z,s,v)=\sum_{n=0}^{\infty}(v+n)^{-s}z^{n}
\end{equation}
where $|z|<1, v \neq 0,-1,-2,-3,..,$ and is continued analytically by its integral representation given by
\begin{equation}\label{knuth:lerch1}
\Phi(z,s,v)=\frac{1}{\Gamma(s)}\int_{0}^{\infty}\frac{t^{s-1}e^{-(v-1)t}}{e^{t}-z}dt
\end{equation}
where $Re(v)>0$, and either $|z| \leq 1, z \neq 1, Re(s)>0$, or $z=1, Re(s)>1$.
\subsection{First contour integral representation involving the Pochhammer symbol}
This contour integral representation is derived by taking the Mellin transform of the reciprocal Pochhammer symbol with $c x^v+b$ as the first variable and $n+1$ as the second variable. Then we transform the cosecant function to the secant function by shifting the angle by $\pi/2$. Next we apply the Cauchy integral formula given in \cite{reyn4} and simplify.
\begin{multline}\label{eq:prelim_1}
\frac{1}{2\pi i}\int_{C}\int_{0}^{\infty}\frac{a^w w^{-k-1} x^{m+\frac{3 v}{2}+w-1}}{\left(c x^v+b\right){}_{n+1}}dxdw\\
=-\frac{1}{2\pi i}\int_{C}\sum_{j=0}^{n}\frac{\pi  (-1)^j a^w w^{-k-1}
   \binom{n}{j} c^{-\frac{m+w}{v}-\frac{3}{2}} \sec \left(\frac{\pi  (m+w)}{v}\right) (b+j)^{\frac{2 m+v+2 w}{2 v}}}{vn!}dw
\end{multline}
where $0 < Re(a) < n+1; n = 1,2,.., |\arg(b)|<\pi$.
\subsection{Second contour integral representation involving a generalized polynomial function}
This contour integral representation is derived by apply the $n$-th partial derivative of $b$ to equation (3.194.3) in \cite{grad} where $\nu=1$. We simplify the Pochhammer symbol using equation [Wolfram MathWorld,\href{https://mathworld.wolfram.com/PochhammerSymbol.html}{8}] in terms of Stirling's number of the first kind. We then shift the angle of the cosecant function by $\pi/2$ to get the secant function and apply the Cauchy integral formula in \cite{reyn4} and simplify in terms of the Stirling numbers of the first kind $S_n^{(j)}$ see [Wolfram MathWorld, \href{https://mathworld.wolfram.com/StirlingNumberoftheFirstKind.html}{(1)}].
\begin{multline}\label{eq:prelim_2}
\frac{1}{2\pi i}\int_{C}\int_{0}^{\infty}a^w w^{-k-1} \left(b x^v+1\right)^{-n-1} x^{m+\left(n+\frac{1}{2}\right) v+w}dxdw\\
=-\frac{1}{2\pi i}\int_{C}\sum_{j=0}^{n}\sum_{l=0}^{j}\frac{\pi  (-1)^{-j} a^w
   \left(-\frac{1}{v}\right)^l \binom{j}{l} S_n^{(j)} w^{-k+l-1} \sec \left(\frac{\pi  (m+w+1)}{v}\right) b^{-\frac{2 m+2n v+v+2 w+2}{2 v}} }{v n!}\\ \times \left(-\frac{m+1}{v}-n+\frac{1}{2}\right)^{j-l}dw
\end{multline}
where $0 < Re(a) < n+1; n = 1,2,.., |\arg(b)|<\pi$.
\subsection{Third contour integral representation}
This contour integral is derived in equation (3.2) in \cite{reyn_aims}.
\begin{multline}\label{eq:prelim_3}
\frac{2^{k+1} (i b)^k e^{i b m} \Phi \left(-e^{2 i b m},-k,\frac{b-i \log (a)}{2 b}\right)}{k!}=\frac{1}{2\pi i}\int_{C}a^w w^{-k-1} \sec
   (b (m+w))dw
\end{multline}
where $Re(\frac{w+m}{b})<0$.
\section{Derivation of contour integral representations}
In this section we derive the contour integral representations used to derive the main theorems in this work.
\subsection{The contour integral involving the reciprocal Pochhammer symbol}
\subsubsection{Left-hand side contour integral}
We use the method in~\cite{reyn4}. Using a generalization of Cauchy's integral formula we form the integral by replacing $y$ by $ax$ and multiply both sides by $\frac{x^{m+\frac{3 v}{2}-1}}{\left(c x^v+b\right){}_{n+1}}$ then the integral with respect to $x\in [0,\infty)$ to obtain
\begin{multline}\label{eq:poch_1}
\int_{0}^{\infty}\frac{x^{m+\frac{3 v}{2}-1} \log ^k(a x)}{k! \left(c x^v+b\right){}_{n+1}}dx\\
=\frac{1}{2\pi i}\int_{0}^{\infty}\int_{C}\frac{w^{-k-1} (a x)^w x^{m+\frac{3v}{2}-1}}{\left(c x^v+b\right){}_{n+1}}dwdx\\
=\frac{1}{2\pi i}\int_{C}\int_{0}^{\infty}\frac{w^{-k-1} (a x)^w x^{m+\frac{3v}{2}-1}}{\left(c x^v+b\right){}_{n+1}}dxdw
\end{multline}
where $0< Re(m) <1$. We are able to switch the order of integration over $x$ and $w$ using Fubini's theorem for multiple integrals see page 178 in \cite{gelca}, since the integrand is of bounded measure over the space $\mathbb{C} \times [0,\infty) $.
\subsubsection{Right-hand side contour integral}
Here we use equation \ref{eq:prelim_1} and replace $b\to \frac{\pi }{v},a\to a c^{-1/v} (b+j)^{\frac{1}{v}}$ then multiply both sides by $-\frac{\pi  (-1)^j \binom{n}{j} c^{-\frac{m}{v}-\frac{3}{2}} (b+j)^{\frac{2 m+v}{2 v}}}{v n!}$ take the finite sum over $j\in [0,n]$ to obtain
\begin{multline}\label{eq:poch_2}
\sum_{j=0}^{n}\frac{(-1)^j (2 \pi )^{k+1} v \left(\frac{i}{v}\right)^{k+2} e^{\frac{i \pi  m}{v}} \binom{n}{j}
   c^{-\frac{m}{v}-\frac{3}{2}} (b+j)^{\frac{m}{v}+\frac{1}{2}} }{k! n!}\\ \times
   \Phi \left(-e^{\frac{2 i m \pi }{v}},-k,\frac{\pi -i v
   \log \left(a c^{-1/v} (b+j)^{\frac{1}{v}}\right)}{2 \pi }\right)\\
=-\frac{1}{2\pi i}\sum_{j=0}^{n}\int_{C}\frac{\pi  (-1)^j a^w w^{-k-1} \binom{n}{j}
   c^{-\frac{m+w}{v}-\frac{3}{2}} \sec \left(\frac{\pi  (m+w)}{v}\right) (b+j)^{\frac{2 m+v+2 w}{2 v}}}{v n!}dw\\
   =-\frac{1}{2\pi i}\int_{C}\sum_{j=0}^{n}\frac{\pi  (-1)^j a^w w^{-k-1} \binom{n}{j}
   c^{-\frac{m+w}{v}-\frac{3}{2}} \sec \left(\frac{\pi  (m+w)}{v}\right) (b+j)^{\frac{2 m+v+2 w}{2 v}}}{v n!}dw
\end{multline}
where $0 < Re(a) < n+1; n = 1,2,.., |\arg(b)|<\pi, Im\left(\frac{\pi  (m+w)}{v}\right)>0$ in order for the sum to converge. We are able to switch the order of integration and summation using Tonelli's theorem for integrals and series see page 178 in \cite{gelca}, since the integrand and summand are of bounded measure over the space $\mathbb{C} \times [0,n] $.
\subsection{The contour integral involving the reciprocal polynomial function}
\subsubsection{Left-hand side contour integral}
We use the method in~\cite{reyn4}. Using a generalization of Cauchy's integral formula we form the integral by replacing $y$ by $ax$ and multiply both sides by $\left(b x^v+1\right)^{-n-1} x^{m+\left(n+\frac{1}{2}\right) v}$ then the integral with respect to $x\in [0,\infty)$ to obtain
\begin{multline}\label{eq:poly_1}
\int_{0}^{\infty}\frac{\log ^k(a x) \left(b x^v+1\right)^{-n-1} x^{m+\left(n+\frac{1}{2}\right) v}}{k!}dx\\
=\frac{1}{2\pi i}\int_{0}^{\infty}\int_{C}w^{-k-1} (a x)^w \left(b
   x^v+1\right)^{-n-1} x^{m+\left(n+\frac{1}{2}\right) v}dwdx\\
=\frac{1}{2\pi i}\int_{C}\int_{0}^{\infty}w^{-k-1} (a x)^w \left(b
   x^v+1\right)^{-n-1} x^{m+\left(n+\frac{1}{2}\right) v}dxdw
\end{multline}
where $0< Re(m) <1$. We are able to switch the order of integration over $x$ and $w$ using Fubini's theorem for multiple integrals see page 178 in \cite{gelca}, since the integrand is of bounded measure over the space $\mathbb{C} \times [0,\infty) $.
\subsubsection{Right-hand side contour integral}
Here we use equation \ref{eq:prelim_1} and replace $b\to \frac{\pi }{v},m\to m+1,a\to a b^{-1/v},k\to k-l$ then multiply both sides by 
\begin{equation}
\frac{\pi  (-1)^{-j} \left(-\frac{1}{v}\right)^l \binom{j}{l} S_n^{(j)} b^{-\frac{2 m+2 n v+v+2}{2 v}}
   \left(-\frac{m+1}{v}-n+\frac{1}{2}\right)^{j-l}}{v n!}
   \end{equation}
   take the finite sums over $j\in [0,n]$ and $l\in [0,j]$ to obtain
\begin{multline}\label{eq:poly_2}
\sum _{j=0}^n \sum _{l=0}^j \frac{(-1)^{-j} b^{-\frac{2+2 m+v+2 n v}{2 v}} e^{\frac{i (1+m) \pi }{v}} (2 \pi
   )^{1+k-l} \left(\frac{1}{2}-n-\frac{1+m}{v}\right)^{j-l} \left(-\frac{1}{v}\right)^l \left(\frac{i}{v}\right)^{k-l}
   \binom{j}{l}}{v (k-l)! n!}\\ \times
 \Phi \left(-e^{\frac{2 i (1+m) \pi }{v}},-k+l,\frac{\pi -i v \log \left(a b^{-1/v}\right)}{2 \pi }\right)
   S_n^{(j)}\\
=\frac{1}{2\pi i}\sum _{j=0}^n \sum _{l=0}^j \int_{C}\frac{(-1)^{-j} a^w b^{-\frac{2+2 m+v+2 n v+2 w}{2 v}} \pi \left(\frac{1}{2}-n-\frac{1+m}{v}\right)^{j-l} \left(-\frac{1}{v}\right)^l w^{-1-k+l} \binom{j}{l}}{v n!}\\  \times\sec \left(\frac{\pi (1+m+w)}{v}\right) S_n^{(j)}dw\\
=\frac{1}{2\pi i}\int_{C}\sum _{j=0}^n \sum _{l=0}^j \frac{(-1)^{-j} a^w b^{-\frac{2+2 m+v+2 n v+2 w}{2 v}} \pi \left(\frac{1}{2}-n-\frac{1+m}{v}\right)^{j-l} \left(-\frac{1}{v}\right)^l w^{-1-k+l} \binom{j}{l}}{v n!}\\  \times\sec \left(\frac{\pi (1+m+w)}{v}\right) S_n^{(j)}dw
\end{multline}
where $0 < Re(a) < n+1; n = 1,2,.., |\arg(b)|<\pi, Im\left(\frac{\pi  (1+m+w)}{v}\right)>0$ in order for the sum to converge. We are able to switch the order of integration and summation using Tonelli's theorem for integrals and series see page 178 in \cite{gelca}, since the integrand and summand are of bounded measure over the space $\mathbb{C} \times [0,n] \times [0,j] $.
\section{Generalized infinite logarithm integrals in terms of finite series}
In this section we derive definite infinite integrals involving the product of generalized logarithm and reciprocal Pochhammer symbol functions for the first formula. The second infinite integral involves the product of generalized logarithm and reciprocal polynomial function. 
\begin{theorem}
A general integral involving logarithmic powers Pochhammer symbol and rational functions. For all $k,b,c,v\in\mathbb{C}$ then,
\begin{multline}\label{eq:thm1}
\int_0^{\infty } \frac{x^{-1+m} \log ^k(a x)}{\left(b+c x^v\right){}_{1+n}} \, dx\\
=\sum _{j=0}^n \frac{(-1)^j
   c^{-\frac{3}{2}-\frac{m-\frac{3 v}{2}}{v}} e^{\frac{i \pi  \left(m-\frac{3 v}{2}\right)}{v}}
   (b+j)^{\frac{1}{2}+\frac{m-\frac{3 v}{2}}{v}} (2 \pi )^{1+k} \left(\frac{i}{v}\right)^{2+k} v \binom{n}{j} }{n!}\\ \times
\Phi
   \left(-e^{\frac{2 i \pi  \left(m-\frac{3 v}{2}\right)}{v}},-k,\frac{\pi -i v \log \left(a c^{-1/v}
   (b+j)^{1/v}\right)}{2 \pi }\right)
\end{multline}
where $Re(m)>0,Re(v)>0$ and there exists a singularity at $x=i/a$. If $Re(b)<0, Re(a)>0$.
\end{theorem}
\begin{proof}
The right-hand sides of relations (\ref{eq:poch_1}) and (\ref{eq:poch_2}) are identical, relative to equation (\ref{eq:prelim_1}); hence, the left-hand sides of the same are identical too. Simplifying with the Pochhammer function in [Wolfram MathWorld, \href{https://mathworld.wolfram.com/PochhammerSymbol.html}{1}] yields the desired conclusion.
\end{proof}
\begin{theorem}
A Mellin-Transform integral with logarithmic powers. For all $k,b,c,v\in\mathbb{C}$ then,
\begin{multline}\label{eq:thm2}
\int_0^{\infty } \frac{x^{-1+m} \log ^k(a x)}{\left(1+b x^v\right)^{n+1}} \, dx\\
=\sum _{j=0}^n \sum _{l=0}^j
   \frac{(1+k-l)_l (-1)^{-j} b^{-\frac{m}{v}} e^{\frac{i \pi  \left(m-\left(\frac{1}{2}+n\right) v\right)}{v}} (2 \pi
   )^{1+k-l} \left(1-\frac{m}{v}\right)^{j-l} \left(-\frac{1}{v}\right)^l \left(\frac{i}{v}\right)^{k-l} \binom{j}{l}
   }{v n!}\\ \times
\Phi \left(e^{\frac{2 i \pi  (m-n v)}{v}},-k+l,\frac{\pi -i v \log \left(a b^{-1/v}\right)}{2 \pi }\right)
   S_n^{(j)}
\end{multline}
when $Re(b)<0$, then there exists a singularity at $x=i/a$.
\end{theorem}
\begin{proof}
The right-hand sides of relations (\ref{eq:poly_1}) and (\ref{eq:poly_2}) are identical, relative to equation (\ref{eq:prelim_2}); hence, the left-hand sides of the same are identical too. Simplifying with the Pochhammer function in [Wolfram MathWorld, \href{https://mathworld.wolfram.com/PochhammerSymbol.html}{1}] yields the desired conclusion.
\end{proof}
\section{Derivations and evaluations of definite integrals and transformations}
\begin{example}
In this example we use equation (\ref{eq:thm1}) and take the first partial derivative with respect to $k$ and set $k\to 0$. Next form a second equation by replacing $a\to 1/a$ and add the two equations and simplify;
\begin{multline}\label{eq1}
\int_0^{\infty } \frac{x^{-1+m} \log \left(\log ^2(a)-\log ^2(x)\right)}{\left(b+c x^v\right){}_{1+n}} \,
   dx\\
=\sum _{j=0}^n \frac{2 c^{-\frac{m}{v}} e^{i \pi  \left(j+\frac{m}{v}\right)} (b+j)^{-1+\frac{m}{v}} \pi 
   \binom{n}{j} }{v n!}\\ \times
\left(\left(-i+\cot \left(\frac{m \pi }{v}\right)\right) \log \left(\frac{2 i \pi }{v}\right)+i
   \Phi'\left(e^{\frac{2 i m \pi }{v}},0,\frac{\pi -i v \log \left(\frac{c^{-1/v}
   (b+j)^{1/v}}{a}\right)}{2 \pi }\right)\right. \\ \left.
+i \Phi'\left(e^{\frac{2 i m \pi
   }{v}},0,\frac{\pi -i v \log \left(a c^{-1/v} (b+j)^{1/v}\right)}{2 \pi }\right)\right)\\
-\sum _{j=0}^n
   \frac{i c^{-\frac{m}{v}} e^{i j \pi } (b+j)^{-1+\frac{m}{v}} \pi ^2 \binom{n}{j} \csc \left(\frac{m \pi
   }{v}\right)}{v n!}
\end{multline}
where $Re(m)>0,Re(v)>0$ and there exists a singularity at $x=i/a$. If $Re(b)<0, Re(a)>0$.
\end{example}
\begin{example}
 In this example we use equation (\ref{eq:thm1}) and set $k \to -1,a \to e^{i \pi  a}$. Then we form a second equation by replacing $a \to -a$ and take their difference. Next we replace $m \to s$ and take the difference and simplify.  A Mellin-Type Integral Involving a logarithmic denominator, power-law numerator, and generalized rational decay. 
\begin{multline}
\int_0^{\infty } \frac{-x^m+x^s}{\left(a^2 \pi ^2+\log ^2(x)\right) \left(b+c x^v\right){}_{1+n}} \, dx
=-\sum _{j=0}^n \frac{i (-1)^j c^{-1/v} (b+j)^{-1+\frac{1}{v}} \binom{n}{j}}{2 a \pi  n!}\\
\left(c^{-\frac{m}{v}} e^{\frac{i (1+m) \pi }{v}} (b+j)^{m/v} \left(\Phi \left(e^{\frac{2 i (1+m) \pi }{v}},1,\frac{\pi -a \pi  v-i v \log \left(c^{-1/v} (b+j)^{1/v}\right)}{2 \pi}\right)\right.\right. \\ \left.\left.
-\Phi \left(e^{\frac{2 i (1+m) \pi }{v}},1,\frac{\pi +a \pi  v-i v \log \left(c^{-1/v} (b+j)^{1/v}\right)}{2 \pi }\right)\right)\right. \\ \left.
+c^{-\frac{s}{v}} e^{\frac{i \pi  (1+s)}{v}}
   (b+j)^{s/v} \left(-\Phi \left(e^{\frac{2 i \pi  (1+s)}{v}},1,\frac{\pi -a \pi  v-i v \log \left(c^{-1/v} (b+j)^{1/v}\right)}{2 \pi }\right)\right.\right. \\ \left.\left.
   +\Phi \left(e^{\frac{2 i \pi (1+s)}{v}},1,\frac{\pi +a \pi  v-i v \log \left(c^{-1/v} (b+j)^{1/v}\right)}{2 \pi }\right)\right)\right)
\end{multline}
where $|Re(a)| < 1/2, 0 < Re(v)\leq \pi/2$ and there exists a singularity at $x=i e^{\pi  a}$.
\end{example}
\begin{example}
In this example we use equation (\ref{eq:thm1}) and set $k\to -1,a\to e^{i \pi  a},b\to c\to v\to 1$. Then we form a second equation by replacing $a\to -a$ and take their difference.  Then we take the limit as $m\to 1$ and simplify using equation [Wolfram MathWorld, \href{https://mathworld.wolfram.com/HurwitzZetaFunction.html}{(19)}]. A Mellin-type Integral with logarithmic denominator and Pochhammer decay. 
\begin{multline}
\int_0^{\infty } \frac{1}{\left(a^2 \pi ^2+\log ^2(x)\right) (1+x)_{1+n}} \, dx\\
=\sum _{j=0}^n \frac{i (-1)^j \binom{n}{j} \left(\psi ^{(0)}\left(\frac{\pi -a \pi -i \log (1+j)}{2
   \pi }\right)-\psi ^{(0)}\left(\frac{\pi +a \pi -i \log (1+j)}{2 \pi }\right)\right)}{2 a \pi  n!}
\end{multline}
where $|Re(a)| \leq 1/4$.
\end{example}
\begin{example}
In this example we use equation (\ref{eq1}) and set $b\to -\frac{1}{2},c\to \frac{1}{2},m\to 1,v\to 2$ and simplify;
\begin{multline}\label{eq2}
\int_0^{\infty } \frac{\log \left(\log ^2(a)-\log ^2(x)\right)}{\left(\frac{1}{2}
   \left(-1+x^2\right)\right){}_{1+n}} \, dx\\
=\frac{2 \pi  }{n!}\sum _{j=0}^n \frac{(-1)^j \binom{n}{j} }{\sqrt{-1+2 j}}\log
   \left(\frac{2 i \pi  \Gamma \left(\frac{3}{4}-\frac{i \log \left(\frac{\sqrt{-1+2 j}}{a}\right)}{2 \pi }\right)
   \Gamma \left(\frac{3}{4}-\frac{i \log \left(a \sqrt{-1+2 j}\right)}{2 \pi }\right)}{\Gamma \left(\frac{\pi -2 i
   \log \left(\frac{\sqrt{-1+2 j}}{a}\right)}{4 \pi }\right) \Gamma \left(\frac{\pi -2 i \log \left(a \sqrt{-1+2
   j}\right)}{4 \pi }\right)}\right)\\
-\frac{\pi ^2 i }{n!}\sum _{j=0}^n \frac{(-1)^j
   \binom{n}{j}}{\sqrt{-1+2 j}}
\end{multline}
where $Re(a)>0$ and there exists a singularity at $x=i/a$.
\end{example}
\begin{example}
In this example we use equation (\ref{eq2}) and set $a\to e^{\pi},n\to 1$ and simplify;
\begin{multline}
\int_0^{\infty } \frac{\log \left(\pi ^2-\log ^2(x)\right)}{1-x^4} \, dx=\frac{1}{2} \pi  \log \left(\frac{2
   \pi ^{1+i} \coth ^i\left(\frac{\pi }{2}\right) \Gamma \left(\frac{3}{4}-\frac{i}{2}\right) \Gamma
   \left(\frac{3}{4}+\frac{i}{2}\right)}{\Gamma \left(\frac{1}{4}-\frac{i}{2}\right) \Gamma
   \left(\frac{1}{4}+\frac{i}{2}\right)}\right)
\end{multline}
where there exists a singularity at $x=ie^{-\pi}$. 
\end{example}
\begin{example}
In this example we use equation (\ref{eq2}) and set $a\to e^{\pi},n\to 0$ and simplify;
\begin{equation}
\int_0^{\infty } \frac{\log \left(\pi ^2-\log ^2(x)\right)}{1-x^2} \, dx=i \pi  \log \left(\pi  \coth
   \left(\frac{\pi }{2}\right)\right)
\end{equation}
where there exists a singularity at $x=ie^{\pi}$.
\end{example}
\begin{example}
In this example we use equation (\ref{eq1}) and take the first partial derivative with respect to $a$ and simplify;
\begin{multline}
\int_0^{\infty } \frac{1}{\left(\log ^2(a)-\log ^2(x)\right) \left(\frac{1}{2}
   \left(-1+x^2\right)\right){}_{1+n}} \, dx\\
=-\sum _{j=0}^n \frac{i (-1)^j \binom{n}{j} }{2 \sqrt{-1+2 j} n! \log (a)}\left(\psi
   ^{(0)}\left(\frac{\pi -2 i \log \left(\frac{\sqrt{-1+2 j}}{a}\right)}{4 \pi }\right)\right. \\ \left.
-\psi
   ^{(0)}\left(\frac{3}{4}-\frac{i \log \left(\frac{\sqrt{-1+2 j}}{a}\right)}{2 \pi }\right)-\psi
   ^{(0)}\left(\frac{\pi -2 i \log \left(a \sqrt{-1+2 j}\right)}{4 \pi }\right)\right. \\ \left.
+\psi ^{(0)}\left(\frac{3}{4}-\frac{i
   \log \left(a \sqrt{-1+2 j}\right)}{2 \pi }\right)\right)
\end{multline}
where $Re(a)>0$ and there exists a singularity at $x=i/a$.
\end{example}
\begin{example}
In this example we look at a generalized Malmsten form. Here we use equation (\ref{eq:thm1}) and take the first partial derivative with respect to $k$ and set $k\to 0$ and simplify;
\begin{multline}
\int_0^{\infty } \frac{x^{-1+m} \log (\log (x))}{\left(b+c x^v\right){}_{1+n}} \, dx\\
=\sum _{j=0}^n \frac{2 i (-1)^j
   c^{-\frac{m}{v}} e^{\frac{i m \pi }{v}} (b+j)^{-1+\frac{m}{v}} \pi  \binom{n}{j} }{\left(-1+e^{\frac{2 i m \pi
   }{v}}\right) v n!}\\ \times
\left(\log \left(\frac{2 i \pi
   }{v}\right)+\left(-1+e^{\frac{2 i m \pi }{v}}\right) \Phi'\left(e^{\frac{2 i m \pi
   }{v}},0,\frac{\pi -i v \log \left(c^{-1/v} (b+j)^{1/v}\right)}{2 \pi }\right)\right)
\end{multline}
where $n = 1,2,.., |\arg(b)|<\pi$ where a singularity exists at $x=i/2$.
\end{example}
%
%
\begin{example}
In this example we use equation (\ref{eq2}) and set $b\to -1,n\to 0$. Next take the first partial derivative with respect to $k$ and set $k\to 0$ and simplify;
\begin{multline}
\int_0^{\infty } \frac{x^{-1+m} \log (\log (a x))}{\left(-1+x^v\right)^2} \, dx
=\frac{(-1)^{-\frac{m}{v}}
   e^{\frac{i m \pi }{v}} }{\left(-1+e^{\frac{2 i m \pi
   }{v}}\right) v^2}\\ \times
\left(\left(-1+e^{\frac{2 i m \pi }{v}}\right) v \Phi \left(e^{\frac{2 i m \pi }{v}},1,\frac{\pi
   -i v \log \left((-1)^{-1/v} a\right)}{2 \pi }\right)-2 i \pi  (m-v) \left(\log \left(\frac{2 i \pi
   }{v}\right)\right.\right. \\ \left.\left.
+\left(-1+e^{\frac{2 i m \pi }{v}}\right) \Phi'\left(e^{\frac{2 i m \pi
   }{v}},0,\frac{\pi -i v \log \left((-1)^{-1/v} a\right)}{2 \pi }\right)\right)\right)
\end{multline}
where $0< Re(m)<1,Re(v)>1$ where there exists a singularity at $x=i/a$.
\end{example}
\begin{example}
Extended form of equation (3.6) in \cite{kolbig}. In this example we use equation (\ref{eq:thm2}) and set $b\to -1,v\to 1$ and simplify;
\begin{multline}
\int_{0}^{\infty } \frac{x^{m-1} \log ^k(a x)}{(1-x)^{n+1}} \, dx
=\sum _{j=0}^n \sum _{l=0}^j
   \frac{(-1)^{-j+l-m} i^{k-l} e^{i \left(-\frac{1}{2}+m-n\right) \pi } (1-m)^{j-l} (2 \pi )^{1+k-l} \binom{j}{l}}{n!}\\ \times
 \Phi
   \left(e^{2 i (m-n) \pi },-k+l,\frac{\pi -i \log (-a)}{2 \pi }\right) (1+k-l)_l S_n^{(j)}
\end{multline}
where $0< Re(m)<1$ where there exists a singularity at $x=i/a$.
\end{example}
\begin{example}
Hurwitz-Lerch zeta transformation. In this example we use equation (\ref{eq:thm2}) and set $n=0$. Then form a second equation by replacing $b\to bi$, then form a third equation by replacing $b\to -b$. Then take the difference of the third and fourth equations, multiply both sides by $1/(2bi)$ and replace $m\to m-v$. Then using the first equation replace $b\to b^2,v\to 2v$ and equate the left-hand sides and write down the transformation in terms of the Hurwitz-Lerch zeta function;
\begin{multline}\label{eq:lerch_transform}
\Phi \left(e^{\frac{i m \pi }{v}},-k,\frac{1}{2}-\frac{i v \log \left(a b^{-1/v}\right)}{\pi }\right)\\
=2^k
   e^{\frac{i \pi  m}{2 v}} b^{-\frac{m}{v}} \left((-i b)^{m/v} \Phi \left(e^{\frac{2 i m \pi }{v}},-k,\frac{\pi -i v
   \log \left(a (i b)^{-1/v}\right)}{2 \pi }\right)\right. \\ \left.
+(i b)^{m/v} \Phi \left(e^{\frac{2 i m \pi }{v}},-k,\frac{\pi -i v
   \log \left(a (-i b)^{-1/v}\right)}{2 \pi }\right)\right)
\end{multline}
where $Re(v)>1$.
\end{example}
\begin{example}
In this example we use equation (\ref{eq:thm2}) and $k\to -1,a\to e^{a i}$, then we form a second equation by replacing $a\to -a$ and taking there difference and simplifying;
\begin{multline}\label{eq:log_a}
\int_0^{\infty } \frac{\left(1+b x^v\right)^{-1-n}}{a^2 \pi ^2-\log ^2(x)} \, dx\\
=-\sum _{j=0}^n \sum _{l=0}^j
   \frac{(-1)^{-j+l} b^{-1/v} e^{-\frac{i \pi  (-1+n v)}{v}} (2 \pi )^{-1-l} \left(-\frac{1}{v}\right)^l
   \left(\frac{i}{v}\right)^{-l} \left(\frac{-1+v}{v}\right)^{j-l} \binom{j}{l} l! }{a n!}\\ \times
\left(-\Phi \left(e^{-\frac{2 i \pi
    (-1+n v)}{v}},1+l,\frac{\pi +i a \pi  v-i v \log \left(b^{-1/v}\right)}{2 \pi }\right)\right. \\ \left.
+\Phi \left(e^{-\frac{2 i
   \pi  (-1+n v)}{v}},1+l,-\frac{i \left(\pi  (i+a v)+v \log \left(b^{-1/v}\right)\right)}{2 \pi }\right)\right)
   S_n^{(j)}
\end{multline}
where $Re(a)<1/Re(v)$.
\end{example}
\begin{example}
In this example we use equation (\ref{eq:log_a}) and take the first partial derivative with respect to $a$ and set $b\to e^{b\pi i}$ and simplify;
\begin{multline}
\int_0^{\infty } \frac{\left(1+e^{i b \pi } x^v\right)^{-1-n}}{\left(-a^2 \pi ^2+\log ^2(x)\right)^2} \,
   dx\\
=\sum _{j=0}^n \sum _{l=0}^j \frac{(-1)^{-j+l} e^{-\frac{i \pi  (-1+b+n v)}{v}} (2 \pi )^{-3-l}
   \left(-\frac{1}{v}\right)^l \left(\frac{i}{v}\right)^{-l} \left(\frac{-1+v}{v}\right)^{j-l} \binom{j}{l} l!
   }{a^3 n!}\\ \times
\left(-2 \Phi \left(e^{-\frac{2 i \pi  (-1+n v)}{v}},1+l,\frac{1}{2} (1-b-i a v)\right)+2 \Phi \left(e^{-\frac{2 i
   \pi  (-1+n v)}{v}},1+l,\frac{1}{2} (1-b+i a v)\right)\right. \\ \left.
+i a (1+l) v \left(\Phi \left(e^{-\frac{2 i \pi  (-1+n
   v)}{v}},2+l,\frac{1}{2} (1-b-i a v)\right)\right.\right. \\ \left.\left.
+\Phi \left(e^{-\frac{2 i \pi  (-1+n v)}{v}},2+l,\frac{1}{2} (1-b+i a
   v)\right)\right)\right) S_n^{(j)}
\end{multline}
where $Re(a)<1/Re(v)$.
\end{example}
\begin{example}
Extended Diekama [Diekama, \href{https://arxiv.org/pdf/2411.08484}{pp 28}] form. In this example we use equation (\ref{eq:thm2}) and set $k\to -1,a\to e^{ai},n\to 1,v\to 2,b\to 1,c\to 1$. Then we replace $a$ to $2\pi]-a$, then take this equation and multiply by -1 and change the coefficient of log(x) to positive 1. Then we form a second equation by replacing $a\to a-2\pi$ and take there difference. Then we take the first partial derivative with respect to $a$ and simplify;
\begin{equation}\label{eq:diekama}
\int_0^{\infty } \frac{1}{(1+x)^2 \left(a^2+\log ^2(x)\right)^2} \, dx=\frac{2 \pi  \psi ^{(1)}\left(\frac{a+\pi }{2 \pi
   }\right)-a \psi ^{(2)}\left(\frac{a+\pi }{2 \pi }\right)}{8 a^3 \pi ^2}
\end{equation}
where $Re(a)>0$.
\end{example}
\begin{example}
In this example we use equation (\ref{eq:diekama}) and set $a=\pi$ and simplify;
\begin{equation}
\int_0^{\infty } \frac{1}{(1+x)^2 \left(\pi ^2+\log ^2(x)\right)^2} \, dx=\frac{\pi ^2+6 \zeta (3)}{24 \pi
   ^4}
\end{equation}
\end{example}
\section{Exercises}
Derive the below integral forms using the results above.
\begin{example}
Doubly-squared logarithmic integral with complex linear denominator.
\begin{multline}
\int_0^{\infty } \frac{1}{\left(1+e^{i b \pi } x\right)^2 \left(a^2 \pi ^2+\log ^2(x)\right)^2} \,
   dx\\
=\frac{e^{-i b \pi } }{16 a^3 \pi ^4}\left(2 \psi ^{(1)}\left(\frac{1}{2} (1+a-b)\right)+2 \psi ^{(1)}\left(\frac{1}{2} (1+a+b)\right)\right. \\ \left.
-a \left(\psi ^{(2)}\left(\frac{1}{2} (1+a-b)\right)+\psi ^{(2)}\left(\frac{1}{2}
   (1+a+b)\right)\right)\right)
\end{multline}
where $|Re(b)|<1$.
\end{example}
\begin{example}
Logarithmically weighted integral with complex pole and quadratic log denominator.
\begin{multline}
\int_0^{\infty } \frac{\log (x)}{\left(1+e^{i b \pi } x\right)^2 \left(a^2 \pi ^2+\log ^2(x)\right)^2} \,
   dx\\
=-\frac{i e^{-i b \pi } \left(\zeta \left(3,\frac{1}{2} (1+a-b)\right)-\zeta \left(3,\frac{1}{2}
   (1+a+b)\right)\right)}{8 a \pi ^3}
\end{multline}
where $|Re(b)|<1$.
\end{example}
%
%
%
\subsection{Generalized Malmsten logarithm integral forms}
In this section we derive generalized forms of Malmsten logarithm integrals. These forms were also studied by authors previously cited.
\begin{example}
In this example we use equation (\ref{eq:thm1}) and take the first partial derivative with respect to $k$ and set $k=0,a=1$ and simplify;
\begin{multline}\label{eq:malm_poch}
\int_0^{\infty } \frac{x^{-1+m} \log (\log (x))}{\left(b+c x^v\right){}_{1+n}} \, dx\\
=\sum _{j=0}^n \frac{2 i (-1)^j
   c^{-\frac{m}{v}} e^{\frac{i m \pi }{v}} (b+j)^{-1+\frac{m}{v}} \pi  \binom{n}{j} }{\left(-1+e^{\frac{2 i m \pi
   }{v}}\right) v n!}\\ \times
\left(\log \left(\frac{2 i \pi
   }{v}\right)+\left(-1+e^{\frac{2 i m \pi }{v}}\right) \Phi'\left(e^{\frac{2 i m \pi
   }{v}},0,\frac{\pi -i v \log \left(c^{-1/v} (b+j)^{1/v}\right)}{2 \pi }\right)\right)
\end{multline}
where $n = 1,2,.., |\arg(b)|<\pi$.
\end{example}
\begin{example}
In this example we use equation (\ref{eq:malm_poch}) and form a second equation by replacing $m\to s$ and taking there difference and simplify;
\begin{multline}\label{eq:malm_poch1}
\int_0^{\infty } \frac{\left(x^s-x^m\right) \log (\log (x))}{\left(b+c x^v\right){}_{1+n}} \, dx
=-\sum _{j=0}^n
   \frac{i (-1)^j \pi  \binom{n}{j} }{(b+j) v n!}\\ \times
\left(c^{-\frac{1+m}{v}} (b+j)^{\frac{1+m}{v}} \left(-i \csc \left(\frac{(1+m) \pi
   }{v}\right) \log \left(\frac{2 i \pi }{v}\right)\right.\right. \\ \left.\left.
+2 e^{\frac{i (1+m) \pi }{v}}
   \Phi'\left(e^{\frac{2 i (1+m) \pi }{v}},0,\frac{\pi -i v \log \left(c^{-1/v}
   (b+j)^{1/v}\right)}{2 \pi }\right)\right)\right. \\ \left.
+i c^{-\frac{1+s}{v}} (b+j)^{\frac{1+s}{v}} \left(\csc \left(\frac{\pi 
   (1+s)}{v}\right) \log \left(\frac{2 i \pi }{v}\right)\right.\right. \\ \left.\left.
+2 i e^{\frac{i \pi  (1+s)}{v}}
   \Phi'\left(e^{\frac{2 i \pi  (1+s)}{v}},0,\frac{\pi -i v \log \left(c^{-1/v}
   (b+j)^{1/v}\right)}{2 \pi }\right)\right)\right)
\end{multline}
where $Re(m)>0,Re(s)>0, n=1,2,3,..,$.
\end{example}
\begin{example}
In this example we use equation (\ref{eq:malm_poch1}) and set $b\to 1,c\to -1,v\to 1,m\to \frac{1}{2},s\to -\frac{1}{2}$ and simplify;
\begin{multline}
\int_0^{\infty } \frac{\left(\frac{1}{\sqrt{x}}-\sqrt{x}\right) \log (\log (x))}{(1-x)_{1+n}} \, dx
=\sum_{j=0}^n \frac{i (-1)^j \pi  \binom{n}{j} }{(1+j) n!}\\ \times
\left(\sqrt{1+j} \left(\log (2 i \pi )-2
   \Phi'\left(-1,0,\frac{\pi -i \log (-1-j)}{2 \pi }\right)\right)\right. \\ \left.
+i (1+j)^{3/2} \left(i
   \log (2 i \pi )-2 i \Phi'\left(-1,0,\frac{\pi -i \log (-1-j)}{2 \pi
   }\right)\right)\right)
\end{multline}
where $n=1,2,3,..,$ and there exists a singularity at $x=i$.
\end{example}
%
%
%
\begin{example}
In this example we derive a generalized Malmsten form in terms of a finite double series involving the first partial derivative of the Hurwitz-Lerch zeta function. There exists a singularity at $x=1/a$. In this example we look at deriving a generalized integral form originally studied by Malmsten $\log\log(x)$ \cite{malmsten} infinite integrals. Here we use equation (\ref{eq:thm2}) and take the first partial derivative with respect to $k$ and set $k\to 0$. We use equation [Wolfram MathWorld, \href{http://functions.wolfram.com/06.10.20.0007.01}{01}] to simplify the Pochammer symbol (gamma representation) when taking the derivative. This form is used as the Pochammer symbol for this evaluation is undefined when $k=0$. The finite series representation for the Pochammer symbol removes this discontinuity and thus the finite double series is able to be evaluated. We then simply the remaining terms;
\begin{multline}\label{eq:malm1}
\int_0^{\infty } \frac{x^{m-1} \log (\log (a x))}{\left(1+b x^v\right)^{n+1}} \, dx\\
=\sum _{j=0}^n \sum
   _{l=0}^j \frac{(-1)^{-j} b^{-\frac{m}{v}} e^{\frac{i \pi  \left(m-\left(\frac{1}{2}+n\right) v\right)}{v}} (2 \pi
   )^{1-l} \left(1-\frac{m}{v}\right)^{j-l} \left(-\frac{1}{v}\right)^l \left(\frac{i}{v}\right)^{-l} \binom{j}{l}
   S_n^{(j)} }{v n!}\\ \times
\left(\Phi \left(e^{\frac{2 i m \pi }{v}},l,\frac{\pi -i v \log \left(a b^{-1/v}\right)}{2 \pi }\right)
   \left(\log \left(\frac{2 i \pi }{v}\right) (1-l)_l+\sum _{p=1}^l (-1)^{l+p} (1-l)^{-1+p} p
   S_l^{(p)}\right)\right. \\ \left.
-(1-l)_l \Phi'\left(e^{\frac{2 i m \pi }{v}},l,\frac{\pi -i v \log
   \left(a b^{-1/v}\right)}{2 \pi }\right)\right)
\end{multline}
where $Re(v)>1,Re(m)>0$ and there exists a singularity at $x=1/a$. If $Re(b)<0$ then $Re(a)<0$ and the singularity exists at $x=i/a$.
\end{example}
\begin{example}
In this example we use equation (\ref{eq:malm1}) and form two equations by setting $a\to 1,b\to -1,m\to m, m\to,s$ and take their difference and simplify. Note we treat the quantity $0^0=1$, see \cite{mobius}, during the evaluation of the double series;
\begin{multline}
\int_0^{\infty } \frac{x^{s-1}-x^{m-1} }{(1-x)^{n+1}} \log (\log (x))\, dx\\
=\sum _{j=0}^n \sum
   _{l=0}^j \frac{(-1)^{-j+l-m} i^{-l} e^{i \left(-\frac{1}{2}+m-n\right) \pi } (1-m)^{j-l} (2 \pi )^{1-l} \binom{j}{l}
   S_n^{(j)} }{n!}\\ \times
\left(e^{-2 i m \pi } \text{Li}_l\left(e^{2 i m \pi }\right) \left(\log (2 i \pi ) (1-l)_l+\sum _{p=1}^l
   (-1)^{l+p} (1-l)^{p-1} p S_l^{(p)}\right)\right. \\ \left.
-(1-l)_l \Phi'\left(e^{2 i m \pi
   },l,1\right)\right)\\
-\sum _{j=0}^n \sum _{l=0}^j \frac{(-1)^{-j+l-s} i^{-l} e^{i \pi 
   \left(-\frac{1}{2}-n+s\right)} (2 \pi )^{1-l} (1-s)^{j-l} \binom{j}{l} S_n^{(j)} }{n!}\\ \times
\left(e^{-2 i \pi  s}
   \text{Li}_l\left(e^{2 i \pi  s}\right) \left(\log (2 i \pi ) (1-l)_l+\sum _{p=1}^l (-1)^{l+p} (1-l)^{p-1} p
   S_l^{(p)}\right)\right. \\ \left.
-(1-l)_l \Phi'\left(e^{2 i \pi  s},l,1\right)\right)
\end{multline}
where $Re(s)>0,Re(m)>0$, where there exists a singularity at $x=i$.
\end{example}
\begin{example}
In this example we use equation (\ref{eq:thm2}) and take the first partial derivative with respect to $k$ and set $k=-1$ and simplify. When $0< Re(a) < 1/2, Re(b)>0$ then singularity exists at $x=1/a$, when $Re(a)<0, Re(b)<0$ then singularity exists at $x=i/a$, when $Re(a)>0$ and $Re(b)<0$ then singularity exists at $x=-i/a$.
\begin{multline}\label{eq:malm_la}
\int_0^{\infty } \frac{x^{-1+m}  }{\left(1+b x^v\right)^{n+1}\log (a x)}\log (\log (a x)) \, dx\\
=\sum _{j=0}^n
   \sum _{l=0}^j \frac{(-1)^{-j} b^{-\frac{m}{v}} e^{\frac{i \pi  \left(m-\left(\frac{1}{2}+n\right) v\right)}{v}} (2
   \pi )^{-l} \left(1-\frac{m}{v}\right)^{j-l} \left(-\frac{1}{v}\right)^l \left(\frac{i}{v}\right)^{-1-l} \binom{j}{l}
   S_n^{(j)} }{v n!}\\ \times
\left(\Phi \left(e^{\frac{2 i m \pi }{v}},1+l,\frac{\pi -i v \log \left(a b^{-1/v}\right)}{2 \pi }\right)
   \left(\log \left(\frac{2 i \pi }{v}\right) (-1)^l l!+\sum _{p=1}^l (-1)^{l+p} (-l)^{-1+p} p S_l^{(p)}\right)\right. \\ \left.
-(-1)^l
   l! \Phi'\left(e^{\frac{2 i m \pi }{v}},1+l,\frac{\pi -i v \log \left(a b^{-1/v}\right)}{2
   \pi }\right)\right)
\end{multline}
where $Re(m)>0$.
\end{example}
\begin{example}
In this example we use equation (\ref{eq:malm_la}) and form a second equation by replacing $m\to s$ and take their difference and simplify;
\begin{multline}\label{eq:malm_la1}
\int_0^{\infty } \frac{\left(x^{s-1}-x^{m-1}\right) \log (\log (a x))}{\left(1+b x^v\right)^{n+1} \log (a x)} \,
   dx\\
=-\sum _{j=0}^n \sum _{l=0}^j \frac{(-1)^{-j} b^{-\frac{m}{v}} e^{\frac{i \pi  \left(m-\left(\frac{1}{2}+n\right)
   v\right)}{v}} (2 \pi )^{-l} \left(1-\frac{m}{v}\right)^{j-l} \left(-\frac{1}{v}\right)^l
   \left(\frac{i}{v}\right)^{-1-l} \binom{j}{l} S_n^{(j)} }{v n!}\\ \times
\left(\Phi \left(e^{\frac{2 i m \pi }{v}},1+l,\frac{\pi -i v\log \left(a b^{-1/v}\right)}{2 \pi }\right) \left((-1)^l l! \log \left(\frac{2 i \pi }{v}\right)+\sum _{p=1}^l(-1)^{l+p} (-l)^{-1+p} p S_l^{(p)}\right)\right. \\ \left.
-(-1)^l l! \Phi'\left(e^{\frac{2 i m \pi}{v}},1+l,\frac{\pi -i v \log \left(a b^{-1/v}\right)}{2 \pi }\right)\right)\\
+\sum _{j=0}^n \sum _{l=0}^j\frac{(-1)^{-j} b^{-\frac{s}{v}} e^{\frac{i \pi  \left(s-\left(\frac{1}{2}+n\right) v\right)}{v}} (2 \pi )^{-l}\left(1-\frac{s}{v}\right)^{j-l} \left(-\frac{1}{v}\right)^l \left(\frac{i}{v}\right)^{-1-l} \binom{j}{l} S_n^{(j)}}{v n!}\\ \times
 \left(\Phi \left(e^{\frac{2 i \pi  s}{v}},1+l,\frac{\pi -i v \log \left(a b^{-1/v}\right)}{2 \pi }\right)
   \left((-1)^l l! \log \left(\frac{2 i \pi }{v}\right)+\sum _{p=1}^l (-1)^{l+p} (-l)^{-1+p} p S_l^{(p)}\right)\right. \\ \left.
-(-1)^l
   l! \Phi'\left(e^{\frac{2 i \pi  s}{v}},1+l,\frac{\pi -i v \log \left(a b^{-1/v}\right)}{2
   \pi }\right)\right)
\end{multline}
where $Re(s)>0,Re(m)>0$. There exists a singularity at $x=1/a$.
\end{example}
\begin{example}
In this example we use equation (\ref{eq:malm_la1}) and set $a=b=1$ and simplify;
\begin{multline}
\int_0^{\infty } \frac{\left(-x^{-1+m}+x^{-1+s}\right) \log (\log (x))}{\left(1+x^v\right)^{n+1} \log (x)} \,
   dx\\
=-\sum _{j=0}^n \sum _{l=0}^j \frac{(-1)^{-j} e^{\frac{i \pi  \left(m-\left(\frac{1}{2}+n\right) v\right)}{v}} (2 \pi)^{-l} \left(1-\frac{m}{v}\right)^{j-l} \left(-\frac{1}{v}\right)^l \left(\frac{i}{v}\right)^{-1-l} \binom{j}{l}
   S_n^{(j)} }{v n!}\\ \times
\left(\Phi \left(e^{\frac{2 i m \pi }{v}},1+l,\frac{1}{2}\right) \left((-1)^l l! \log \left(\frac{2 i \pi
   }{v}\right)+\sum _{p=1}^l (-1)^{l+p} (-l)^{-1+p} p S_l^{(p)}\right)\right. \\ \left.
-(-1)^l l!
   \Phi'\left(e^{\frac{2 i m \pi }{v}},1+l,\frac{1}{2}\right)\right)\\
+\sum _{j=0}^n \sum
   _{l=0}^j \frac{(-1)^{-j} e^{\frac{i \pi  \left(s-\left(\frac{1}{2}+n\right) v\right)}{v}} (2 \pi )^{-l}
   \left(1-\frac{s}{v}\right)^{j-l} \left(-\frac{1}{v}\right)^l \left(\frac{i}{v}\right)^{-1-l} \binom{j}{l} S_n^{(j)}
   }{v n!}\\ \times
\left(\Phi \left(e^{\frac{2 i \pi  s}{v}},1+l,\frac{1}{2}\right) \left((-1)^l l! \log \left(\frac{2 i \pi
   }{v}\right)+\sum _{p=1}^l (-1)^{l+p} (-l)^{-1+p} p S_l^{(p)}\right)\right. \\ \left.
-(-1)^l l!
   \Phi'\left(e^{\frac{2 i \pi  s}{v}},1+l,\frac{1}{2}\right)\right)
\end{multline}
where $0 < Re(m) < Re(s) < n; n = 1,2,.., |\arg(-a)|<\pi$.
\end{example}
\begin{example}
Here we use equation (\ref{eq:malm1}) and set $a\to1, b\to1, v\to3, n\to n-1, m\to1$ and simplify;
\begin{multline}\label{eq:malm1_ex}
\int_0^{\infty } \frac{\log (\log (x))}{\left(1+x^3\right)^n} \, dx
=\sum _{j=0}^{-1+n} \sum _{l=0}^j
   \frac{(-1)^{-j+l} i^{-l} 2^{1+j-2 l} 3^{-1-j+l} e^{\frac{1}{3} i \left(1-3 \left(-\frac{1}{2}+n\right)\right) \pi
   } \pi ^{1-l} \binom{j}{l} S_{-1+n}^{(j)} }{(-1+n)!}\\ \times
\left(\Phi \left(e^{\frac{2 i \pi }{3}},l,\frac{1}{2}\right) \left(\log
   \left(\frac{2 i \pi }{3}\right) (1-l)_l+\sum _{p=1}^l (-1)^{l+p} (1-l)^{-1+p} p S_l^{(p)}\right)\right. \\ \left.
-(1-l)_l
   \Phi'\left(e^{\frac{2 i \pi }{3}},l,\frac{1}{2}\right)\right)
\end{multline}
where there exists a singularity at $x=1$.
\end{example}
\begin{example}
Here we use equation (\ref{eq:malm1_ex}) and set $n=3$ and simplify;
\begin{equation}
\int_0^{\infty } \frac{\log (\log (x))}{1+x^3} \, dx=\frac{2}{3} e^{-\frac{1}{6} (i \pi )} \pi 
   \left(\frac{\log \left(\frac{2 i \pi }{3}\right)}{1-e^{\frac{2 i \pi
   }{3}}}-\Phi'\left(e^{\frac{2 i \pi }{3}},0,\frac{1}{2}\right)\right)
\end{equation}
where there exists a singularity at $x=1$.
\end{example}
\begin{example}
Here we use equation (\ref{eq:malm1}) and set $a\to1, b\to1, v\to4, n\to 1, m\to1$ and simplify;
\begin{multline}
\int_0^{\infty } \frac{\log (\log (x))}{\left(1+x^4\right)^2} \, dx=\frac{3 i \pi ^2+(8+8 i) \,
   _2F_1\left(\frac{1}{2},1;\frac{3}{2};i\right)+6 \pi  \left(\log \left(\frac{\pi }{2}\right)-(1-i)
   \Phi'\left(i,0,\frac{1}{2}\right)\right)}{16 \sqrt{2}}
\end{multline}
where there exists a singularity at $x=1$.
\end{example}
\begin{example}
Here we use equation (\ref{eq:malm1}) and set $a\to1, b\to1, v\to 2, n\to 3, m\to1$ and simplify; 
\begin{equation}
\int_0^{\infty } \frac{\log (\log (x))}{\left(1+x^2\right)^4} \, dx=\frac{3 C}{4 \pi }+\frac{1}{96} \pi 
   \left(22 i+15 \log \left(\frac{2 i \pi  \Gamma \left(-\frac{1}{4}\right)^2}{9 \Gamma
   \left(-\frac{3}{4}\right)^2}\right)\right)
\end{equation}
where there exists a singularity at $x=1$.
\end{example}
\begin{example}
Here we use equation (\ref{eq:malm1}) and set $a\to 2, b\to 4, v\to 2, n\to 2, m\to1$ and simplify; 
\begin{multline}
\int_0^{\infty } \frac{\log (\log (2 x))}{\left(1+4 x^2\right)^3} \, dx=\frac{C}{4 \pi }+\frac{i \pi
   }{8}+\frac{3}{16} \pi  \left(\frac{1}{2} \log (i \pi
   )-\Phi'\left(-1,0,\frac{1}{2}\right)\right)
\end{multline}
where there exists a singularity at $x=1/2$.
\end{example}
\begin{example}
Here we use equation (\ref{eq:malm1}) and set $a\to 3, b\to 9, v\to 2, n\to 2, m\to1$ and simplify; 
\begin{equation}
\int_0^{\infty } \frac{\log (\log (3 x))}{\left(1+9 x^2\right)^3} \, dx=\frac{C}{6 \pi }+\frac{i \pi
   }{12}+\frac{1}{8} \pi  \left(\frac{1}{2} \log (i \pi
   )-\Phi'\left(-1,0,\frac{1}{2}\right)\right)
\end{equation}
where there exists a singularity at $x=1/3$.
\end{example}
\begin{example}
Here we use equation (\ref{eq:malm1}) and set $a\to 4, b\to 16, v\to 2, n\to 2, m\to1$ and simplify; 
\begin{multline}
\int_0^{\infty } \frac{\log (\log (4 x))}{\left(1+16 x^2\right)^3} \, dx=\frac{C}{8 \pi }+\frac{i \pi
   }{16}+\frac{3}{32} \pi  \left(\frac{1}{2} \log (i \pi
   )-\Phi'\left(-1,0,\frac{1}{2}\right)\right)
\end{multline}
where there exists a singularity at $x=1/4$.
\end{example}
\begin{example}
Here we use equation (\ref{eq:malm1}) and set $a\to2, b\to8, v\to 3, n\to 2, m\to 3/2$ and simplify; 
\begin{multline}
\int_0^{\infty } \frac{\sqrt{x} \log (\log (2 x))}{\left(1+8 x^3\right)^3} \, dx=\frac{C}{6 \sqrt{2} \pi
   }+\frac{i \pi }{12 \sqrt{2}}+\frac{\pi  \left(\frac{1}{2} \log \left(\frac{2 i \pi
   }{3}\right)-\Phi'\left(-1,0,\frac{1}{2}\right)\right)}{8 \sqrt{2}}
\end{multline}
where there exists a singularity at $x=1/2$.
\end{example}
%
%
%
\section{Derivations of formula in known volumes}
In this section we derive integrals in current literature in terms of finite series.
\begin{example}
Derivation of an alternate form for equation (2.1.2.4) in \cite{brychkov_m} in terms of the finite series involving the Hurwitz-Lerch zeta function.  This is an example of being able to evaluate the finite series over a wider range of values without having to consider the singularity of the Mellin transform. In this example we use equation (\ref{eq:thm2}) and set $k\to 0,v\to 1,b\to -\frac{1}{a},n\to n-1,m\to s$ and simplify;
\begin{multline}\label{eq:6.15}
\int_0^{\infty } \frac{x^{-1+s}}{(a-x)^n} \, dx
=-\sum _{j=0}^n \sum _{l=0}^j \frac{(-1)^{-j+l} i^{-l}
   \left(-\frac{1}{a}\right)^{-s} a^{-n} e^{i \pi  \left(\frac{1}{2}-n+s\right)} (2 \pi )^{1-l} (1-s)^{j-l}
   \binom{j}{l} }{(-1+n)!}\\ \times
\Phi \left(e^{2 i \pi  s},l,\frac{\pi -i \log \left(-a^2\right)}{2 \pi }\right) (1-l)_l
   S_{-1+n}^{(j)}\\
=-\frac{\left(\pi  (-a)^{s-n}\right) \left((1-n+s) (2-n+s)_{-1+n}\right)}{((n-1)! \sin (s
   \pi )) s}
\end{multline}
where $0 < Re(s) < n; n = 1,2,.., |\arg(-a)|<\pi$.
\end{example}
\begin{example}
Series representation for equation (2.1.2.4) in \cite{brychkov_m}. In this example we equate the right-hand sides of equation (\ref{eq:6.15}) and simplify;
\begin{multline}
\sum _{j=0}^n \sum _{l=0}^j (-1)^{-j+l} i^{-l} \left(-\frac{1}{a}\right)^{-s} a^{-n} e^{i \pi  \left(\frac{1}{2}-n+s\right)}
   (2 \pi )^{1-l} (1-s)^{j-l} \binom{j}{l}\\
 \Phi \left(e^{2 i \pi  s},l,\frac{\pi -i \log \left(-a^2\right)}{2 \pi }\right) (1-l)_l
   S_{-1+n}^{(j)}\\
=-\frac{\left(\pi  (-a)^{s-n}\right) \left((1-n+s) (2-n+s)_{-1+n}\right)}{\sin (s \pi ) s}
\end{multline}
where $0 < Re(s) < n; n = 1,2,.., |\arg(-a)|<\pi$.
\end{example}
\begin{example}
Derivation of an alternate form of equation (2.6.4.8) in \cite{prud1}. In this example we use equation (\ref{eq:thm2}) and set $k\to 1,a\to 1,m\to \alpha ,v\to \mu ,b\to z^{-\mu },n\to m-1$ and simplify;
\begin{multline}
\int_0^{\infty } \frac{x^{-1+\alpha } \log (x)}{\left(z^{\mu }+x^{\mu }\right)^m} \, dx\\
=\sum _{j=0}^{m-1} \sum
   _{l=0}^j \frac{z^{-m \mu } (-1)^{-j} e^{\frac{i \pi  \left(\alpha -\left(-\frac{1}{2}+m\right) \mu \right)}{\mu }}
   (2 \pi )^{2-l} \left(z^{-\mu }\right)^{-\frac{\alpha }{\mu }} \left(1-\frac{\alpha }{\mu }\right)^{j-l}
   \left(-\frac{1}{\mu }\right)^l \left(\frac{i}{\mu }\right)^{1-l} \binom{j}{l} }{\mu  (-1+m)!}\\ \times
\Phi \left(e^{\frac{2 i \pi  \alpha
   }{\mu }},-1+l,\frac{\pi -i \mu  \log \left(\left(z^{-\mu }\right)^{-1/\mu }\right)}{2 \pi }\right) (2-l)_l
   S_{-1+m}^{(j)}\\
=\frac{\left(z^{\alpha -\mu  m} \left(1-\frac{\alpha }{\mu }\right)_{m-1} \pi \right)
  }{(m-1)! \mu ^2 \sin \left(\frac{\alpha  \pi }{\mu }\right)} \left(\mu  \log (z)-\sum _{k=1}^{m-1} \frac{\mu }{k \mu -\alpha }-\pi  \cot \left(\frac{\alpha  \pi }{\mu
   }\right)\right)
\end{multline}
where $Re(\mu)>0, \mu |\arg z|< \pi, 0< Re(\alpha)<\mu m, m=1,2,3,..,$.
\end{example}
\begin{example}
Series form for equation (2.6.4.8) in \cite{prud1}. 
\begin{multline}
\sum _{j=0}^{m-1} \sum _{l=0}^j (-1)^{-j} e^{\frac{i \pi  \left(\alpha -\left(-\frac{1}{2}+m\right) \mu
   \right)}{\mu }} (2 \pi )^{2-l} (-1)^l i^{1-l} \mu ^{-j+l} (\mu -\alpha )^{j-l} \binom{j}{l}\\
 \Phi \left(e^{\frac{2 i
   \pi  \alpha }{\mu }},-1+l,\frac{\pi -i \mu  \log \left(\left(z^{-\mu }\right)^{-1/\mu }\right)}{2 \pi }\right)
   (2-l)_l S_{m-1}^{(j)}\\
=\frac{\left(1-\frac{\alpha }{\mu }\right)_{m-1} \pi \left(\mu  \log (z)-\sum
   _{k=1}^{m-1} \frac{\mu }{k \mu -\alpha }-\pi  \cot \left(\frac{\alpha  \pi }{\mu }\right)\right)}{\sin
   \left(\frac{\alpha  \pi }{\mu }\right)}
\end{multline}
where $Re(\mu)>0,0< Re(\alpha)<\mu m, \mu |\arg z|<\pi$.
\end{example}
\begin{example}
Derivation of equation (3.194.4(11)) in \cite{grad} in terms of the finite series involving the Hurwitz-Lerch zeta function. When $Re(b) < 0$ then the singularity exists at $x--i$. In this equation we use equation (\ref{eq:thm2}) and set $k\to 0,v\to 1,a\to 1,m\to a$ and simplify;
\begin{multline}
\int_0^{\infty } \frac{x^{-1+a}}{(1+b x)^{n+1}} \, dx\\
=\sum _{j=0}^n \sum _{l=0}^j \frac{(-1)^{-j+l} i^{-l}
   (1-a)^{j-l} b^{-a} e^{i \left(-\frac{1}{2}+a-n\right) \pi } (2 \pi )^{1-l} \binom{j}{l} }{n!}\\ \times
   \Phi \left(e^{2 i a \pi
   },l,\frac{\pi -i \log \left(\frac{1}{b}\right)}{2 \pi }\right) (1-l)_l S_n^{(j)}\\
=\frac{(-1)^n \pi 
   \binom{a-1}{n} \csc (a \pi )}{b^a}
\end{multline}
where $0 < Re(a) < n+1; n = 1,2,.., |\arg(b)|<\pi$.
\end{example}
\begin{example}
Derivation of the series form for equation (3.194.4(11)) in \cite{grad}.
\begin{multline}
\sum _{j=0}^n \sum _{l=0}^j \frac{(-1)^{-j+l} i^{-l} (1-a)^{j-l} e^{i \left(-\frac{1}{2}+a-n\right) \pi } (2 \pi
   )^{1-l} \binom{j}{l} }{n!}\\ \times
   \Phi \left(e^{2 i a \pi },l,\frac{\pi -i \log \left(\frac{1}{b}\right)}{2 \pi }\right) (1-l)_l
   S_n^{(j)}\\
=(-1)^n \pi  \binom{a-1}{n} \csc (a \pi )
\end{multline}
where $Re(a)>0, n=1,2,3,..,$.
\end{example}
\begin{example}
Derivation of entry (4.267.22) in \cite{grad}. In this example we use equation (eq:thm2) and set $k\to -1,a\to 1,b\to 1,v\to 2 (2 n+1),n\to 0,j\to 0,l\to 0$. Then we form four equations by replacing $m\to, p, m\to p+2, m\to q, m\to q+2$. Then we add the first and second equations and take the difference of the third and fourth and simplify;
\begin{multline}
\int_0^{\infty } \frac{\left(1+x^2\right) \left(x^{p-1}-x^{q-1}\right)}{\left(1+x^{2+4 n}\right) \log (x)} \,
   dx\\
=-2 \left(\tanh ^{-1}\left(e^{\frac{i p \pi }{2+4 n}}\right)+\tanh ^{-1}\left(e^{\frac{i (2+p) \pi }{2+4
   n}}\right)-\tanh ^{-1}\left(e^{\frac{i \pi  q}{2+4 n}}\right)-\tanh ^{-1}\left(e^{\frac{i \pi  (2+q)}{2+4
   n}}\right)\right)\\
=\log \left(\tan \left(\frac{p \pi }{4 (2 n+1)}\right) \tan \left(\frac{(p+2) \pi }{4 (2
   n+1)}\right) \cot \left(\frac{q \pi }{4 (2 n+1)}\right) \cot \left(\frac{(q+2) \pi }{4 (2
   n+1)}\right)\right)
\end{multline}
where $Re(p)>0,Re(q)>0$.
\end{example}
\begin{example}
Derivation of entry (4.267.23) in \cite{grad}. In this example we use equation (eq:thm2) and set $a\to 1,n\to 0,j\to 0,l\to 0,m\to m+1,b\to -1,v\to 2 n$. Then we simplify the Hurwitz-Lerch zeta function by using the formula $\text{Li}_{-s}(z)=\Phi (z,-s,0)$ where the $Re(s)>0$. Next we form four equations by replacing $m\to p, m\to p+2, m\to q, m\to q+2$. Then we add the first and fourth and take the difference of the second and third and simplify.
\begin{multline}
\int_0^{\infty } \frac{\left(-1+x^2\right) \left(x^p-x^q\right) \log ^k(x)}{-1+x^{2 n}} \,
   dx\\
=\frac{1}{n^2}\left(\frac{i}{n}\right)^{-1+k} \pi ^{1+k} \left(\text{Li}_{-k}\left(e^{\frac{i (1+p) \pi
   }{n}}\right)-\text{Li}_{-k}\left(e^{\frac{i (3+p) \pi }{n}}\right)\right. \\ \left.
   -\text{Li}_{-k}\left(e^{\frac{i \pi 
   (1+q)}{n}}\right)+\text{Li}_{-k}\left(e^{\frac{i \pi  (3+q)}{n}}\right)\right)
\end{multline}
where $Re(p)>0,Re(q)>0$.
Then we set $k\to -1$ and simplify;
\begin{multline}
\int_0^{\infty } \frac{\left(1-x^2\right) \left(x^{-1+p}-x^{-1+q}\right)}{\left(1-x^{2 n}\right) \log (x)} \,
   dx\\
=\log \left(\csc \left(\frac{(2+p) \pi }{2 n}\right) \csc \left(\frac{\pi  q}{2 n}\right) \sin \left(\frac{p \pi
   }{2 n}\right) \sin \left(\frac{\pi  (2+q)}{2 n}\right)\right)
\end{multline}
where $Re(p)>0,Re(q)>0$.
\end{example}
\begin{example}
Derivation of entry (4.267.30) in \cite{grad}. In this example we use equation (eq:thm2) and set $a\to 1,n\to 0,j\to 0,l\to 0,m\to m+1,b\to -1,v\to p+q+2 s$ and simplify the Hurwitz-Lerch zeta function using $\text{Li}_{-s}(z)=\Phi (z,-s,0)$ to get;
\begin{equation}
\int_0^{\infty } \frac{x^m \log ^k(x)}{1-x^{p+q+2 s}} \, dx=\frac{(2 \pi )^{1+k} \left(\frac{i}{p+q+2
   s}\right)^{-1+k} \text{Li}_{-k}\left(e^{\frac{2 i (1+m) \pi }{p+q+2 s}}\right)}{(p+q+2 s)^2}
\end{equation}
where $Re(p)>0,Re(q)>0,Re(m)>0$.
Then we form four equations by replacing $m\to s-1, m\to p+s-1, m\to q+s-1, m\to p+q+s-1$. Then we add the first and the fourth then take its difference from the second and third equations and simplify;
\begin{equation}
\int_0^{\infty } \frac{x^{-1+s} \left(1-x^p\right) \left(1-x^q\right)}{\left(1-x^{p+q+2 s}\right) \log (x)} \,
   dx=\log \left(\frac{1+\cos \left(\frac{\pi  (p+q)}{p+q+2 s}\right)}{1+\cos \left(\frac{\pi  (p-q)}{p+q+2
   s}\right)}\right)
\end{equation}
where $Re(p)>0,Re(q)>0,Re(m)>0$.
\end{example}
\begin{example}
Generalized Gr\"{o}bner integral form given by entry (4.267.30) in \cite{grad}. In this example we use equation (eq:thm2) and set $a\to 1, b\to -1$ and simplify the Hurwitz-Lerch zeta function using $\text{Li}_{-s}(z)=\Phi (z,-s,0)$ to get;
\begin{multline}
\int_0^{\infty } x^{-1+m} \left(1-x^v\right)^{-1-n} \log ^k(x) \, dx\\
=\sum _{j=0}^n \sum _{l=0}^j
   \frac{(-1)^{-j-\frac{m}{v}} e^{\frac{i \pi  \left(m-\left(\frac{1}{2}+n\right) v\right)}{v}} (2 \pi )^{1+k-l}
   \left(1-\frac{m}{v}\right)^{j-l}  }{v n!}\\ \times
\left(-\frac{1}{v}\right)^l \left(\frac{i}{v}\right)^{k-l} \binom{j}{l}
   \text{Li}_{-k+l}\left(e^{\frac{2 i m \pi }{v}}\right)(1+k-l)_l S_n^{(j)}
\end{multline}
where $Re(p)>0,Re(q)>0,Re(m)>0$. Then we set $k=-1$ and form four equations by replacing $m\to p, m\to q, m\to p+s, m\to q+s$. Next we add the first and fourth equations then take its difference from the second and third equations and simplify;
\begin{multline}
\int_0^{\infty } \frac{\left(x^{p-1}-x^{q-1}\right) \left(1-x^s\right)}{\left(1-x^v\right)^{n+1} \log (x)} \,
   dx\\
=\sum _{j=0}^n \sum _{l=0}^j \frac{i (-1)^{-j+l-\frac{p}{v}-\frac{q}{v}-\frac{s}{v}} 
   \left(1-\frac{p}{v}\right)^{-l} \left(1-\frac{q}{v}\right)^{-l} \left(1-\frac{p+s}{v}\right)^{-l}
   \left(1-\frac{q+s}{v}\right)^{-l} \left(-\frac{1}{v}\right)^l \left(\frac{i}{v}\right)^{-l} 
   }{(2 \pi )^{l}n!}\\ \times
\binom{j}{l} l!\left(-(-1)^{\frac{q}{v}+\frac{s}{v}} e^{\frac{i \pi  \left(p-\left(\frac{1}{2}+n\right) v\right)}{v}}
   \left(1-\frac{p}{v}\right)^j \left(1-\frac{q}{v}\right)^l \left(1-\frac{p+s}{v}\right)^l
   \left(1-\frac{q+s}{v}\right)^l \text{Li}_{1+l}\left(e^{\frac{2 i p \pi }{v}}\right)\right. \\ \left.
+(-1)^{\frac{p}{v}+\frac{s}{v}}
   e^{\frac{i \pi  \left(q-\left(\frac{1}{2}+n\right) v\right)}{v}} \left(1-\frac{p}{v}\right)^l
   \left(1-\frac{q}{v}\right)^j \left(1-\frac{p+s}{v}\right)^l \left(1-\frac{q+s}{v}\right)^l
   \text{Li}_{1+l}\left(e^{\frac{2 i \pi  q}{v}}\right)\right. \\ \left.
+(-1)^{q/v} e^{\frac{i \pi 
   \left(p+s-\left(\frac{1}{2}+n\right) v\right)}{v}} \left(1-\frac{p}{v}\right)^l \left(1-\frac{q}{v}\right)^l
   \left(1-\frac{p+s}{v}\right)^j \left(1-\frac{q+s}{v}\right)^l \text{Li}_{1+l}\left(e^{\frac{2 i \pi 
   (p+s)}{v}}\right)\right. \\ \left.
-(-1)^{p/v} e^{\frac{i \pi  \left(q+s-\left(\frac{1}{2}+n\right) v\right)}{v}}
   \left(1-\frac{p}{v}\right)^l \left(1-\frac{q}{v}\right)^l \left(1-\frac{p+s}{v}\right)^l
   \left(1-\frac{q+s}{v}\right)^j \text{Li}_{1+l}\left(e^{\frac{2 i \pi  (q+s)}{v}}\right)\right)
   S_n^{(j)}
\end{multline}
where $Re(p)>0,Re(q)>0,Re(s)>0,Re(v)>0$ and there exists a singularity at $x=i$.
\end{example}
\begin{example}
In this example we use equation (\ref{eq:thm2}) and take the first partial derivative with respect to $k$ and set $k=-1$ and simplify;
\begin{multline}
\int_0^{\infty } \frac{x^{-1+m} \log (\log (x))}{\log (x) \left(b+c x^v\right){}_{1+n}} \, dx
=-\sum _{j=0}^n
   \frac{(-1)^j c^{-\frac{m}{v}} e^{\frac{i m \pi }{v}} (b+j)^{-1+\frac{m}{v}} \binom{n}{j} }{n!}\\
   \left(\Phi
   \left(e^{\frac{2 i m \pi }{v}},1,\frac{\pi -i v \log \left(c^{-1/v} (b+j)^{1/v}\right)}{2 \pi }\right) \log
   \left(\frac{2 i \pi }{v}\right)\right. \\ \left.
-\Phi'\left(e^{\frac{2 i m \pi }{v}},1,\frac{\pi -i v
   \log \left(c^{-1/v} (b+j)^{1/v}\right)}{2 \pi }\right)\right)
\end{multline}
where $Re(m)>0$ and there exists a singularity at $x=i$.
\end{example}
\begin{example}
In this example we use equation (\ref{eq:thm2}) and set $k\to -1,a\to 1,m\to m+1$, then form a second equation by replacing $m\to s$ and take there difference.
\begin{multline}\label{eq:poch_loginv_1}
\int_0^{\infty } \frac{-x^m+x^s}{\log (x) \left(b+c x^v\right){}_{1+n}} \, dx
=\sum _{j=0}^n \frac{i (-1)^j
   c^{-\frac{1+m+s}{v}} (b+j)^{-1+\frac{1}{v}} \binom{n}{j}}{n!}\\ \times
 \left(-c^{s/v} e^{\frac{i \pi  \left(1+m-\frac{3
   v}{2}\right)}{v}} (b+j)^{m/v} \Phi \left(e^{\frac{2 i (1+m) \pi }{v}},1,\frac{\pi -i v \log \left(c^{-1/v}
   (b+j)^{1/v}\right)}{2 \pi }\right)\right. \\ \left.
+c^{m/v} e^{\frac{i \pi  \left(1+s-\frac{3 v}{2}\right)}{v}} (b+j)^{s/v} \Phi
   \left(e^{\frac{2 i \pi  (1+s)}{v}},1,\frac{\pi -i v \log \left(c^{-1/v} (b+j)^{1/v}\right)}{2 \pi
   }\right)\right)
\end{multline}
where $Re(m)>0,Re(s)>0,Re(v)>0$ and there exists a singularity at $x=-i/2$ when $Re(c)>0$.
\end{example}
\begin{example}
In this example we use equation (\ref{eq:poch_loginv_1}) and set $c\to -1$ and simplify;
\begin{multline}
\int_0^{\infty } \frac{-x^m+x^s}{\log (x) \left(b-x^v\right){}_{1+n}} \, dx
=\sum _{j=0}^n \frac{i
   (-1)^{j-\frac{1+m+s}{v}} (b+j)^{-1+\frac{1}{v}} \binom{n}{j} }{n!}\\ \times
\left(-(-1)^{s/v} e^{\frac{i \pi  \left(1+m-\frac{3
   v}{2}\right)}{v}} (b+j)^{m/v} \Phi \left(e^{\frac{2 i (1+m) \pi }{v}},1,\frac{\pi -i v \log \left((-1)^{-1/v}
   (b+j)^{1/v}\right)}{2 \pi }\right)\right. \\ \left.
+(-1)^{m/v} e^{\frac{i \pi  \left(1+s-\frac{3 v}{2}\right)}{v}} (b+j)^{s/v} \Phi
   \left(e^{\frac{2 i \pi  (1+s)}{v}},1,\frac{\pi -i v \log \left((-1)^{-1/v} (b+j)^{1/v}\right)}{2 \pi
   }\right)\right)
\end{multline}
where $Re(m)>0,Re(s)>0,Re(v)>0,Re(b)\neq 1$ and there exists a singularity at $x=-i/2$.
\end{example}
\begin{example}
In this example we use equation (\ref{eq:thm2}) and set $k\to -1,a\to 1,c\to -1,m\to m+1$. Next we form four equations by replacing $m\to p, m\to q, m\to p+s, m\to q+s$. We add the first and fourth equations and take the difference of the second and third and simplify;
\begin{multline}
\int_0^{\infty } \frac{\left(x^p-x^q\right) \left(-1+x^s\right)}{\log (x) \left(b-x^v\right){}_{1+n}} \,
   dx
=\sum _{j=0}^n \frac{e^{i j \pi } (b+j)^{-1+\frac{1}{v}} \binom{n}{j} }{n!}\\ \times
\left((b+j)^{p/v} \Phi \left(e^{\frac{2 i
   (1+p) \pi }{v}},1,\frac{\log (b+j)}{2 \pi  i}\right)-(b+j)^{q/v} \Phi \left(e^{\frac{2 i \pi 
   (1+q)}{v}},1,\frac{\log (b+j)}{2 \pi  i}\right)\right. \\ \left.
-(b+j)^{\frac{p}{v}+\frac{s}{v}} \Phi \left(e^{\frac{2 i \pi 
   (1+p+s)}{v}},1,\frac{\log (b+j)}{2 \pi  i}\right)+(b+j)^{\frac{q}{v}+\frac{s}{v}} \Phi \left(e^{\frac{2 i \pi 
   (1+q+s)}{v}},1,\frac{\log (b+j)}{2 \pi  i}\right)\right)
\end{multline}
where $Re(p)>0,Re(q)>0,Re(s)>0,Re(v)>0,Re(b)\neq 1$ and there exists a singularity at $x=-i/2$.
\end{example}
\begin{example}
Infinite integral involving the  polylogarithm function. In this example we use equation (\ref{eq:thm1}) and expand the finite series over $j\in[0,0]+[1,n]$. Then set $a\to \left(-\frac{b}{c}\right)^{-1/v}$ and simplify using equation [Wolfram MathWorld,\href{https://mathworld.wolfram.com/LerchTranscendent.html}{(6)}];
\begin{multline}\label{eq:poly_ex1}
\int_0^{\infty } \frac{x^{-1+m} \log ^k\left(\left(-\frac{b}{c}\right)^{-1/v} x\right)}{\left(b+c
   x^v\right){}_{1+n}} \, dx\\
=\sum _{j=1}^n \frac{(-1)^j c^{-\frac{m}{v}} e^{\frac{i m \pi }{v}}
   (b+j)^{-1+\frac{m}{v}} (2 \pi )^{1+k} \left(\frac{i}{v}\right)^{-1+k} \binom{n}{j} }{v^2 n!}\\
+\frac{b^{-1+\frac{m}{v}} c^{-\frac{m}{v}} e^{\frac{i m \pi }{v}} (2 \pi )^{1+k}
   \left(\frac{i}{v}\right)^{-1+k} \text{Li}_{-k}\left(e^{\frac{2 i m \pi }{v}}\right)}{v^2 n!}\\ \times
\Phi \left(e^{\frac{2 i m \pi
   }{v}},-k,\frac{\pi -i v \log \left(\left(-\frac{b}{c}\right)^{-1/v} c^{-1/v} (b+j)^{1/v}\right)}{2 \pi
   }\right)
\end{multline}
where $Re(m)>0$, and there exists a singularity at $x=\frac{i e^{\left(-\frac{b}{c}\right)^{-1/v}}}{v-2}$.
\end{example}
\begin{example}
In this example we use equation (\ref{eq:poly_ex1}) and form a second equation by replacing $m\to s$ and take their difference and simplify;
\begin{multline}\label{eq:poly_ex2}
\int_0^{\infty } \frac{\left(-x^m+x^s\right) \log ^k\left(\left(-\frac{b}{c}\right)^{-1/v} x\right)}{\left(b+c
   x^v\right){}_{1+n}} \, dx
=\sum _{j=1}^n \frac{(-1)^j c^{-\frac{2+m+s}{v}} (2 \pi )^{1+k}
   \left(\frac{i}{v}\right)^{1+k} \binom{n}{j} }{(b+j)
   n!}\\ \times
\left(c^{\frac{1+s}{v}} e^{\frac{i (1+m) \pi }{v}}
   (b+j)^{\frac{1+m}{v}} \Phi \left(e^{\frac{2 i (1+m) \pi }{v}},-k,\frac{\pi -i v \log
   \left(\left(-\frac{b}{c}\right)^{-1/v} c^{-1/v} (b+j)^{1/v}\right)}{2 \pi }\right)\right. \\ \left.
-c^{\frac{1+m}{v}} e^{\frac{i
   \pi  (1+s)}{v}} (b+j)^{\frac{1+s}{v}} \Phi \left(e^{\frac{2 i \pi  (1+s)}{v}},-k,\frac{\pi -i v \log
   \left(\left(-\frac{b}{c}\right)^{-1/v} c^{-1/v} (b+j)^{1/v}\right)}{2 \pi }\right)\right)\\
+\frac{c^{-\frac{2+m+s}{v}} (2 \pi )^{1+k} \left(\frac{i}{v}\right)^{1+k} \left(b^{\frac{1+m}{v}}
   c^{\frac{1+s}{v}} e^{\frac{i (1+m) \pi }{v}} \text{Li}_{-k}\left(e^{\frac{2 i (1+m) \pi
   }{v}}\right)-b^{\frac{1+s}{v}} c^{\frac{1+m}{v}} e^{\frac{i \pi  (1+s)}{v}} \text{Li}_{-k}\left(e^{\frac{2 i \pi 
   (1+s)}{v}}\right)\right)}{b n!}
\end{multline}
where $Re(m)>-1,Re(s)>-1$, and there exists a singularity at $x=\frac{i e^{\left(-\frac{b}{c}\right)^{-1/v}}}{v-2}$.
\end{example}
\subsection{A few special case examples of equation (\ref{eq:poly_ex2})}
\begin{example}
When $k\to -1,b\to 1,c\to -1,v\to1, n\to1, m\to 1/2,s\to-1/2$,
\begin{equation}
\int_0^{\infty } \frac{\frac{1}{\sqrt{x}}-\sqrt{x}}{(1-x) (2-x) \log (x)} \, dx=-\frac{\Phi
   \left(-1,1,\frac{\pi -i (i \pi +\log (2))}{2 \pi }\right)}{\sqrt{2}}
\end{equation}
where there exists a singularity at $x=-ie$.
\end{example}
\begin{example}
When $k\to -1,b\to 1,c\to -1,v\to1, m\to 1/2,s\to-1/2$,
\begin{multline}
\int_0^{\infty } \frac{\frac{1}{\sqrt{x}}-\sqrt{x}}{\log (x) (1-x)_{1+n}} \, dx\\
=\sum _{j=1}^n \frac{(-1)^j
   \binom{n}{j} \left(-\sqrt{1+j} \Phi \left(-1,1,\frac{\pi -i \log (-1-j)}{2 \pi }\right)+(1+j)^{3/2} \Phi
   \left(-1,1,\frac{\pi -i \log (-1-j)}{2 \pi }\right)\right)}{(1+j) n!}
\end{multline}
where there exists a singularity at $x=-ie$.
\end{example}
\begin{example}
When $k\to -1,b\to 1,c\to -2,v\to1, m\to 1/2,s\to-1/2$,
\begin{multline}
\int_0^{\infty } \frac{\frac{1}{\sqrt{x}}-\sqrt{x}}{\log (2 x) (1-2 x)_{1+n}} \, dx
=\frac{\log (2)}{2 \sqrt{2}
   n!}\\
+\sum _{j=1}^n \frac{(-1)^j \binom{n}{j} }{4
   (1+j) n!}\left(-2 \sqrt{2} \sqrt{1+j} \Phi \left(-1,1,\frac{\pi -i \log
   (-1-j)}{2 \pi }\right)\right. \\ \left.
+\sqrt{2} (1+j)^{3/2} \Phi \left(-1,1,\frac{\pi -i \log (-1-j)}{2 \pi }\right)\right)
\end{multline}
where there exists a singularity at $x=-ie$.
\end{example}
\begin{example}
In this example we use equation (\ref{eq:thm1}) and take the first partial derivative with respect to $k$ then set $k=-1$ and simplify;
\begin{multline}\label{eq:poch_malm_ex1}
\int_0^{\infty } \frac{\left(-x^m+x^s\right) \log (\log (x))}{\log (x) \left(b+c x^v\right){}_{1+n}} \,
   dx\\
=-\sum _{j=0}^n \frac{i (-1)^j c^{-1/v} (b+j)^{-1+\frac{1}{v}} \binom{n}{j} }{n!}\left(c^{-\frac{m}{v}} e^{\frac{i
   \pi  \left(1+m-\frac{3 v}{2}\right)}{v}} (b+j)^{m/v} \right. \\ \left.
\left(\Phi \left(e^{\frac{2 i (1+m) \pi }{v}},1,\frac{\pi -i
   v \log \left(c^{-1/v} (b+j)^{1/v}\right)}{2 \pi }\right) \log \left(\frac{2 i \pi
   }{v}\right)\right.\right. \\ \left.\left.
-\Phi'\left(e^{\frac{2 i (1+m) \pi }{v}},1,\frac{\pi -i v \log
   \left(c^{-1/v} (b+j)^{1/v}\right)}{2 \pi }\right)\right)\right. \\ \left.
+c^{-\frac{s}{v}} e^{\frac{i \pi  \left(1+s-\frac{3
   v}{2}\right)}{v}} (b+j)^{s/v} \left(-\Phi \left(e^{\frac{2 i \pi  (1+s)}{v}},1,\frac{\pi -i v \log \left(c^{-1/v}
   (b+j)^{1/v}\right)}{2 \pi }\right) \log \left(\frac{2 i \pi
   }{v}\right)\right.\right. \\ \left.\left.
+\Phi'\left(e^{\frac{2 i \pi  (1+s)}{v}},1,\frac{\pi -i v \log
   \left(c^{-1/v} (b+j)^{1/v}\right)}{2 \pi }\right)\right)\right)
\end{multline}
where $Re(c)>0,Re(b)>0$, and there exists a singularity at $x=1$.
\end{example}
\begin{example}
In this example we use equation (\ref{eq:poch_malm_ex1}) and set $m\to \frac{1}{2},s\to -\frac{1}{2},b\to 1,c\to -1,v\to 1$ and simplify;
\begin{multline}
\int_0^{\infty } \frac{\left(\frac{1}{\sqrt{x}}-\sqrt{x}\right) \log (\log (x))}{\log (x) (1-x)_{1+n}} \,
   dx\\
=\sum _{j=0}^n \frac{i (-1)^j \binom{n}{j} }{n!}\left(-i \sqrt{1+j} \left(\Phi \left(-1,1,\frac{\pi -i \log (-1-j)}{2
   \pi }\right) \log (2 i \pi )\right.\right. \\ \left.\left.
-\Phi'\left(-1,1,\frac{\pi -i \log (-1-j)}{2 \pi
   }\right)\right)\right. \\ \left.
-\frac{i \left(-\Phi \left(-1,1,\frac{\pi -i \log (-1-j)}{2 \pi }\right) \log (2 i \pi
   )+\Phi'\left(-1,1,\frac{\pi -i \log (-1-j)}{2 \pi
   }\right)\right)}{\sqrt{1+j}}\right)
\end{multline}
where there exists a singularity at $x=-i$.
\end{example}
\begin{example}
The definite integral of the double logarithm function divided by a generalized rational function see \cite{malmsten}. In this example we use equation (\ref{eq:thm1}) and set $b\to 1,c\to 1,v\to 1,m\to 1,a\to 1$ and take the first partial derivative with respect to $k$ then set $k\to 0$ and simplify in terms of the Hurwitz zeta function using equation [Wolfram MathWorld, \href{https://mathworld.wolfram.com/HurwitzZetaFunction.html}{(1)}].  
\begin{multline}\label{eq:malm_poxh_1}
\int_0^{\infty } \frac{\log (\log (x))}{(1+x)_{1+n}} \, dx\\
=-\frac{1}{2 n!}\sum _{j=0}^n i^{2 j+1} \binom{n}{j} \left(\log (1+j) \left(-i \log \left(-4 \pi ^2\right)\right)+4 \pi \zeta'\left(0,\frac{\pi -i \log (1+j)}{2 \pi }\right)\right)
\end{multline}
where $n\geq 1$ and there exists a singularity at $x=1$.
\end{example}
\begin{example}
In this example we look at the plot of the integrand in equation (\ref{eq:malm_poxh_1}) and where the singularities exist along the $x$-axis over $n\in[0,3]$.
\begin{figure}[H]
\includegraphics[scale=0.6]{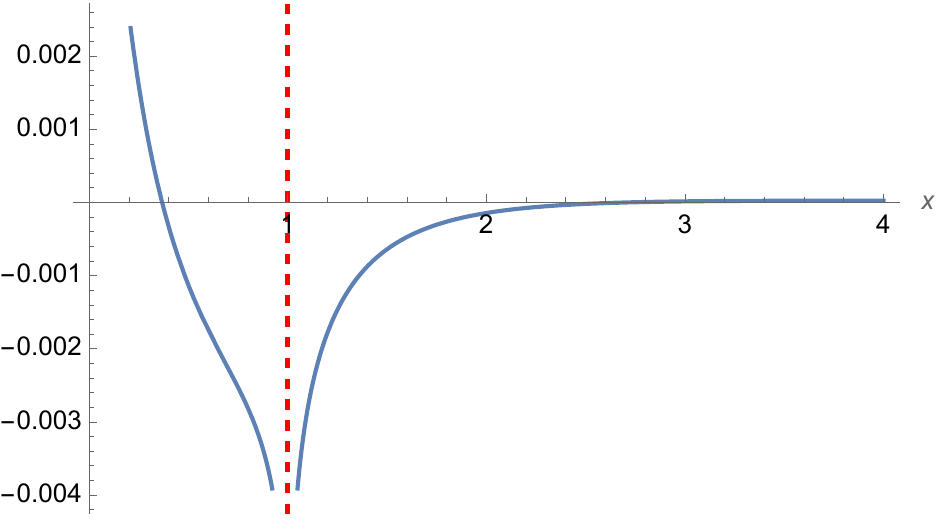}
\caption{Plot of the integrand in (\ref{eq:malm_poxh_1})}
   \label{fig:fig2}
\end{figure}
\end{example}
\begin{example}
Improper definite integral of the difference of powers function divided by the product of a logarithm and rising factorial. In this example we use equation (\ref{eq:thm1}) and set $b\to 1,c\to 1,v\to 1,a\to 1,m\to m+1$ then we form a second equation by replacing $m\to s$ and take their difference and simplify. Then we set $k\to -1,m\to 1/2, s\to -1/2$ and simplify. 
\begin{multline}
\int_0^{\infty } \frac{-x^m+x^s}{\log (x) (1+x)_{1+n}} \, dx
=-\sum _{j=0}^n \frac{(-1)^j \binom{n}{j} }{n!}\left(e^{i m \pi } (1+j)^m \Phi \left(e^{2 i m \pi },1,\frac{\pi -i \log(1+j)}{2 \pi }\right)\right. \\ \left.
-e^{i \pi  s} (1+j)^s \Phi \left(e^{2 i \pi  s},1,\frac{\pi -i \log (1+j)}{2 \pi }\right)\right)
\end{multline}
where $n\geq 1$ and there exists a singularity at $x=1$.
\end{example}
\begin{example}
Improper definite integral of a linear function divided by the product of a square root, logarithm, and rising factorial.  In this example we use equation (\ref{eq:thm1}) and set $b\to 1,c\to 1,v\to 1,a\to 1,m\to m+1$ then we form a second equation by replacing $m\to s$ and take their difference and simplify. Then we set $k\to -1,m\to 1/2, s\to -1/2$ and simplify in terms of the digamma function using equations [Wolfram Functions, \href{http://functions.wolfram.com/10.06.03.0072.01}{(1)}] and [Wolfram MathWorld, \href{https://mathworld.wolfram.com/HurwitzZetaFunction.html}{(20)}] then simplify.
\begin{multline}\label{eq:poch_sing_2}
\int_0^{\infty } \frac{1-x}{\sqrt{x} \log (x) (1+x)_{1+n}} \, dx\\
=\sum _{j=0}^n \frac{i (-1)^j (2+j) }{2 \sqrt{1+j} n!}\binom{n}{j} \left(\psi ^{(0)}\left(\frac{\pi -i \log (1+j)}{4 \pi }\right)-\psi
   ^{(0)}\left(\frac{3}{4}-\frac{i \log (1+j)}{4 \pi }\right)\right)
\end{multline}
where $n\geq 1$ and there exists a singularity at $x=1$.
\end{example}
\begin{example}
The improper integral involving a falling factorial, square root, and logarithm. This integrand is a complex function defined by the product of a rational expression involving a Pochhammer symbol, a square root, and a logarithmic term. The integrand has several singularities (poles) within the integration interval $[0,\infty)$; at $x=0$ from the $\frac{1}{\sqrt{x}}$ term, at $x=1$ from the $\frac{1}{\log(x)}$ term and at $x=2,3,..,n+1$ from the reciprocal Pochhammer term. These multiple singularities means the integral will require Cauchy principal value method for evaluation. In this example we use equation (\ref{eq:thm1}) and set $b\to 1,c\to -1,v\to 1,a\to 1,m\to m+1$ then we form a second equation by replacing $m\to s$ and take their difference and simplify. Then we set $k\to -1,m\to 1/2, s\to -1/2$ and simplify in terms of the digamma function using equations [Wolfram Functions, \href{http://functions.wolfram.com/10.06.03.0072.01}{(1)}] and [Wolfram MathWorld, \href{https://mathworld.wolfram.com/HurwitzZetaFunction.html}{(20)}] then simplify.
\begin{multline}\label{eq:poch_sing_1}
\int_0^{\infty } \frac{1-x}{ (1-x)_{1+n}} \, \frac{dx}{\sqrt{x} \log (x)}\\
=\frac{1 }{2  n!}\sum _{j=0}^n \frac{(-1)^j j }{ \sqrt{1+j} }\binom{n}{j} \left(\psi ^{(0)}\left(\frac{1}{2}+\frac{\log (1+j)}{4 \pi i}\right)-\psi ^{(0)}\left(1+\frac{\log (1+j)}{4 \pi  i}\right)\right)
\end{multline}
where $n\geq 1$ and there exists a singularity at $x=i$.
\end{example}
\begin{example}
In this example we look at the plot of the integrand in equation (\ref{eq:poch_sing_1}). We se the many singularities along the $x$-axis which lead to the Cauchy principal value method of evaluating the integral. The is of the function $ \frac{1-x}{\sqrt{x} \log (x) (1-x)_{1+n}}$ over the range $n\in[0,5]$.
\begin{figure}[H]
\includegraphics[scale=0.6]{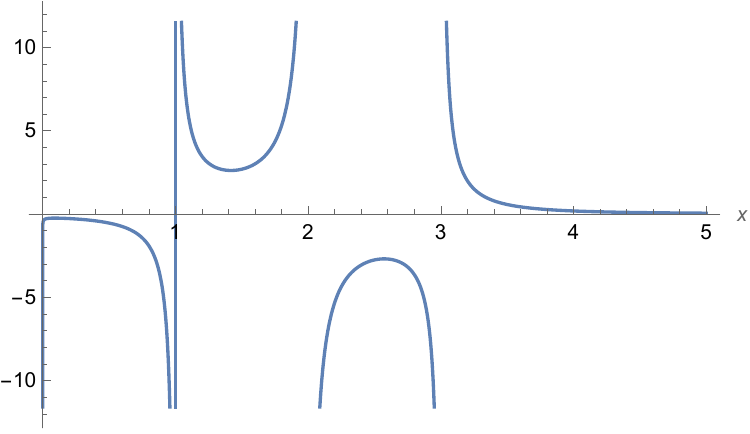}
\caption{Plot of the integrand in (\ref{eq:poch_sing_1})}
   \label{fig:fig2}
\end{figure}
\end{example}
\begin{example}
Definite integral involving iterated logarithm, square root, and Pochhammer symbol see \cite{malmsten}. In this example we use equation (\ref{eq:thm1}) and set $b\to 1,c\to -1,v\to 1,a\to 1,m\to m+1$ then we form a second equation by replacing $m\to s$ and take their difference and simplify. Then we take the first partial derivative with respect to $k$ and set $k\to -1,m\to 1/2, s\to -1/2$ and simplify.
\begin{multline}
\int_0^{\infty } \frac{(1-x) \log (\log (x))}{\sqrt{x} \log (x) (1-x)_{1+n}} \, dx\\
=-\sum _{j=0}^n \frac{e^{i j
   \pi } j \binom{n}{j} \left(\Phi \left(-1,1,\frac{\pi -i \log (-1-j)}{2 \pi }\right) (i \pi +2 \log (2 \pi ))-2
   \Phi'\left(-1,1,\frac{\pi -i \log (-1-j)}{2 \pi }\right)\right)}{2 \sqrt{1+j}
   n!}
\end{multline}
where $n\geq 1$ there exists a singularity at $x=i$.
\end{example}
\begin{example}
Extended Gr\"{o}bner integral form over integers. Table 324(11(b)) in \cite{grobner}. This is an example of why infinite integrals in terms finite series is useful. The Gr\"{o}bner form does not admit positive integers while the other form does allowing for iterative analysis of the definite integral. In this example we use equation (\ref{eq:thm2}) and set $m\to \frac{\alpha  m}{2},b\to 1,v\to \alpha ,k\to 2 n,n\to m-1,a\to 1$ and simplify;
\begin{multline}
\int_0^{\infty } \frac{x^{-1+\frac{m \alpha }{2}} \log ^{2 n}(x)}{\left(1+x^{\alpha }\right)^m} \, dx\\
=\begin{cases}
			\sum\limits_{j=0}^m \sum \limits_{l=0}^j \frac{(-1)^{-j} e^{\frac{i \pi  \left(-\left(\left(-\frac{1}{2}+m\right) \alpha\right)+\frac{m \alpha }{2}\right)}{\alpha }} \left(1-\frac{m}{2}\right)^{j-l} (2 \pi )^{1-l+2 n}\left(-\frac{1}{\alpha }\right)^l  }{\alpha  (-1+m)!}\\ \times
\left(\frac{i}{\alpha }\right)^{-l+2 n} \binom{j}{l}\Phi \left(e^{i m \pi },l-2
   n,\frac{1}{2}\right) (1-l+2 n)_l S_{-1+m}^{(j)}, &\text{$m\in\mathbb{Z_{+}}$}\\\\
            2(2n)!\left(\frac{2}{\alpha}\right)^{2n+1}\sum\limits_{v=0}^{\infty}\binom{-m}{v}\frac{1}{(m+2v)^{2n+1}}, & \text{otherwise}
		 \end{cases}
\end{multline}
where $Re(\alpha)>0,0 < m < 2n+2, n=0,1,2,..,$, and there exists a singularity at $x=1$.
\end{example}
\section{Schr\"{o}der's integral formula and Gregory coefficients}
In this section we study a special case of equation  (\ref{eq:thm2}) and evaluate definite integrals. We also derive a closed form solution for Schr\"{o}der's integral formula \cite{schroder} in terms of Gregory coefficients for $G_{n}$ where $n\geq 2$. 
\subsection{Ernst Schröder }
Ernst Schr\"{o}der (1841–1902), a German mathematician born in Mannheim, made contributions to various fields, including set theory, functional equations, finite differences, and algebraic logic.  His most famous work, the Schr\"{o}der–Bernstein theorem (1896) \cite{halmos}, established a cornerstone of modern set theory. It states that if there are injections between two sets in both directions, the sets have equal cardinality.  Schr\"{o}der achieved this using iterative chain arguments.

In 1870, he introduced the Schr\"{o}der functional equation \cite{schroder_fe} to study function iteration. This allows one to construct fractional and continuous iterates, which are precursors to dynamical systems and chaos theory.  That same year, saw the development of the Schröder integral formula \cite{schroder}, which expressed infinite integrals as series of forward differences at zero using Gregory coefficients and derived integral representations for functions related to the inverse logarithm and generating series, laying groundwork for the form involving $\ln^2 x + \pi^2$.

Other published work include; Vorlesungen über die Algebra der Logik (1890–1905) \cite{schroder_logic}, systematized logic as an algebraic structure. This work lead to the development of Boolean algebras and lattice theory.  His symbolic operator methods also laid early groundwork for umbral calculus.
\begin{example}
In this example we use equation  (\ref{eq:thm2}) and set $v=1$ and simplify;
\begin{multline}\label{eq:schroder}
\int_0^{\infty } \frac{x^{-1+m} \log ^k(a x)}{(1+b x)^{n+1}} \, dx\\
=-\sum _{j=1}^n \sum _{l=0}^j
   \frac{(-1)^{-j+n} b^{-m} e^{i (m-n) \pi } (m-n)^{j-l} (2 i \pi )^{1+k-l} \binom{j}{l} (1+k-l)_l }{n!}\\
\times \Phi \left(e^{2 i
   (m-n) \pi },-k+l,-\frac{i \left(i \pi +\log (a)+\log \left(\frac{1}{b}\right)\right)}{2 \pi }\right)
   S_n^{(j)}
\end{multline}
where $k,a,b\in \mathbb{C}, n\in \mathbb{Z_{+}},|Re(m)|<1$.
\end{example}
\begin{example}
Generalized Schr\"{o}der's integral formula. In this example we use equation (\ref{eq:schroder}) and set $k\to -1,a\to e^a,m\to 1$. Then we form two equation by replacing $a\to -a\pi i$ and $a\to a\pi i$ and take their difference and simplify the Huriwtz-Lerch zeta function using equations [\href{https://dlmf.nist.gov/25.14.E2}{25.14.2}] and [\href{https://dlmf.nist.gov/25.12.E13}{25.12.13}] and simplify;
\begin{multline}\label{eq:schroder1}
\int_0^{\infty } \frac{1}{(1+b x)^{n+1} \left(a^2 \pi ^2+\log ^2(x)\right)} \, dx\\
=\frac{i}{a b n!} \sum _{j=0}^n
   \sum _{l=0}^j  e^{-\frac{1}{2} i (2 j+l) \pi } \left(-(-1+n)^2\right)^{-l} (2 \pi )^{-1-l} \binom{j}{l}\\
   \left(e^{i j \pi } (1-n)^l (-1+n)^j \psi ^{(l)}\left(\frac{\pi +a \pi -i \log (b)}{2 \pi }\right)\right. \\ \left.
-(1-n)^j (-1+n)^l
   \psi ^{(l)}\left(\frac{\pi +a \pi +i \log (b)}{2 \pi }\right)\right) S_n^{(j)}
\end{multline}
where $a,b\in \mathbb{C}, n\geq 0$.
\end{example}
\begin{example}
An integral and finite double series representation for Gregory coefficients pp. 284-287 in \cite{jordan}. In this example we use equation (\ref{eq:schroder1}) and set $b\to 1, a\to \pi$, multiply both sides by $(-1)^n$ and simplify in terms of the Bernoulli number using equation [Wolfram MathWorld, \href{https://mathworld.wolfram.com/RiemannZetaFunction.html}{(61)}];
\begin{multline}\label{eq:schroder2}
\int_0^{\infty } \frac{(-1)^n}{(1+x)^{n+1} \left(\log ^2(x)+\pi ^2\right)} \, dx\\
=\frac{(-1)^{n+\frac{3}{2}} }{2 \pi  n!}\sum _{j=1}^n \sum _{l=1}^j e^{-\frac{1}{2} i (2 j+3 l) \pi }
   \left((-1)^{j+l}-(-1)^j\right) (n-1)^{j-l} (2 \pi )^{-l}\\
    \binom{j}{l} (-1)^l l! \zeta (l+1) S_n^{(j)}\\
   =\frac{(-1)^n }{2 \Gamma(1+n)}\sum _{j=1}^n \sum _{l=0}^j \frac{\binom{j}{1+2 l}
   S_n^{(j)} B_{2+2 l}}{(1+l) (n-1)^{1-j+2 l}}\\
=\frac{1}{(n+1)!}\sum _{l=0}^{n+1} \frac{S_{n+1}^{(l)}}{l+1}
\end{multline}
where $n\geq 2$.
\end{example}
\begin{example}
Double finite sum with Bernoulli numbers and Stirling numbers of the first kind $S_{n+1}^{l}$ in terms of Gregory coefficients. In this example we take the right-hand side of equation (\ref{eq:schroder2}).
\begin{equation}\label{eq:gregory}
\frac{(-1)^n }{2 \Gamma(1+n)}\sum _{j=1}^n \sum _{l=0}^j \frac{\binom{j}{1+2 l}
   S_n^{(j)} B_{2+2 l}}{(1+l) (n-1)^{1-j+2 l}}=\frac{1}{(n+1)!}\sum _{l=0}^{n+1} \frac{S_{n+1}^{(l)}}{l+1}=G_{n+1}
\end{equation}
where $n>1$.
\end{example}
\begin{example}
In this example we look at the partial sum of the left-hand side of equation (\ref{eq:gregory}) to see the point where the series is converging. The series seems to be converging to $\approx 0.026$. A closed form solution for the sum to infinity will be studied in proceeding work.
\begin{figure}[H]
\includegraphics[scale=0.6]{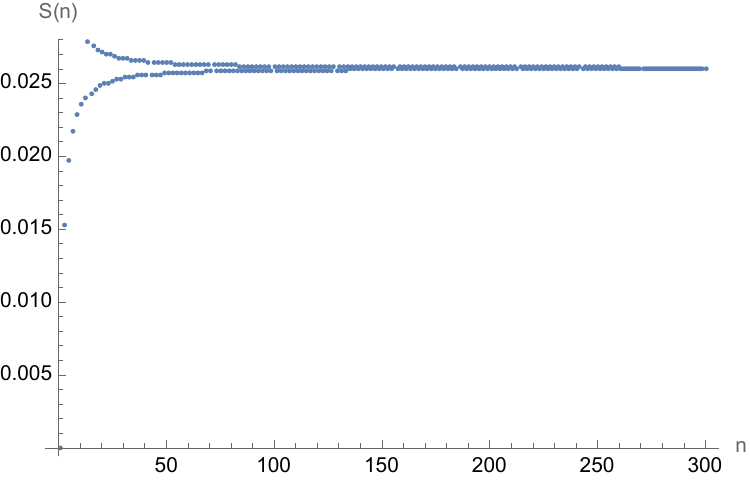}
\caption{Plot of the partial sums of equation (\ref{eq:gregory})}
   \label{fig:fig2}
\end{figure}
\end{example}
\begin{example}
Generalized logarithm Schröder's integral formula. In this example we use equation (\ref{eq:schroder}) and set $k\to -1,a\to e^a,m\to 1$. Then we form two equation by replacing $a\to -a\pi i$ and $a\to a\pi i$ and add these two equations and simplify the Huriwtz-Lerch zeta function using equations [\href{https://dlmf.nist.gov/25.14.E2}{25.14.2}] and [\href{https://dlmf.nist.gov/25.12.E13}{25.12.13}] and simplify;
\begin{multline}
\int_0^{\infty } \frac{\log (x)}{(1+b x)^{n+1} \left(a^2 \pi ^2+\log ^2(x)\right)} \, dx\\
=-\frac{1}{2b n!}\sum _{j=1}^n \sum
   _{l=0}^j 2^{-l} e^{-\frac{1}{2} i (2 j+l) \pi } \left(-(-1+n)^2\right)^{-l} \pi ^{-l} \binom{j}{l}\\
 \left(e^{i j \pi } (1-n)^l (-1+n)^j \psi ^{(l)}\left(\frac{\pi +a \pi -i \log (b)}{2 \pi }\right)\right. \\ \left.
+(1-n)^j (-1+n)^l
   \psi ^{(l)}\left(\frac{\pi +a \pi +i \log (b)}{2 \pi }\right)\right) S_n^{(j)}
\end{multline}
where $a,b\in \mathbb{C}, n\geq 0$.
\end{example}
\begin{example}
Mellin transform of a generalized Schr\"{o}der-type kernel. In this example we use equation (\ref{eq:schroder}) and set $k\to -1,a\to e^a$. Then we form two equation by replacing $a\to -a\pi i$ and $a\to a\pi i$ and take their difference and simplify the Huriwtz-Lerch zeta function using equations [\href{https://dlmf.nist.gov/25.14.E2}{25.14.2}] and [\href{https://dlmf.nist.gov/25.12.E13}{25.12.13}] and simplify;
\begin{multline}
\int_0^{\infty } \frac{x^{m-1}}{(1+b x)^{n+1} \left(\log ^2(x)+a^2 \pi ^2\right)} \, dx\\
=\frac{1}{a \pi  i n!}\sum _{j=0}^n \sum
   _{l=0}^j (-1)^{-j+l+n} 2^{-1-l} b^{-m} e^{i (m-n) \pi } (m-n)^{j-l} (i \pi )^{-l} \binom{j}{l} l! \\
\times\left(\Phi
   \left(e^{2 i (m-n) \pi },1+l,\frac{1}{2}-\frac{a}{2}-\frac{i \log \left(\frac{1}{b}\right)}{2 \pi }\right)\right. \\ \left.
-\Phi
   \left(e^{2 i (m-n) \pi },1+l,\frac{1}{2}+\frac{a}{2}-\frac{i \log \left(\frac{1}{b}\right)}{2 \pi }\right)\right)
   S_n^{(j)}
\end{multline}
where $|Re(a)|<1, |Re(m)|<1, Re(b)>0$.
\end{example}
\begin{example}
Infinite integrals with logarithmic difference kernels and polynomial structures. In this example we use equation (\ref{eq:thm1}) and set $k\to -1$. Next form a second equation by replacing $a\to 1/a$ and taking their difference and simplifying. 
\begin{multline}\label{eq:inf_int_1}
\int_0^{\infty } \frac{x^{-1+m}}{\left(\log ^2(a)-\log ^2(x)\right) \left(b+c x^v\right){}_{1+n}} \, dx\\
=\sum
   _{j=0}^n \frac{(-1)^j c^{-\frac{m}{v}} e^{\frac{i m \pi }{v}} (b+j)^{-1+\frac{m}{v}} \binom{n}{j} }{2 n! \log (a)}\\ \times
\left(\Phi
   \left(e^{\frac{2 i m \pi }{v}},1,\frac{\pi -i v \log \left(\frac{\left(\frac{b+j}{c}\right)^{1/v}}{a}\right)}{2 \pi
   }\right)-\Phi \left(e^{\frac{2 i m \pi }{v}},1,\frac{\pi -i v \log \left(a
   \left(\frac{b+j}{c}\right)^{1/v}\right)}{2 \pi }\right)\right)
\end{multline}
where $|Re(a)|\leq 1/2,Re(v)>0,Re(m)>0$ and there exists a singularity at $x=ai$.
\end{example}
\begin{example}
Mellin-Type integrals with logarithmic kernels and rational denominators. In this example we use equation (\ref{eq:inf_int_1}) and set $a\to ai$ and simplify.
\begin{multline}\label{eq:inf_int_2}
\int_0^{\infty } \frac{x^{m-1}}{\left(a^2+\log ^2(x)\right) \left(b+c x^v\right){}_{1+n}} \, dx\\
=-\sum _{j=0}^n \frac{(-1)^j c^{-\frac{m}{v}} e^{\frac{i m \pi }{v}} (b+j)^{-1+\frac{m}{v}} \binom{n}{j}
   }{2 n! a i}\\ \times
\left(\Phi \left(e^{\frac{2 i m \pi }{v}},1,\frac{\pi -i v \left(-a i+\log \left(\left(\frac{b+j}{c}\right)^{1/v}\right)\right)}{2 \pi }\right)\right. \\ \left.-\Phi \left(e^{\frac{2 i m \pi }{v}},1,\frac{\pi -i v \left(a
   i+\log \left(\left(\frac{b+j}{c}\right)^{1/v}\right)\right)}{2 \pi }\right)\right)
\end{multline}
where $|Re(a)|\leq 1/2,Re(m)>0,Re(v)>0$ and there exists a singularity at $x=i e^{-e^{i a}}$.
\end{example}
\begin{example}
Mellin-type integrals with squared logarithmic kernels and Pochhammer denominators. In this example we use equation (\ref{eq:inf_int_2}) and take the first partial derivative with respect to $a$ and simplify.
\begin{multline}
\int_0^{\infty } \frac{x^{m-1}}{\left(a^2+\log ^2(x)\right)^2 \left(b+c x^v\right){}_{1+n}} \, dx
=\sum _{j=0}^n
   \frac{i (-1)^j c^{-\frac{m}{v}} e^{\frac{i m \pi }{v}} (b+j)^{-1+\frac{m}{v}} \binom{n}{j} }{8 a^3 \pi  n!}\\ \times
\left(2 \pi  \Phi\left(e^{\frac{2 i m \pi }{v}},1,\frac{\pi -a v-i v \log \left(\left(\frac{b+j}{c}\right)^{1/v}\right)}{2 \pi}\right)\right. \\ \left.
-2 \pi  \Phi \left(e^{\frac{2 i m \pi }{v}},1,\frac{\pi +a v-i v \log\left(\left(\frac{b+j}{c}\right)^{1/v}\right)}{2 \pi }\right)\right. \\ \left.
-a v \left(\Phi \left(e^{\frac{2 i m \pi}{v}},2,\frac{\pi -a v-i v \log \left(\left(\frac{b+j}{c}\right)^{1/v}\right)}{2 \pi }\right)\right.\right. \\ \left.\left.
+\Phi \left(e^{\frac{2i m \pi }{v}},2,\frac{\pi +a v-i v \log \left(\left(\frac{b+j}{c}\right)^{1/v}\right)}{2 \pi}\right)\right)\right)
\end{multline}
where $|Re(a)|\leq 1,Re(v)>0,Re(m)>0$ and there exists a singularity at $x=i e^{-\frac{e^{i a}}{4}}$.
\end{example}
\begin{example}
 Integral representation involving logarithmic singularities and power terms. In this example we use equation (\ref{eq:inf_int_2}) and replace $a\to ai$ to form a second equation then take their difference and simplify. 
 \begin{multline}\label{eq:inf_int_3}
\int_0^{\infty } \frac{x^{-1+m}}{\left(a^4-\log ^4(x)\right) \left(b+c x^v\right){}_{1+n}} \, dx
=\sum _{j=0}^n
   \frac{(-1)^j c^{-\frac{m}{v}} e^{\frac{i m \pi }{v}} (b+j)^{-1+\frac{m}{v}} \binom{n}{j} }{4 a^3 n!}\\ \times
\left(\Phi\left(e^{\frac{2 i m \pi }{v}},1,\frac{\pi +i v \left(a-\log\left(\left(\frac{b+j}{c}\right)^{1/v}\right)\right)}{2 \pi }\right)\right. \\ \left.
+i \left(\Phi \left(e^{\frac{2 i m \pi}{v}},1,\frac{\pi -a v-i v \log \left(\left(\frac{b+j}{c}\right)^{1/v}\right)}{2 \pi }\right)\right.\right. \\ \left.\left.
-\Phi \left(e^{\frac{2i m \pi }{v}},1,\frac{\pi +a v-i v \log \left(\left(\frac{b+j}{c}\right)^{1/v}\right)}{2 \pi }\right)\right.\right. \\ \left.\left.
+i \Phi \left(e^{\frac{2 i m \pi }{v}},1,\frac{\pi -i v \left(a+\log\left(\left(\frac{b+j}{c}\right)^{1/v}\right)\right)}{2 \pi }\right)\right)\right)
\end{multline}
where $|Re(a)|\leq 1/2,Re(v)>0,Re(m)>0$ and there exists a singularity at $x=i e^{-e^{3 i a}}$.
\end{example}
\begin{example}
Generalized Mellin integrals with quartic logarithmic kernels and Pochhammer denominators. In this example we use equation (\ref{eq:inf_int_3}) and form a second equation by replacing $a\to a\sqrt{i}$ then take their difference and simplify.
\begin{multline}
\int_0^{\infty } \frac{x^{-1+m}}{\left(a^4+\log ^4(x)\right) \left(b+c x^v\right){}_{1+n}} \, dx
=\sum _{j=0}^n\frac{(-1)^{\frac{1}{4}+j} c^{-\frac{m}{v}} e^{\frac{i m \pi }{v}} (b+j)^{-1+\frac{m}{v}} \binom{n}{j} }{4 a^3 n!}\\ \times
\left(i \Phi\left(e^{\frac{2 i m \pi }{v}},1,\frac{\pi -\sqrt[4]{-1} a v-i v \log
   \left(\left(\frac{b+j}{c}\right)^{1/v}\right)}{2 \pi }\right)\right. \\ \left.
-i \Phi \left(e^{\frac{2 i m \pi }{v}},1,\frac{\pi+\sqrt[4]{-1} a v-i v \log \left(\left(\frac{b+j}{c}\right)^{1/v}\right)}{2 \pi }\right)\right. \\ \left.
+\Phi \left(e^{\frac{2 i m\pi }{v}},1,\frac{\pi +(-1)^{3/4} a v-i v \log \left(\left(\frac{b+j}{c}\right)^{1/v}\right)}{2 \pi }\right)\right. \\ \left.
-\Phi \left(e^{\frac{2 i m \pi }{v}},1,\frac{\pi -i v \left(\sqrt[4]{-1} a+\log
   \left(\left(\frac{b+j}{c}\right)^{1/v}\right)\right)}{2 \pi }\right)\right)
\end{multline}
where $|Re(a)|\leq 1/2,Re(v)>0,Re(m)>0$ and there exists a singularity at $x=i e^{-e^{3 i a}}$.
\end{example}
\begin{example}
 Improper integrals with logarithmic and polynomial structures: From zero to infinity. In this example we use equation (\ref{eq:inf_int_2}) and set $b\to 1,c\to 1,v\to 2 l,a\to \pi ,m\to 2 k+1$ and simplify. 
 \begin{multline}
\int_0^{\infty } \frac{x^{2 k}}{\left(\pi ^2+\log ^2(x)\right) \left(1+x^{2 l}\right){}_{1+n}} \, dx
=\sum_{j=0}^n \frac{i (-1)^j e^{\frac{i (1+2 k) \pi }{2 l}} (1+j)^{\frac{1+2 k-2 l}{2 l}} \binom{n}{j} }{2 \pi  n!}\\ \times
\left(\Phi
   \left(e^{\frac{i (1+2 k) \pi }{l}},1,\frac{1}{2}-l-\frac{i \log (1+j)}{2 \pi }\right)-\Phi \left(e^{\frac{i (1+2 k)
   \pi }{l}},1,\frac{1}{2}+l-\frac{i \log (1+j)}{2 \pi }\right)\right)\\
-\frac{1}{(n+2)!}
\end{multline}
where $k < 4l$ in order for the integral to converge.and there exists a singularity at $x=1$.
\end{example}
\begin{example}
Improper integrals with logarithmic singularities and polynomial growth. In this example we use equation (\ref{eq:inf_int_2}) and set $a\to \pi ,b\to 1,c\to 1,v\to 2 l,m\to 1$ and simplify. 
\begin{multline}
\int_0^{\infty } \frac{1}{\left(\pi ^2+\log ^2(x)\right) \left(1+x^{2 l}\right){}_{1+n}} \, dx
=\sum _{j=0}^n \frac{i (-1)^j e^{\frac{i \pi }{2 l}} (1+j)^{-1+\frac{1}{2 l}} \binom{n}{j} }{2 \pi  n!}\\ \times
\left(\Phi
   \left(e^{\frac{i \pi }{l}},1,\frac{\pi -2 i l \left(-i \pi +\log \left((1+j)^{\frac{1}{2 l}}\right)\right)}{2 \pi }\right)\right. \\ \left.
-\Phi \left(e^{\frac{i \pi }{l}},1,\frac{\pi -2 i l \left(i \pi +\log
   \left((1+j)^{\frac{1}{2 l}}\right)\right)}{2 \pi }\right)\right)-\frac{1}{(n+2)!}
\end{multline}
where $l=1,2,3,..$ and $n=0,1,2,3,..$ and there exists a singularity at $x=1$.
\end{example}
\begin{example}
 Improper integrals with logarithmic singularities and rational denominators. In this example we use equation (\ref{eq:inf_int_2}) and set $m\to 1,v\to 2,b\to 1,c\to 1$ and simplify.
 \begin{multline}\label{eq:inf_int_4}
\int_0^{\infty } \frac{1}{\left(a^2+\log ^2(x)\right) \left(1+x^2\right){}_{1+n}} \, dx
=\sum _{j=0}^n\frac{(-1)^j \binom{n}{j} }{4 a \sqrt{1+j} n!}\\ \times
\left(\psi ^{(0)}\left(\frac{-2 a+\pi -i \log (1+j)}{4 \pi }\right)-\psi
   ^{(0)}\left(\frac{2 a+\pi -i \log (1+j)}{4 \pi }\right)\right. \\ \left.
+\psi ^{(0)}\left(\frac{2 a+3 \pi -i \log (1+j)}{4 \pi
   }\right)-\psi ^{(0)}\left(-\frac{2 a-3 \pi +i \log (1+j)}{4 \pi }\right)\right)
\end{multline}
where $|Re(a)|\leq 1$ and there exists a singularity at $x=i e^{-e^{3 i a}}$.
\end{example}
\begin{example}
Schr\"{o}der-Quadratic-Pochhammer integral formula. In this example we use equation (\ref{eq:inf_int_2}) and set $a\to \pi$ and simplify. The inclusion of $1/(n+2)!$ comes after looking at the case when $n=0$. Then we take the difference of the left-hand side and right-hand side over $n$ and the result yielding $1/(n+2)!$. We then went back to the original formula equation (\ref{eq:inf_int_4}) and subtracted this term from the right-hand side to give the current result.
\begin{multline}\label{eq:int_inf_5}
\int_0^{\infty } \frac{1}{\left(\pi ^2+\log ^2(x)\right) \left(1+x^2\right){}_{1+n}} \, dx\\
=\sum _{j=0}^n \frac{(-1)^j \binom{n}{j} }{4 \sqrt{1+j} \pi 
   n!}\left(\psi ^{(0)}\left(\frac{-\pi -i \log (1+j)}{4 \pi }\right)-\psi
   ^{(0)}\left(\frac{3 \pi -i \log (1+j)}{4 \pi }\right)\right. \\ \left.
+\psi ^{(0)}\left(\frac{5 \pi -i \log (1+j)}{4 \pi }\right)-\psi ^{(0)}\left(-\frac{-\pi +i \log (1+j)}{4 \pi }\right)\right)-\frac{1}{(n+2)!}
\end{multline}
where there exists a singularity at $x=1$.
\end{example}
\subsection{Table of a few Schr\"{o}der-Pochhammer numbers}
In this short section we look at a few iterates of equation (\ref{eq:int_inf_5}) called the Schroder-Pochhammer numbers. These numbers are expressed in terms of the finite series involving the digamma function $\psi^{(0)}(z)$ see [Wolfram MathWorld, \href{https://mathworld.wolfram.com/PolygammaFunction.html}{(1)}].
\begin{example}
An improper integral with logarithmic and quadratic denominators. From equation (\ref{eq:int_inf_5}) when $n=0$.
\begin{multline}
\int_0^{\infty } \frac{1}{\left(1+x^2\right) \left(\pi ^2+\log ^2(x)\right)} \, dx=-\frac{1}{2}+\frac{\psi
   ^{(0)}\left(-\frac{1}{4}\right)-\psi ^{(0)}\left(\frac{1}{4}\right)-\psi ^{(0)}\left(\frac{3}{4}\right)+\psi
   ^{(0)}\left(\frac{5}{4}\right)}{4 \pi }
\end{multline}
\end{example}
\begin{example}
Improper integral with multiple rational factors and logarithmic singularities. From equation (\ref{eq:int_inf_5}) when $n=1$.
\begin{multline}
\int_0^{\infty } \frac{1}{\left(1+x^2\right) \left(2+x^2\right) \left(\pi ^2+\log ^2(x)\right)} \,
   dx\\
=-\frac{1}{6}+\frac{\psi ^{(0)}\left(-\frac{1}{4}\right)-\psi ^{(0)}\left(\frac{1}{4}\right)
-\psi^{(0)}\left(\frac{3}{4}\right)+\psi ^{(0)}\left(\frac{5}{4}\right)}{4 \pi }\\
-\frac{\psi ^{(0)}\left(\frac{-\pi -i
   \log (2)}{4 \pi }\right)-\psi ^{(0)}\left(\frac{3 \pi -i \log (2)}{4 \pi }\right)+\psi ^{(0)}\left(\frac{5 \pi -i
   \log (2)}{4 \pi }\right)-\psi ^{(0)}\left(-\frac{-\pi +i \log (2)}{4 \pi }\right)}{4 \sqrt{2} \pi }
\end{multline}
\end{example}
\begin{example}
Improper integral with multiple quadratic factors and logarithmic singularities. From equation (\ref{eq:int_inf_5}) when $n=2$.
\begin{multline}
\int_0^{\infty } \frac{1}{\left(1+x^2\right) \left(2+x^2\right) \left(3+x^2\right) \left(\pi ^2+\log
   ^2(x)\right)} \, dx\\
=-\frac{1}{24}+\frac{\psi ^{(0)}\left(-\frac{1}{4}\right)-\psi
   ^{(0)}\left(\frac{1}{4}\right)-\psi ^{(0)}\left(\frac{3}{4}\right)+\psi ^{(0)}\left(\frac{5}{4}\right)}{8 \pi
   }\\
-\frac{\psi ^{(0)}\left(\frac{-\pi -i \log (2)}{4 \pi }\right)-\psi ^{(0)}\left(\frac{3 \pi -i \log (2)}{4 \pi
   }\right)+\psi ^{(0)}\left(\frac{5 \pi -i \log (2)}{4 \pi }\right)-\psi ^{(0)}\left(-\frac{-\pi +i \log (2)}{4 \pi
   }\right)}{4 \sqrt{2} \pi }\\
+\frac{\psi ^{(0)}\left(\frac{-\pi -i \log (3)}{4 \pi }\right)-\psi ^{(0)}\left(\frac{3
   \pi -i \log (3)}{4 \pi }\right)+\psi ^{(0)}\left(\frac{5 \pi -i \log (3)}{4 \pi }\right)-\psi
   ^{(0)}\left(-\frac{-\pi +i \log (3)}{4 \pi }\right)}{8 \sqrt{3} \pi }
\end{multline}
\end{example}
\begin{example}
Multi-factor improper integral involving quadratic and logarithmic components. From equation (\ref{eq:int_inf_5}) when $n=3$.
\begin{multline}
\int_0^{\infty } \frac{1}{\left(1+x^2\right) \left(2+x^2\right) \left(3+x^2\right) \left(4+x^2\right) \left(\pi
   ^2+\log ^2(x)\right)} \, dx\\
=-\frac{1}{120}+\frac{\psi ^{(0)}\left(-\frac{1}{4}\right)-\psi
   ^{(0)}\left(\frac{1}{4}\right)-\psi ^{(0)}\left(\frac{3}{4}\right)+\psi ^{(0)}\left(\frac{5}{4}\right)}{24 \pi
   }\\
-\frac{\psi ^{(0)}\left(\frac{-\pi -i \log (2)}{4 \pi }\right)-\psi ^{(0)}\left(\frac{3 \pi -i \log (2)}{4 \pi
   }\right)+\psi ^{(0)}\left(\frac{5 \pi -i \log (2)}{4 \pi }\right)-\psi ^{(0)}\left(-\frac{-\pi +i \log (2)}{4 \pi
   }\right)}{8 \sqrt{2} \pi }\\
+\frac{\psi ^{(0)}\left(\frac{-\pi -i \log (3)}{4 \pi }\right)-\psi ^{(0)}\left(\frac{3
   \pi -i \log (3)}{4 \pi }\right)+\psi ^{(0)}\left(\frac{5 \pi -i \log (3)}{4 \pi }\right)-\psi
   ^{(0)}\left(-\frac{-\pi +i \log (3)}{4 \pi }\right)}{8 \sqrt{3} \pi }\\
-\frac{\psi ^{(0)}\left(\frac{-\pi -i \log
   (4)}{4 \pi }\right)-\psi ^{(0)}\left(\frac{3 \pi -i \log (4)}{4 \pi }\right)+\psi ^{(0)}\left(\frac{5 \pi -i \log
   (4)}{4 \pi }\right)-\psi ^{(0)}\left(-\frac{-\pi +i \log (4)}{4 \pi }\right)}{48 \pi }
\end{multline}
\end{example}
\begin{example}
Multi-factor improper integral combining quadratic and logarithmic components. From equation (\ref{eq:int_inf_5}) when $n=4$.
\begin{multline}
\int_0^{\infty } \frac{1}{\left(1+x^2\right) \left(2+x^2\right) \left(3+x^2\right) \left(4+x^2\right)
   \left(5+x^2\right) \left(\pi ^2+\log ^2(x)\right)} \, dx\\
=-\frac{1}{720}+\frac{\psi
   ^{(0)}\left(-\frac{1}{4}\right)-\psi ^{(0)}\left(\frac{1}{4}\right)-\psi ^{(0)}\left(\frac{3}{4}\right)+\psi
   ^{(0)}\left(\frac{5}{4}\right)}{96 \pi }\\
-\frac{\psi ^{(0)}\left(\frac{-\pi -i \log (2)}{4 \pi }\right)-\psi
   ^{(0)}\left(\frac{3 \pi -i \log (2)}{4 \pi }\right)+\psi ^{(0)}\left(\frac{5 \pi -i \log (2)}{4 \pi }\right)-\psi
   ^{(0)}\left(-\frac{-\pi +i \log (2)}{4 \pi }\right)}{24 \sqrt{2} \pi }\\
+\frac{\psi ^{(0)}\left(\frac{-\pi -i \log
   (3)}{4 \pi }\right)-\psi ^{(0)}\left(\frac{3 \pi -i \log (3)}{4 \pi }\right)+\psi ^{(0)}\left(\frac{5 \pi -i \log
   (3)}{4 \pi }\right)-\psi ^{(0)}\left(-\frac{-\pi +i \log (3)}{4 \pi }\right)}{16 \sqrt{3} \pi }\\
-\frac{\psi
   ^{(0)}\left(\frac{-\pi -i \log (4)}{4 \pi }\right)-\psi ^{(0)}\left(\frac{3 \pi -i \log (4)}{4 \pi }\right)+\psi
   ^{(0)}\left(\frac{5 \pi -i \log (4)}{4 \pi }\right)-\psi ^{(0)}\left(-\frac{-\pi +i \log (4)}{4 \pi }\right)}{48
   \pi }\\
+\frac{\psi ^{(0)}\left(\frac{-\pi -i \log (5)}{4 \pi }\right)-\psi ^{(0)}\left(\frac{3 \pi -i \log (5)}{4 \pi
   }\right)+\psi ^{(0)}\left(\frac{5 \pi -i \log (5)}{4 \pi }\right)-\psi ^{(0)}\left(-\frac{-\pi +i \log (5)}{4 \pi
   }\right)}{96 \sqrt{5} \pi }
\end{multline}
\end{example}
\subsection{The logarithmic Schr\"{o}der-Pochhammer numbers}
\begin{example}
Improper integral involving logarithmic functions and rational denominators. In this example we use equation (\ref{eq:inf_int_2}) and take the first partial derivative with respect to $m$ and set $m\to 1$ and simplify. In this derivation we looked at the first iterate case and yielded a difference of $\frac{\pi i}{(n+2)!}$ as the difference between the left-hand side and right-hand side. We include this term on the right-hand side to yield the correct integral formula.
\begin{multline}
\int_0^{\infty } \frac{\log (x)}{\left(\pi ^2+\log ^2(x)\right) \left(1+x^2\right){}_{1+n}} \, dx
=-\sum _{j=0}^n\frac{i (-1)^j \binom{n}{j} }{4 \sqrt{1+j} n!}\\
\left(\psi ^{(0)}\left(\frac{\pi -i \log (1+j)}{4 \pi }\right)-\psi^{(0)}\left(-\frac{\pi +i \log (1+j)}{4 \pi }\right)-\psi ^{(0)}\left(\frac{3}{4}-\frac{i \log (1+j)}{4 \pi}\right)\right. \\ \left.
+\psi ^{(0)}\left(\frac{5}{4}-\frac{i \log (1+j)}{4 \pi }\right)\right)-\frac{\pi i}{(n+2)!}
\end{multline}
where there exists a singularity at $x=1$.
\end{example}
\begin{example}
Integral representation involving logarithmic weights and polynomial decay. Here we use equation (\ref{eq:inf_int_1}) and take the first partial derivative with respect to $a,m$ and set $b\to 1, c\to 1, v\to 2$ and simplify
\begin{multline}
\int_0^{\infty } \frac{x^{m-1} \log (x)}{\left(a^2+\log ^2(x)\right)^2 \left(1+x^2\right){}_{1+n}} \, dx
=\sum_{j=0}^n \frac{(-1)^j e^{\frac{i m \pi }{2}} (1+j)^{-1+\frac{m}{2}} \binom{n}{j} }{4 a \pi  n!}\\ \times
\left(\Phi \left(e^{i m \pi
   },2,\frac{-2 a+\pi -i \log (1+j)}{2 \pi }\right)-\Phi \left(e^{i m \pi },2,\frac{2 a+\pi -i \log (1+j)}{2 \pi
   }\right)\right)
\end{multline}
where $|Re(a)|<1$ there exists a singularity at $x=i e^{-\frac{e^{i a}}{4}}$.
\end{example}
\begin{example}
A generalized definite integral involving logarithmic and rational functions. In this example we use equation (\ref{eq:inf_int_2}) and set $b\to 1,c\to 1,v\to 2,m\to m+1$. Next we form a second equation by replacing $m\to s$ and take their difference and simplify.
\begin{multline}\label{eq:mellin_schroder}
\int_0^{\infty } \frac{x^s-x^m}{\left(a^2+\log ^2(x)\right) \left(1+x^2\right){}_{1+n}} \, dx
=-\sum _{j=0}^n\frac{i (-1)^j \binom{n}{j} }{2 a \sqrt{1+j} n!}\\ \times
\left(e^{\frac{1}{2} i (1+m) \pi } (1+j)^{m/2} \left(\Phi \left(e^{i (1+m) \pi},1,\frac{-2 a+\pi -i \log (1+j)}{2 \pi }\right)\right.\right. \\ \left. \left.
-\Phi \left(e^{i (1+m) \pi },1,\frac{2 a+\pi -i \log (1+j)}{2 \pi}\right)\right)\right. \\ \left.
+e^{\frac{1}{2} i \pi  (1+s)} (1+j)^{s/2} \left(-\Phi \left(e^{i \pi  (1+s)},1,\frac{-2 a+\pi -i \log (1+j)}{2 \pi }\right)\right.\right. \\ \left.\left.
+\Phi \left(e^{i \pi  (1+s)},1,\frac{2 a+\pi -i \log (1+j)}{2 \pi}\right)\right)\right)
\end{multline}
where $|Re(a)|<\pi/2$ there exists a singularity at $x=e^{\frac{i \pi  a}{3}}$.
\end{example}
\begin{example}
A generalized definite integral involving logarithmic, rational, and power functions. In this example we use equation (\ref{eq:mellin_schroder}) and replace $s\to -m$ and simplify.
\begin{multline}\label{eq:mellin_schroder_1}
\int_0^{\infty } \frac{x^{-m}-x^m}{\left(a^2+\log ^2(x)\right) \left(1+x^2\right){}_{1+n}} \, dx=-\sum _{j=0}^n\frac{i (-1)^j \binom{n}{j} }{2 a \sqrt{1+j} n!}\\ \times
\left(e^{\frac{1}{2} i (1-m) \pi } (1+j)^{-\frac{m}{2}} \left(-\Phi \left(e^{i (1-m)
   \pi },1,\frac{-2 a+\pi -i \log (1+j)}{2 \pi }\right)\right.\right. \\ \left.\left.
+\Phi \left(e^{i (1-m) \pi },1,\frac{2 a+\pi -i \log (1+j)}{2\pi }\right)\right)\right.\\ \left.
+e^{\frac{1}{2} i (1+m) \pi } (1+j)^{m/2} \left(\Phi \left(e^{i (1+m) \pi },1,\frac{-2 a+\pi -i\log (1+j)}{2 \pi }\right)\right.\right. \\ \left.\left.
-\Phi \left(e^{i (1+m) \pi },1,\frac{2 a+\pi -i \log (1+j)}{2 \pi}\right)\right)\right)
\end{multline}
where $n\geq 1$ there exists a singularity at $x=1$.
\end{example}
\subsection{Definite integrals of logarithmic and fractional power functions involving Schr\"{o}der factors $\phi(n)$}\label{eq:method}
The method involves finding the additional factor called Schr\"{o}der factor $\phi(n)$, in terms of reciprocal factorial and other fundamental constants, when $a=\pi$. We use equation (\ref{eq:mellin_schroder_1}) and set $a\to \pi$ and evaluate the first few iterates of $n=1,..,6$. Then we analyze the difference of the left-hand side and right-sides and determine the fractional form for the difference. Then we formulate a recurrence identity for the pattern generated from the iterated process and add this factor to the right-hand side of the finite series. The integral formula are valid for $n\geq 1$ and are listed for $m=1/2,1/3,1/4,1/5,1/6$ respectively..
\begin{example}
In this example we use equation (\ref{eq:mellin_schroder_1}) with $m=1/2,\phi(n)=\frac{2i}{(n+2)!}$.
\begin{multline}
\int_0^{\infty } \frac{\frac{1}{\sqrt{x}}-\sqrt{x}}{\left(\pi ^2+\log ^2(x)\right) \left(1+x^2\right){}_{1+n}}
   \, dx
=-\sum _{j=0}^n \frac{i (-1)^j \binom{n}{j} }{2 \sqrt{1+j} \pi  n!}\\
\times
\left(e^{\frac{3 i \pi }{4}} \sqrt[4]{1+j} \left(\Phi\left(-i,1,\frac{-\pi -i \log (1+j)}{2 \pi }\right)-\Phi \left(-i,1,\frac{3 \pi -i \log (1+j)}{2 \pi}\right)\right)\right. \\ \left.
+\frac{e^{\frac{i \pi }{4}} \left(-\Phi \left(i,1,\frac{-\pi -i \log (1+j)}{2 \pi }\right)+\Phi\left(i,1,\frac{3 \pi -i \log (1+j)}{2 \pi }\right)\right)}{\sqrt[4]{1+j}}\right)+\frac{2i}{(n+2)!}
\end{multline}
where $n\geq 1$ there exists a singularity at $x=1$.
\end{example}
\begin{example}
In this example we use equation (\ref{eq:mellin_schroder_1}) with $m=1/3,\phi(n)=\frac{i\sqrt{3}}{(n+2)!}$.
\begin{multline}
\int_0^{\infty } \frac{\frac{1}{\sqrt[3]{x}}-\sqrt[3]{x}}{\left(\pi ^2+\log ^2(x)\right)
   \left(1+x^2\right){}_{1+n}} \, dx
=-\sum _{j=0}^n \frac{i (-1)^j \binom{n}{j} }{2 \sqrt{1+j} \pi  n!}\\ \times
\left(e^{\frac{2 i \pi }{3}}
   \sqrt[6]{1+j} \left(\Phi \left(e^{-\frac{1}{3} (2 i \pi )},1,\frac{-\pi -i \log (1+j)}{2 \pi }\right)-\Phi
   \left(e^{-\frac{1}{3} (2 i \pi )},1,\frac{3 \pi -i \log (1+j)}{2 \pi }\right)\right)\right. \\ \left.
+\frac{e^{\frac{i \pi }{3}}
   \left(-\Phi \left(e^{\frac{2 i \pi }{3}},1,\frac{-\pi -i \log (1+j)}{2 \pi }\right)+\Phi \left(e^{\frac{2 i \pi
   }{3}},1,\frac{3 \pi -i \log (1+j)}{2 \pi }\right)\right)}{\sqrt[6]{1+j}}\right)+\frac{i\sqrt{3}}{(n+2)!}
\end{multline}
where $n\geq 1$ there exists a singularity at $x=1$.
\end{example}
\begin{example}
In this example we use equation (\ref{eq:mellin_schroder_1}) with $m=1/4,\phi(n)=\frac{i
   \sqrt{2}}{(n+2)!}$.
\begin{multline}
\int_0^{\infty } \frac{\frac{1}{\sqrt[4]{x}}-\sqrt[4]{x}}{\left(\pi ^2+\log ^2(x)\right)
   \left(1+x^2\right){}_{1+n}} \, dx=-\sum _{j=0}^n \frac{i (-1)^j \binom{n}{j} }{2 \sqrt{1+j} \pi  n!}\\ \times
\left(e^{\frac{5 i \pi }{8}}
   \sqrt[8]{1+j} \left(\Phi \left(e^{-\frac{1}{4} (3 i \pi )},1,\frac{-\pi -i \log (1+j)}{2 \pi }\right)-\Phi
   \left(e^{-\frac{1}{4} (3 i \pi )},1,\frac{3 \pi -i \log (1+j)}{2 \pi }\right)\right)\right. \\ \left.
+\frac{e^{\frac{3 i \pi }{8}}
   \left(-\Phi \left(e^{\frac{3 i \pi }{4}},1,\frac{-\pi -i \log (1+j)}{2 \pi }\right)+\Phi \left(e^{\frac{3 i \pi
   }{4}},1,\frac{3 \pi -i \log (1+j)}{2 \pi }\right)\right)}{\sqrt[8]{1+j}}\right)+\frac{i
   \sqrt{2}}{(n+2)!}
\end{multline}
where $n\geq 1$ there exists a singularity at $x=1$.
\end{example}
\begin{example}
In this example we use equation (\ref{eq:mellin_schroder_1}) with $m=1/5,\phi(n)=\frac{i
   \sqrt{\frac{1}{2} \left(5-\sqrt{5}\right)}}{(2+n)!}$ which is closely related to the Golden ratio number [Wolfram MathWorld, \href{https://mathworld.wolfram.com/GoldenRatio.html}{(6)}].
\begin{multline}
\int_0^{\infty } \frac{\frac{1}{\sqrt[5]{x}}-\sqrt[5]{x}}{\left(\pi ^2+\log ^2(x)\right)
   \left(1+x^2\right){}_{1+n}} \, dx=-\sum _{j=0}^n \frac{i (-1)^j \binom{n}{j} }{2 \sqrt{1+j} \pi  n!}\\ \times
\left(e^{\frac{3 i \pi }{5}}
   \sqrt[10]{1+j} \left(\Phi \left(e^{-\frac{1}{5} (4 i \pi )},1,\frac{-\pi -i \log (1+j)}{2 \pi }\right)-\Phi
   \left(e^{-\frac{1}{5} (4 i \pi )},1,\frac{3 \pi -i \log (1+j)}{2 \pi }\right)\right)\right. \\ \left.
+\frac{e^{\frac{2 i \pi }{5}}
   \left(-\Phi \left(e^{\frac{4 i \pi }{5}},1,\frac{-\pi -i \log (1+j)}{2 \pi }\right)+\Phi \left(e^{\frac{4 i \pi
   }{5}},1,\frac{3 \pi -i \log (1+j)}{2 \pi }\right)\right)}{\sqrt[10]{1+j}}\right)+\frac{i
   \sqrt{\frac{1}{2} \left(5-\sqrt{5}\right)}}{(2+n)!}
\end{multline}
where $n\geq 1$ there exists a singularity at $x=1$.
\end{example}
\begin{example}
In this example we use equation (\ref{eq:mellin_schroder_1}) with $m=1/6,\phi(n)=\frac{i}{(n+2)!}$
\begin{multline}
\int_0^{\infty } \frac{\frac{1}{\sqrt[6]{x}}-\sqrt[6]{x}}{\left(\pi ^2+\log ^2(x)\right)
   \left(1+x^2\right){}_{1+n}} \, dx=-\sum _{j=0}^n \frac{i (-1)^j \binom{n}{j} }{2 \sqrt{1+j} \pi 
   n!}\\ \times
\left(e^{\frac{7 i \pi }{12}}
   \sqrt[12]{1+j} \left(\Phi \left(e^{-\frac{1}{6} (5 i \pi )},1,\frac{-\pi -i \log (1+j)}{2 \pi }\right)-\Phi
   \left(e^{-\frac{1}{6} (5 i \pi )},1,\frac{3 \pi -i \log (1+j)}{2 \pi }\right)\right)\right. \\ \left.
+\frac{e^{\frac{5 i \pi }{12}}
   \left(-\Phi \left(e^{\frac{5 i \pi }{6}},1,\frac{-\pi -i \log (1+j)}{2 \pi }\right)+\Phi \left(e^{\frac{5 i \pi
   }{6}},1,\frac{3 \pi -i \log (1+j)}{2 \pi }\right)\right)}{\sqrt[12]{1+j}}\right)+\frac{i}{(n+2)!}
\end{multline}
where $n\geq 1$ there exists a singularity at $x=1$.
\end{example}
\begin{example}
Here we use equation (\ref{eq:inf_int_1}) and take the first partial derivative with respect to $m$ and set $m\to m+1$. Next we replace $m\to -m,a\to \pi$ and simplify. Here we use the method in Section (\ref{eq:method}) where $m=1/2, \phi(n)=\frac{2\pi}{(n+2)!}$.
\begin{multline}\label{eq:schroder_log_1}
\int_0^{\infty } \frac{\left(-\frac{1}{\sqrt{x}}+\sqrt{x}\right) \log (x)}{\left(\pi ^2+\log ^2(x)\right)
   \left(1+x^2\right){}_{1+n}} \, dx=\sum _{j=0}^n \frac{(-1)^j e^{\frac{i \pi }{4}} \binom{n}{j}}{2 (1+j)^{3/4} n!}\\ \times
 \left(-i \sqrt{1+j}\left(\Phi \left(-i,1,-\frac{\pi +i \log (1+j)}{2 \pi }\right)+\Phi \left(-i,1,\frac{3}{2}-\frac{i \log (1+j)}{2\pi }\right)\right)\right. \\ \left.
+\Phi \left(i,1,-\frac{\pi +i \log (1+j)}{2 \pi }\right)+\Phi \left(i,1,\frac{3}{2}-\frac{i \log
   (1+j)}{2 \pi }\right)\right)+\frac{2 \pi }{(2+n)!}
\end{multline}
where $n\geq 1$ there exists a singularity at $x=1$.
\end{example}
%
%
\begin{example}
Integral involving a log-squared term and rational function with integer parameter $m$. Here we use the method in Section (\ref{eq:method}) where the Schr\"{o}der factor is $\phi(n)=-\frac{(-1)^{m+1}}{(n+2)!}$. In this example we use equation (\ref{eq:inf_int_2})  and set $b\to 1,c\to 1,v\to 2 m,a\to \pi$ and simplify.
\begin{multline}
\int_0^{\infty } \frac{x^{m-1}}{\left(\pi ^2+\log ^2(x)\right) \left(1+x^{2 m}\right){}_{1+n}} \, dx
=\sum _{j=0}^n \frac{(-1)^j \binom{n}{j} }{4 \sqrt{1+j} \pi  n!}\\ \times
\left(\psi ^{(0)}\left(\frac{\pi -2 m \pi
   -i \log (1+j)}{4 \pi }\right)-\psi ^{(0)}\left(\frac{\pi +2 m \pi -i \log (1+j)}{4 \pi }\right)\right. \\ \left.
+\psi ^{(0)}\left(\frac{1}{2} \left(\frac{3}{2}+m-\frac{i \log (1+j)}{2 \pi
   }\right)\right)-\psi ^{(0)}\left(\frac{1}{4} \left(3-2 m-\frac{i \log (1+j)}{\pi }\right)\right)\right)-\frac{(-1)^{m+1}}{(n+2)!}
\end{multline}
where $m=1,2,3,..$ and $n=1,2,3,..$ and there exists a singularity at $x=1$.
\end{example}
\begin{example}
Definite integral involving a logarithmic function and a generalized Pochhammer symbol. In this example we use equation (\ref{eq:inf_int_1}) and set $m\to 1,a\to 2 \pi ,b\to 1,c\to 1$ then take the limit as $v\to 1$ using l'Hopital's rule and simplify in terms of the digamma function using [Wolfram MathWorld, \href{https://mathworld.wolfram.com/HurwitzZetaFunction.html}{(20)}]. Here we use the method in Section (\ref{eq:method}) where $\phi(n)=\frac{1}{2 (n+2)!}$.
\begin{multline}
\int_0^{\infty } \frac{1}{\left(4 \pi ^2+\log ^2(x)\right) (1+x)_{1+n}} \, dx\\
=\sum _{j=0}^n \frac{i (-1)^j}{4 \pi  n!}\binom{n}{j} \left(\psi ^{(0)}\left(-\frac{\pi +i \log (1+j)}{2 \pi }\right)-\psi ^{(0)}\left(\frac{3}{2}-\frac{i
   \log (1+j)}{2 \pi }\right)\right)\\
   +\frac{1}{2 (n+2)!}
\end{multline}
where $n\geq 1$ there exists a singularity at $x=1$.
\end{example}
\begin{example}
A definite integral involving logarithmic and rational functions. In this example we use equation (\ref{eq:inf_int_1}) and set $a\to 2 \pi ,b\to 1,c\to 1,v\to 2 q+1,m\to 1$ and simplify. Here we use the method in Section (\ref{eq:method}) where $\phi(n)=\frac{1}{2 (n+2)!}$.
\begin{multline}
\int_0^{\infty } \frac{1}{\left(4 \pi ^2+\log ^2(x)\right) \left(1+x^{1+2 q}\right){}_{1+n}} \, dx
=\sum_{j=0}^n \frac{i (-1)^j e^{\frac{i \pi }{1+2 q}} (1+j)^{-\frac{2 q}{1+2 q}} }{4 \pi  n!}\\ 
\binom{n}{j} \left(\Phi
   \left(e^{\frac{2 i \pi }{1+2 q}},1,-\frac{\pi +4 \pi  q+i \log (1+j)}{2 \pi }\right)-\Phi \left(e^{\frac{2 i \pi
   }{1+2 q}},1,\frac{3}{2}+2 q-\frac{i \log (1+j)}{2 \pi }\right)\right)\\
+\frac{1}{2 (n+2)!}
\end{multline}
where $n,q\in\mathbb{Z_{+}}$ and there exists a singularity at $x=1$.
\end{example}
\begin{example}
Definite integral involving a logarithmic term and a Pochhammer symbol. In this example we use equation (\ref{eq:inf_int_1}) and set $a\to i \pi ,b\to 1,c\to 1,v\to 1,m\to 1$ then take the limit as $v\to 1$ using l'Hopital's rule and simplify in terms of the digamma function using [Wolfram MathWorld, \href{https://mathworld.wolfram.com/HurwitzZetaFunction.html}{(20)}]. 
\begin{multline}\label{eq:schroder_negative_pi_1}
\int_0^{\infty } \frac{1}{\left(\pi ^2-\log ^2(x)\right) (1+x)_{1+n}} \, dx\\
=\sum _{j=0}^n \frac{(-1)^j}{2 \pi  n!}\binom{n}{j} \left(\psi ^{(0)}\left(\frac{(1+i) \pi -i \log (1+j)}{2 \pi }\right)-\psi ^{(0)}\left(-\frac{i ((1+i)\pi +\log (1+j))}{2 \pi }\right)\right)
\end{multline}
where $n\geq 1$ there exists a singularity at $x=i e^{\pi }$.
\end{example}
\begin{example}
In this example we look at the plot of the integrand for equation (\ref{eq:schroder_negative_pi_1}) and where it's singularities lie at $x=e^{-\pi}$ and $x=e^{\pi}$.
\begin{figure}[H]
\includegraphics[scale=0.6]{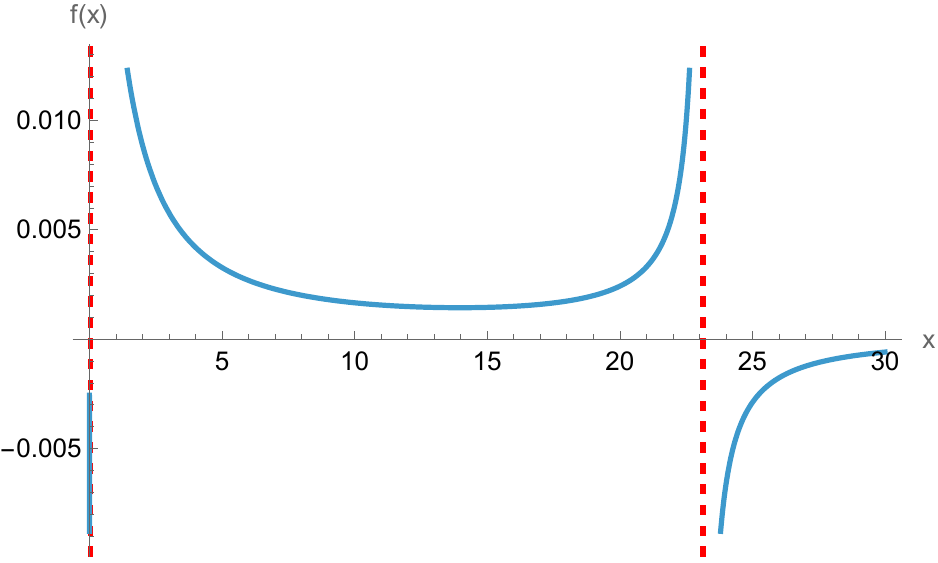}
\caption{ Plot of the $\frac{1}{\left(\pi ^2-\log ^2(x)\right) (x+1)_{n+1}}$}
   \label{fig:fig2}
\end{figure}
\end{example}
\begin{example}
Definite integral involving a logarithmic term and a Pochhammer symbol. In this example we use equation (\ref{eq:inf_int_1}) and take the first partial derivative with respect to $m$ and set $a\to 2 \pi ,b\to 1,c\to 1,v\to 2,m\to 1$ then take the limit as $v\to 1$ using l'Hopital's rule and simplify in terms of the digamma function using [Wolfram MathWorld, \href{https://mathworld.wolfram.com/HurwitzZetaFunction.html}{(20)}]. Here we use the method in Section (\ref{eq:method}) where $\phi(n)=\frac{\pi i}{ (n+2)!}$.
\begin{multline}
\int_0^{\infty } \frac{\log (x)}{\left(4 \pi ^2+\log ^2(x)\right) \left(1+x^2\right){}_{1+n}} \, dx
=\sum _{j=0}^n \frac{i (-1)^j }{4 \sqrt{1+j} n!}\\
\binom{n}{j} \left(-\psi ^{(0)}\left(-\frac{\pi +i
   \log (1+j)}{4 \pi }\right)+\psi ^{(0)}\left(-\frac{3}{4}-\frac{i \log (1+j)}{4 \pi }\right)\right. \\ \left.
+\psi ^{(0)}\left(\frac{5}{4}-\frac{i \log (1+j)}{4 \pi }\right)-\psi^{(0)}\left(\frac{7}{4}-\frac{i \log (1+j)}{4 \pi }\right)\right)+\frac{\pi  i}{(n+2)!}
\end{multline}
where $n\geq 1$.
\end{example}
\begin{example}
Definite integral involving a reciprocal logarithmic term and a Pochhammer symbol. In this example we use equation (\ref{eq:inf_int_1}) and set $b\to 1,c\to 1,v\to 1,m\to 1$ then simplify in terms of the digamma function using [Wolfram MathWorld, \href{https://mathworld.wolfram.com/HurwitzZetaFunction.html}{(20)}]. We then set $a\to 2 a p \pi$ and simplify. Here we use the method in Section (\ref{eq:method}) where $\phi(n)=\frac{\pi i}{ 2p(n+2)!}$.
\begin{multline}
\int_0^{\infty } \frac{1}{\left(4 p^2 \pi ^2+\log ^2(x)\right) (1+x)_{1+n}} \, dx\\
=\frac{1}{4 i \pi  p n!}\sum _{j=0}^n (-1)^j \binom{n}{j} \left(\psi ^{(0)}\left(\frac{1}{2}+p-\frac{i \log (1+j)}{2 \pi }\right)-\psi
   ^{(0)}\left(\frac{1}{2}-p-\frac{i \log (1+j)}{2 \pi }\right)\right)\\
+\frac{1}{2 p (n+2)!}
\end{multline}
where $n,p\geq 1$.
\end{example}
\begin{example}
The definite integral of the reciprocal of a squared logarithmic and generalized rational function product. In this example we use equation (\ref{eq:inf_int_1}) and set $b\to 1,c\to 1,v\to 1,m\to 1$ then we simplify the right-hand side in terms of the digamma function using equation [Wolfram MathWorld, \href{https://mathworld.wolfram.com/HurwitzZetaFunction.html}{(20)}]. Then we take the first partial derivative with respect to $a$ and simplify.
\begin{multline}\label{eq:alog_1}
\int_0^{\infty } \frac{1}{\left(a^2 \pi ^2+\log ^2(x)\right)^2 (1+x)_{1+n}} \, dx
=\sum _{j=0}^n \frac{i (-1)^j }{8 a^3 \pi ^3 n!}\\
\binom{n}{j} \left(2 \psi ^{(0)}\left(\frac{\pi -a \pi -i \log (1+j)}{2 \pi }\right)-2
   \psi ^{(0)}\left(\frac{\pi +a \pi -i \log (1+j)}{2 \pi }\right)\right. \\ \left.
+a \left(\psi ^{(1)}\left(\frac{\pi -a \pi -i \log (1+j)}{2 \pi }\right)+\psi ^{(1)}\left(\frac{\pi +a \pi -i \log (1+j)}{2 \pi}\right)\right)\right)
\end{multline}
where $|Re(a)|\leq1/3,n\geq 1$ and there exists a singularity at $x=i e^{-\frac{1}{3} i \pi  a}$
\end{example}
\begin{example}
The definite integral of the reciprocal of a cubed logarithmic and generalized rational function product. In this example we use equation (\ref{eq:alog_1}) and take the first partial derivative with respect to $a$ and simplify. 
\begin{multline}\label{eq:alog_2}
\int_0^{\infty } \frac{1}{\left(a^2 \pi ^2+\log ^2(x)\right)^3 (1+x)_{1+n}} \, dx\\
=\sum _{j=0}^n \frac{i (-1)^j}{64 a^5 \pi ^5 n!} \binom{n}{j} \left(12 \psi ^{(0)}\left(\frac{\pi -a \pi -i \log (1+j)}{2 \pi }\right)-12\psi ^{(0)}\left(\frac{\pi +a \pi -i \log (1+j)}{2 \pi }\right)\right. \\ \left.
+a \left(6 \psi ^{(1)}\left(\frac{\pi -a \pi -i \log (1+j)}{2 \pi }\right)+6 \psi ^{(1)}\left(\frac{\pi +a \pi -i \log (1+j)}{2 \pi}\right)\right.\right. \\ \left.\left.
+a \psi ^{(2)}\left(\frac{\pi -a \pi -i \log (1+j)}{2 \pi }\right)-a \psi ^{(2)}\left(\frac{\pi +a \pi -i \log (1+j)}{2 \pi }\right)\right)\right)
\end{multline}
where $|Re(a)|\leq1/3,n\geq 1$ and there exists a singularity at $x=i e^{-\frac{1}{3} i \pi  a}$.
\end{example}
\begin{example}
The definite integral of the reciprocal of a fourth-power logarithmic and generalized rational function product. In this example we use equation (\ref{eq:alog_2}) and take the first partial derivative with respect to $a$ and simplify. 
\begin{multline}
\int_0^{\infty } \frac{1}{\left(a^2 \pi ^2+\log ^2(x)\right)^4 (1+x)_{1+n}} \, dx\\
=\sum _{j=0}^n \frac{i (-1)^j }{768 a^7 \pi ^7 n!}\binom{n}{j} \left(120 \psi ^{(0)}\left(\frac{\pi -a \pi -i \log (1+j)}{2 \pi
   }\right)\right. \\ \left.
-120 \psi ^{(0)}\left(\frac{\pi +a \pi -i \log (1+j)}{2 \pi }\right)+a \left(60 \psi ^{(1)}\left(\frac{\pi -a \pi -i \log (1+j)}{2 \pi }\right)\right.\right. \\ \left.\left.
+60 \psi ^{(1)}\left(\frac{\pi +a \pi -i \log(1+j)}{2 \pi }\right)+a \left(12 \psi ^{(2)}\left(\frac{\pi -a \pi -i \log (1+j)}{2 \pi }\right)\right.\right.\right. \\ \left.\left.\left.
-12 \psi ^{(2)}\left(\frac{\pi +a \pi -i \log (1+j)}{2 \pi }\right)+a \left(\psi ^{(3)}\left(\frac{\pi -a\pi -i \log (1+j)}{2 \pi }\right)\right.\right.\right.\right. \\ \left.\left.\left.\left.
+\psi ^{(3)}\left(\frac{\pi +a \pi -i \log (1+j)}{2 \pi }\right)\right)\right)\right)\right)
\end{multline}
where $|Re(a)|\leq1/3,n\geq 1$ and there exists a singularity at $x=i e^{-\frac{1}{3} i \pi  a}$.
\end{example}
\begin{example}
Improper definite integral of the reciprocal of the product of a square root, a squared logarithmic term, and a rising factorial. In this example we use equation (\ref{eq:thm1}) and set $m\to \frac{1}{2},b\to 1,v\to 1,a\to e^a$ and simplify using equation [Wolfram Functions,\href{http://functions.wolfram.com/10.06.03.0074.01}{(1)}]. The we take the limit as $k\to -1$ using l'Hopital's rule and simplify in terms of the digamma function using equation [Wolfram Functions, \href{http://functions.wolfram.com/10.02.25.0001.01}{(1)}]. Next we form a second equation by replacing $a\to -a$ and take their difference and simplify.
\begin{multline}
\int_0^{\infty } \frac{1}{\sqrt{x} \left(a^2+\log ^2(x)\right) (1+c x)_{1+n}} \, dx\\=
\begin{cases}
			\text{sgn}(c)\sum\limits_{j=0}^n \frac{(-1)^j }{4 a \sqrt{c} \sqrt{1+j} n!}\binom{n}{j} \left(-\psi ^{(0)}\left(\frac{a+\pi -i \log \left(\frac{1+j}{c}\right)}{4 \pi}\right)+\psi ^{(0)}\left(\frac{a+3 \pi -i \log \left(\frac{1+j}{c}\right)}{4 \pi }\right)\right. \\ \left.
-\psi ^{(0)}\left(-\frac{a-3 \pi +i \log \left(\frac{1+j}{c}\right)}{4 \pi }\right)+\psi ^{(0)}\left(-\frac{a-\pi
   +i \log \left(\frac{1+j}{c}\right)}{4 \pi }\right)\right), & \text{if $c$ real}\\\\
            \sum\limits _{j=0}^n \frac{(-1)^j }{4 a \sqrt{c} \sqrt{1+j} n!}\binom{n}{j} \left(-\psi ^{(0)}\left(\frac{a+\pi -i \log \left(\frac{1+j}{c}\right)}{4 \pi}\right)+\psi ^{(0)}\left(\frac{a+3 \pi -i \log \left(\frac{1+j}{c}\right)}{4 \pi }\right)\right. \\ \left.
-\psi ^{(0)}\left(-\frac{a-3 \pi +i \log \left(\frac{1+j}{c}\right)}{4 \pi }\right)+\psi ^{(0)}\left(-\frac{a-\pi
   +i \log \left(\frac{1+j}{c}\right)}{4 \pi }\right)\right), & \text{if $c$ complex}
		 \end{cases}
\end{multline}
where $n\geq 1$ and there exists a singularity at $x=i e^{-i a}$.
\end{example}
\section{Ratio of polynomials}
In this section we extend the latter results by looking at infinite integrals involving the product of the ratio of generalized polynomials and generalized logarithm function. These integrals are expressed in terms of multiple finite series involving special functions.
\begin{example}
In this example we use equation (3.241.4(12)) in \cite{grad} with $p\to 1, n\to 0, q\to b$. Next we take the $n$-th partial derivative with respect to $b$ and simplify in terms of the Pochhammer symbol in equation [Wolfram MathWorld,\href{https://mathworld.wolfram.com/PochhammerSymbol.html}{(8)}]. Next we replace $u\to j (\beta -\alpha )+\alpha  p$ and multiply both sides by $(-1)^j \binom{p}{j}$ and take the finite sum of both sides over $j\in [0,p]$.Then we apply the contour integral method in \cite{reyn4} and equation (5) in \cite{reyn5} and simplify.
\begin{multline}\label{eq:gen_poly1}
\int_0^{\infty } \frac{\left(x^{\alpha }-x^{\beta }\right)^p }{\left(1+b x^v\right)^{n+1}}\log ^k(a x) \,
   dx\\
=\sum _{j=0}^n \sum _{l=0}^j \sum _{h=0}^p \frac{(-1)^{h-j} b^{-\frac{1+p \alpha +h (-\alpha +\beta )}{v}} \exp\left(\frac{i \pi  \left(1-\left(\frac{1}{2}+n\right) v+p \alpha +h (-\alpha +\beta )\right)}{v}\right) }{v n!}\\ \times
(2 \pi
   )^{1+k-l} \left(-\frac{1}{v}\right)^l \left(\frac{i}{v}\right)^{k-l} \left(1-\frac{1+p \alpha +h (-\alpha +\beta
   )}{v}\right)^{j-l} \binom{j}{l} \binom{p}{h}\\
 \Phi \left(e^{\frac{2 i \pi  (1-n v+p \alpha +h (-\alpha +\beta
   ))}{v}},-k+l,\frac{\pi -i v \log \left(a b^{-1/v}\right)}{2 \pi }\right) (1+k-l)_l S_n^{(j)}
\end{multline}
where $|Re(s)|<1,|Re(r)|<1,Re(v)>0$ and $Re(v) > p \times Re(\beta)$ in order for the integral to converge and there exists a singularity at $x=1/a$.
\end{example}
\begin{example}
Reciprocal logarithm infinite integral example. Similar integrals over [0,1] are given in section (4.212) in \cite{grad}. In this example we use equation (\ref{eq:gen_poly1}) and set $k\to -1, a\to e^{ai}$. Then we form a second equation by replacing $a\to -a$ and take their difference and simplify. Note the range of integration is given by $x\in[0,i e^{-a},e^a,\infty)$.
\begin{multline}
\int_0^{\infty } \frac{\left(x^s-x^r\right)^p}{\left(1+b x^v\right)^{n+1} \left(a^2-\log ^2(x)\right)} \,
   dx\\
=\sum _{j=0}^n \sum _{l=0}^j \sum _{h=0}^p \frac{(-1)^{h-j+l} 2^{-1-l} b^{-\frac{1+h r-h s+p s}{v}} \exp
   \left(\frac{i \pi  \left(1+h (r-s)+p s-\left(\frac{1}{2}+n\right) v\right)}{v}\right) \left(\frac{i}{\pi }\right)^l
   }{a i n!}\\ \times
\left(\frac{-1-p s+h (-r+s)+v}{v}\right)^{j-l} \binom{j}{l} \binom{p}{h} l! \\
\left(\Phi \left(e^{\frac{2 i \pi  (1+h
   (r-s)+p s-n v)}{v}},1+l,\frac{\pi -a v i+i \log (b)}{2 \pi }\right)\right. \\ \left.
-\Phi \left(e^{\frac{2 i \pi  (1+h (r-s)+p s-n
   v)}{v}},1+l,\frac{\pi +a v i+i \log (b)}{2 \pi }\right)\right) S_n^{(j)}
\end{multline}
where $Re(v)>Re(m)+Re(s), Re(a)\geq \pi/3$ and a singularities exists at $x=i e^{-a},x=e^a$.
\end{example}
\begin{example}
This is the logarithm version to equation (3.246) and (3.247.2) in \cite{grad}. In this example we use equation (\ref{eq:gen_poly1}) and set $\alpha \to m,\beta \to s,p\to n,n\to n-1,a\to 1,b\to -1$ and simplify using the Hurwitz-Lerch zeta function in terms of the Polylogarithm function given by $\Phi (z,-s,0)=\text{Li}_{-s}(z)$ where $Re(s)>0$.
 \begin{multline}\label{eq:poyl_n}
\int_0^{\infty } \left(\frac{x^m-x^s}{1-x^v}\right)^n \log ^k(x) \, dx\\
=\sum _{j=0}^{n-1} \sum _{l=0}^j \sum
   _{h=0}^n \frac{e^{i (h-j-n) \pi } (2 \pi )^{1+k-l} \left(-\frac{1}{v}\right)^l \left(\frac{i}{v}\right)^{1+k-l}
   \left(\frac{-1-m n+h (m-s)+v}{v}\right)^{j-l} \binom{j}{l} \binom{n}{h}}{\Gamma (n)}\\ \times
 \text{Li}_{-k+l}\left(e^{\frac{2 i \pi 
   (1+m n+h (-m+s)-n v)}{v}}\right) (1+k-l)_l S_{-1+n}^{(j)}
\end{multline}
where $Re(v)>Re(m)+Re(s)$ and a singularity exists at $x=i/2$.
\end{example}
\begin{example}
The Hurwitz zeta form. In this example we use equation (\ref{eq:poyl_n}) and simplify using equation (\ref{eq:prelim_1}). 
\begin{multline}\label{eq:poyl_n1}
\int_0^{\infty } \left(\frac{x^m-x^s}{1-x^v}\right)^n \log ^k(x) \, dx\\
=\sum _{j=0}^{n-1} \sum _{l=0}^j \sum
   _{h=0}^n \frac{i^{-k-l} e^{i (h-j-n) \pi } \left(-\frac{1}{v}\right)^l \left(\frac{i}{v}\right)^{k-l}
   \left(\frac{-1-m n+h (m-s)+v}{v}\right)^{j-l} \binom{j}{l} \binom{n}{h} \Gamma (1+k)}{v \Gamma (n)}\\ \times
 \left(i^{2 l} \zeta
   \left(1+k-l,n-\frac{1-h m+m n+h s}{v}\right)\right. \\ \left.
-i^{2 k} \zeta \left(1+k-l,\frac{1-h m+m n+h s+v-n v}{v}\right)\right)
   S_{-1+n}^{(j)}
\end{multline}
where $Re(v)>Re(m)+Re(s)$ and a singularity exists at $x=i/2$.
\end{example}
\begin{example}
 In this example we use equation (\ref{eq:poyl_n}) and simplify the Pochhammer expression using equation [Wolfram Research, \href{http://functions.wolfram.com/06.10.06.0001.01}{01}]. 
 \begin{multline}\label{eq:poyl_n2}
\int_0^{\infty } \left(\frac{x^m-x^s}{1-x^v}\right)^n \log ^k(x) \, dx\\
=\sum _{j=0}^{n-1} \sum _{l=0}^j \sum
   _{h=0}^n \sum _{p=0}^l \sum _{q=0}^p \frac{(-1)^{2 l-p} e^{i (h-j-n) \pi } \binom{p}{q} (k+1)^{p-q} (-l)^q (2 i \pi)^{1+k-l} v^{-1-j-k+l} }{\Gamma (n)}\\ \times
(-1-m n+h (m-s)+v)^{j-l} \binom{j}{l} \binom{n}{h} \text{Li}_{-k+l}\left(e^{\frac{2 i \pi (1+m n+h (-m+s)-n v)}{v}}\right) S_l^{(p)} S_{-1+n}^{(j)}
\end{multline}
where $Re(v)>Re(m)+Re(s)$ and a singularity exists at $x=i/2$.
\end{example}
\begin{example}
Generalized Malmsten-Erdeyli integral form. In this example we look at an integrand form which comprises the nested logarithm function studied by Malmsten \cite{malmsten} and the ratio of polynomials raised to an integer power with a singularity studied by Erdeyli see equation (3.246) in \cite{grad}. In this example we use equation (\ref{eq:poyl_n2}) and take the first partial derivative with respect to $k$ and set $k\to 0$ and simplify;
\begin{multline}
\int_0^{\infty } \left(\frac{x^m-x^s}{1-x^v}\right)^n \log (\log (x)) \, dx\\
=\sum _{j=0}^{n-1} \sum _{l=0}^j
   \sum _{h=0}^n \sum _{p=0}^l \sum _{q=0}^p \frac{i (-1)^{2 l-p} i^{-l} e^{i (h-j-n) \pi } (-l)^q (2 \pi )^{1-l}
   v^{-1-j+l} }{\Gamma(n)}\\ \times
(-1-m n+h (m-s)+v)^{j-l} \binom{j}{l} \binom{n}{h} \binom{p}{q} S_l^{(p)} S_{-1+n}^{(j)} \\
\left((p-q+\log(2 i \pi )-\log (v)) \text{Li}_l\left(e^{\frac{2 i \pi  (1+m n+h (-m+s)-n
   v)}{v}}\right)-\text{Li}_l'\left(e^{\frac{2 i \pi  (1+m n+h (-m+s)-n v)}{v}}\right)\right)
\end{multline}
where $Re(v)>Re(m)+Re(s)$ and a singularity exists at $x=i/2$.
\end{example}
%
%
%
\begin{example}
In this example we repeat the process in example (\ref{eq:gen_poly1}) with multiply both sides by $\binom{p}{j}$ and simplify;
\begin{multline}\label{eq:gen_poly2}
\int_0^{\infty } \frac{\left(x^{\alpha }+x^{\beta }\right)^p \log ^k(a x)}{\left(1+b x^v\right)^{n+1}} \,
   dx\\
=\sum _{j=0}^n \sum _{l=0}^j \sum _{h=0}^p \frac{(-1)^{-j} b^{-\frac{1+p \alpha +h (-\alpha +\beta )}{v}} \exp\left(\frac{i \pi  \left(1-\left(\frac{1}{2}+n\right) v+p \alpha +h (-\alpha +\beta )\right)}{v}\right) (2 \pi
   )^{1+k-l} \left(-\frac{1}{v}\right)^l \left(\frac{i}{v}\right)^{k-l} }{v n!}\\ \times
\left(1-\frac{1+p \alpha +h (-\alpha +\beta
   )}{v}\right)^{j-l} \binom{j}{l} \binom{p}{h}\\
 \Phi \left(e^{\frac{2 i \pi  (1-n v+p \alpha +h (-\alpha +\beta
   ))}{v}},-k+l,\frac{\pi -i v \log \left(a b^{-1/v}\right)}{2 \pi }\right) (1+k-l)_l S_n^{(j)}
\end{multline}
where $|Re(\alpha)|<1,|Re(\beta)|<1,Re(v)>0$.
\end{example}
\begin{example}
In this example we use equation (\ref{eq:gen_poly2}) and set $k\to -1$. When deriving the Ramanujan reciprocal logarithm integral form \cite{wood}. First set $k\to -1$ then apply the respective singularity points at $x=-ai, x=i/a$ when $a=1/a$ and $a=a$. Next take their difference and simplify. The singularity for the final form exists at $x=i/a$ where $Re(a)>0$. The final form can also be written in terms of the variable $a$ in the denominator by replacing $a \to e^a$ where there singularity exists at $x=e^{i/a}$.
\begin{multline}
\int_0^{\infty } \frac{\left(x^m+x^s\right)^p}{\left(1+b x^v\right)^n}\frac{\,
   dx}{ \left(a^2-\log ^2(x)\right)} \\
=\frac{i }{2 (n-1)! a}\sum _{j=0}^{n-1} \sum _{l=0}^j \sum _{h=0}^p (-1)^{-j} b^{-\frac{1-h m+m p+h s}{v}} e^{\frac{i \pi \left(1+m p+h (-m+s)+\frac{v}{2}-n v\right)}{v}} \left(-\frac{i}{2 \pi }\right)^l \\
\left(\frac{-1-m p+h
   (m-s)+v}{v}\right)^{j-l} \binom{j}{l} \binom{p}{h} l! \\
\left(\Phi \left(e^{\frac{2 i \pi  (1+m p+h (-m+s)+v-n
   v)}{v}},1+l,\frac{\pi +i v a+i \log (b)}{2 \pi }\right)\right. \\ \left.
-\Phi \left(e^{\frac{2 i \pi  (1+m p+h (-m+s)+v-n
   v)}{v}},1+l,\frac{\pi -i a v+i \log (b)}{2 \pi }\right)\right) S_{n-1}^{(j)}
\end{multline}
where $|Re(s)|<1,|Re(r)|<1,Re(v)>0$.
\end{example}
\begin{example}
 In this example we use equation (\ref{eq:gen_poly2}) and take the first partial derivative with respect to $k$ then set $k\to 0$ and simplify. Malmsten generalized infinite integral form see equations (14) and (47) in \cite{malmsten_specimen}. 
\begin{multline}\label{eq:malm_gen_poly}
\int_0^{\infty } \frac{\left(x^m+x^r\right)^p \log (\log (x))}{\left(1+b x^v\right)^n} \, dx\\
=\sum _{j=0}^n \sum
   _{l=0}^j \sum _{h=0}^p \sum _{q=0}^l \sum _{f=0}^q \frac{(-1)^{-j-q} 4^{-l} b^{-\frac{1-h m+m p+h r}{v}} \exp
   \left(\frac{i \pi  (2+2 m p+2 h (-m+r)+v+l v-2 n v)}{2 v}\right)  }{v \Gamma (n)}\\ \times
(-l)^f \pi ^{1-2 l}\left(\frac{-1-m p+h
   (m-r)+v}{v}\right)^{j-l} \binom{j}{l} \binom{p}{h} \binom{q}{f} S_l^{(q)} S_{-1+n}^{(j)}\\
 \left(e^{i l \pi } (2 \pi
   )^l \Phi \left(e^{\frac{2 i \pi  (1+m p+h (-m+r)+v-n v)}{v}},l,\frac{\pi +i \log (b)}{2 \pi }\right) \left(-2 f+i
   \pi +2 q+2 \log \left(\frac{2 \pi }{v}\right)\right)\right. \\ \left.
+(-2)^{1+l} \pi ^l
   \Phi'\left(e^{\frac{2 i \pi  (1+m p+h (-m+r)+v-n v)}{v}},l,\frac{\pi +i \log (b)}{2 \pi
   }\right)\right)
\end{multline}
where $|Re(m)|<1,|Re(r)|<1,Re(v)>0$.
\end{example}
\begin{example}
In this example we use equation (\ref{eq:malm_gen_poly}) and set $b\to 1, n\to 2, p\to 2$ and simplify equations (\ref{eq:prelim_3}), (\ref{eq:prelim_4}) and (\ref{eq:lerch_transform}).
\begin{multline}
\int_0^{\infty } \frac{\left(\sqrt{x}+x\right)^2 \log (\log (x))}{\left(1+x^4\right)^2} \, dx\\
=\log
   \left(\sqrt[8]{-1} 2^{-\frac{1}{16} \pi  \left(-9 \sqrt[8]{-1}+6 (-1)^{3/8}-9 (-1)^{5/8}+6 (-1)^{7/8}-\sqrt{2}+3
   (-1)^{5/8} \sqrt{2}+3 \sec \left(\frac{\pi }{8}\right)\right)}\right. \\ \left.
 \exp \left(\tanh
   ^{-1}\left((-1)^{5/8}\right)+\frac{1}{2} \tanh ^{-1}\left((-1)^{3/4}\right)+\frac{1}{32} i \pi ^2
   \left(2+\sqrt{2}+3 \sec \left(\frac{\pi }{8}\right)\right)\right)\right. \\ \left.
 \pi ^{\frac{1}{16} \pi  \left(2+\sqrt{2}+3 \sec
   \left(\frac{\pi }{8}\right)\right)} \left(\frac{3 \Gamma \left(-\frac{3}{4}\right)}{\Gamma
   \left(-\frac{1}{4}\right)}\right)^{-\frac{\pi }{4}} \left(\frac{\Gamma \left(\frac{9}{16}\right)}{\Gamma
   \left(\frac{1}{16}\right)}\right)^{\frac{3}{8} \sqrt[8]{-1} \pi } \left(\frac{\Gamma
   \left(\frac{5}{8}\right)}{\Gamma \left(\frac{1}{8}\right)}\right)^{\frac{1}{8} \sqrt[4]{-1} \pi }\right. \\ \left.
   \left(\frac{\Gamma \left(\frac{3}{16}\right)}{\Gamma \left(\frac{11}{16}\right)}\right)^{\frac{3}{8} (-1)^{3/8} \pi
   } \left(\frac{\Gamma \left(\frac{13}{16}\right)}{\Gamma \left(\frac{5}{16}\right)}\right)^{\frac{3}{8} (-1)^{5/8}
   \pi } \left(\frac{\Gamma \left(\frac{3}{8}\right)}{\Gamma \left(\frac{7}{8}\right)}\right)^{\frac{1}{8} (-1)^{3/4}
   \pi } \left(\frac{\Gamma \left(\frac{7}{16}\right)}{\Gamma \left(\frac{15}{16}\right)}\right)^{\frac{3}{8}
   (-1)^{7/8} \pi }\right)
\end{multline}
where a singularity exists at $x=1$.
\end{example}
\begin{example}
In this example we use equation (\ref{eq:malm_gen_poly}) and set $m\to \frac{1}{3},r\to \frac{1}{4},b\to 1,v\to \frac{3}{2}$ and simplify;
\begin{multline}
\int_0^{\infty } \frac{\left(\sqrt[4]{x}+\sqrt[3]{x}\right)^p \log (\log (x))}{\left(1+x^{3/2}\right)^n} \,
   dx\\
=\sum _{j=0}^n \sum _{l=0}^j \sum _{h=0}^p \sum _{q=0}^l \sum _{f=0}^q \frac{(-1)^{-j-q} 2^{1-j} 3^{-1-2 j+2 l}
   e^{-\frac{1}{18} i (-21+h-9 l+18 n-4 p) \pi } (-l)^f (6+h-4 p)^{j-l}}{\Gamma
   (n)}\\ \times
 \pi ^{1-l} \binom{j}{l} \binom{p}{h}
   \binom{q}{f} S_l^{(q)} S_{-1+n}^{(j)} \left(e^{i l \pi } \Phi \left(e^{-\frac{1}{9} i (-30+h+18 n-4 p) \pi
   },l,\frac{1}{2}\right)\right. \\ \left.
 \left(-2 f+i \pi +2 q+\log \left(\frac{16}{9}\right)+2 \log (\pi )\right)-2 (-1)^l
   \Phi'\left(e^{-\frac{1}{9} i (-30+h+18 n-4 p) \pi },l,\frac{1}{2}\right)\right)
\end{multline}
where a singularity exists at $x=1$.
\end{example}
\begin{example}
In this example we use equation (\ref{eq:malm_gen_poly}) and set $p\to n,n\to n-1,\alpha \to m,\beta \to s$ and simplify;
\begin{multline}\label{eq:malm_gen_poly_k}
\int_0^{\infty } \left(\frac{x^m+x^s}{1+b x^v}\right)^n \log ^k(a x) \, dx\\
=\sum _{j=0}^{n-1} \sum _{l=0}^j \sum
   _{h=0}^n \frac{(-1)^{-j} b^{-\frac{1+m n+h (-m+s)}{v}} \exp \left(\frac{i \pi  \left(1+m n+h
   (-m+s)-\left(-\frac{1}{2}+n\right) v\right)}{v}\right) (2 \pi )^{1+k-l} }{v (-1+n)!}\\ \times
\left(1-\frac{1+m n+h(-m+s)}{v}\right)^{j-l} \left(-\frac{1}{v}\right)^l \left(\frac{i}{v}\right)^{k-l} \binom{j}{l} \binom{n}{h}\\
 \Phi\left(\exp \left(\frac{2 i \pi  (1+m n+h (-m+s)-(-1+n) v)}{v}\right),-k+l,\frac{\pi -i v \log \left(a b^{-1/v}\right)}{2 \pi }\right)\\
 (1+k-l)_l S_{-1+n}^{(j)}
\end{multline}
where $Re(v)>Re(m)+Re(s)$ and a singularity exists at $x=1/a, Re(b)>0$ and a singularity at $x=i/a,Re(b)<0$.
\end{example}
\begin{example}
In this example we use equation (\ref{eq:malm_gen_poly_k}) and set $b\to -1, a\to 1$ and simplify using equation (\ref{eq:prelim_2});
\begin{multline}
\int_0^{\infty } \left(\frac{x^m+x^s}{1-x^v}\right)^n \log ^k(x) \, dx\\
=\sum _{j=0}^{n-1} \sum _{l=0}^j \sum
   _{h=0}^n \frac{e^{-i (j+n) \pi } (2 \pi )^{1+k-l} \left(-\frac{1}{v}\right)^l \left(\frac{i}{v}\right)^{1+k-l}
   \left(\frac{-1-m n+h (m-s)+v}{v}\right)^{j-l} \binom{j}{l} \binom{n}{h} }{(-1+n)!}\\ \times
\text{Li}_{-k+l}\left(e^{\frac{2 i \pi 
   (1+m n+h (-m+s)-n v)}{v}}\right) (1+k-l)_l S_{-1+n}^{(j)}
\end{multline}
where $Re(v)>Re(m)+Re(s)$ and a singularity exists at $x=i$.
\end{example}
\section{Products of special functions, logarithm and ratio of generalized polynomial functions}
In this section we derive and evaluate infinite integral forms using equations (\ref{eq:thm1}) and (\ref{eq:thm2}). The integrands of these infinite integral forms involve the products of special functions and the logarithm and rational functions. These integrals are expressed in terms of finite series involving special functions and fundamental constants. Special cases are derived to demonstrate how these generalized integrals can be evaluated. The reciprocal logarithm integrals are also listed in \cite{nahin} pages 91-92. Integrals of generalized rational functions raised to an integer power are studied in section (3.5), [Section 4, pp. 122] in \cite{nahin}. Page 184 studies the infinite integral of the reciprocal logarithm function.
\subsection{The Bessel function $J_{v}(b x)$ and finite reciprocal Pochhammer symbol}
The Bessel function contour integral representation is derived using equations (\ref{eq:mellin_poch}), [DLMF, \href{https://dlmf.nist.gov/10.2.E2}{10.2.2}], [Weisstein, Eric W. "Binomial Theorem." From MathWorld--A Wolfram Resource., \href{https://mathworld.wolfram.com/BinomialTheorem.html}{1}], and \cite{reyn4}. The contour integral representation used in this derivation is given by;
\begin{multline}\label{eq:bessel_contour}
\frac{1}{2\pi i}\int_{C}\int_{0}^{\infty}\frac{a^w w^{-k-1} x^w \left(x^s-x^r\right)^p J_{\nu }\left(x^{\tau } \mu \right)}{\left(b
   x^v+1\right){}_{n+1}}dxdw\\
=\frac{1}{2\pi i}\int_{C}\sum _{j=0}^n \sum _{h=0}^p \sum _{l=0}^{\infty }\frac{\pi  a^w w^{-k-1} 2^{-2 l-\nu } \binom{p}{h} \binom{n}{j} (-1)^{h+j+l} \mu ^{2 l+\nu }
   b^{-\frac{h (r-s)+\tau  (2 l+\nu )+p s+w+1}{v}}  }{v l! n! \Gamma (l+\nu +1)}\\ \times
(j+1)^{\frac{h (r-s)+\tau  (2 l+\nu )+p s+w+1}{v}-1}\csc
   \left(\frac{\pi  (h (r-s)+\tau  (2 l+\nu )+p s+w+1)}{v}\right)
\end{multline}
where $Re(v)>1,0< Re(s)<1,0< Re(r)<1 $. Next we apply \cite{reyn4} and equation (5) in \cite{reyn5} to equation (\ref{eq:bessel_contour}) in order to derive the Mellin transform of the product involving the Bessel function $J_{v}(b x)$ and the reciprocal Pochhammer symbol modulated by a logarithm to a power.
\begin{theorem}
A generalized Mellin-Bessel integral with logarithmic powers and rational weight. For all $|Re(s)|<1, |Re(r)|<1, Re(b)>0,  Re(v)>1, Re(\nu)>0, Re(\mu)>0, 0< Re(\tau) < 1$ then,
\begin{multline}
\int_0^{\infty } \frac{\left(x^s-x^r\right)^p J_{\nu }\left(x^{\tau } \mu \right) \log ^k(a x)}{\left(1+b
   x^v\right){}_{1+n}} \, dx\\
=\sum _{j=0}^n \sum _{h=0}^p \sum _{l=0}^{\infty } \frac{(-1)^{h+j+l} 2^{1+k-2 l-\nu }
    e^{\frac{i \pi  (1+h (r-s)+p s+2 l \tau +\nu  \tau )}{v}}
   (1+j)^{\frac{1+h r-h s+p s-v+2 l \tau +\nu  \tau }{v}} }{v^2 l! n! \Gamma (1+l+\nu
   ) b^{\frac{1+h r-h s+p s+2 l \tau +\nu  \tau }{v}}}\\ \times
\pi ^{1+k} \left(\frac{i}{v}\right)^{-1+k} \mu ^{2 l+\nu }
   \binom{n}{j} \binom{p}{h} \Phi \left(\exp \left(\frac{2 i \pi  (1+h (r-s)+p s+2 l \tau +\nu  \tau
   )}{v}\right),\right. \\ \left.
-k,\frac{\pi -i v \log \left(a b^{-1/v} (1+j)^{1/v}\right)}{2 \pi }\right)
\end{multline}
where there exists a singularity at $x=1$.
\end{theorem}
\subsection{The hypergeometric function $\, _2F_1\left(\alpha ,\beta ;\gamma;x^{\epsilon } \delta \right)$}
To derive the contour integral representation we use the method stated in Example 9.1 along with equation [DLMF, \href{https://dlmf.nist.gov/15.2.E1}{15.2.1}] and equation [Weisstein, Eric W. "Binomial Theorem." From MathWorld--A Wolfram Resource. ,\href{https://mathworld.wolfram.com/BinomialTheorem.html}{(1)}]. The hypergeometric contour integral representation is given by;
\begin{multline}\label{eq:hyper_contour}
\frac{1}{2\pi i}\int_{C}\int_{0}^{\infty}a^w w^{-1-k} x^w \left(-x^r+x^s\right)^p \left(1+b x^v\right)^{-1-n} \, _2F_1\left(\alpha ,\beta ;\gamma
   ;x^{\epsilon } \delta \right)dxdw\\
=\frac{1}{2\pi i}\int_{C}\sum _{j=0}^n \sum _{l=0}^j \sum _{h=0}^p \sum _{g=0}^{\infty }\frac{(-1)^{h-j} a^w b^{-\frac{1+h (r-s)+p s+w+g \epsilon }{v}} \pi 
   \left(-\frac{1}{v}\right)^l w^{-1-k+l} \delta ^g }{v g! n! (\gamma )_g}\\ \times
\left(-\frac{1+h (r-s)+p s+(-1+n) v-n v+g \epsilon
   }{v}\right)^{j-l} \binom{j}{l} \binom{p}{h} \\
\csc \left(\frac{\pi  (1+h (r-s)+p s-n v+w+g \epsilon )}{v}\right)
   (\alpha )_g (\beta )_g S_n^{(j)}dw
\end{multline}
where $0< Re(r)<1,0< Re(s)<1$.
Next we apply \cite{reyn4} and equation (5) in \cite{reyn5} to equation (\ref{eq:hyper_contour}) in order to derive the Mellin transform of the product involving the Bessel function $J_{v}(b x)$ and the reciprocal Pochhammer symbol modulated by a logarithm to a power. 
\begin{theorem}
A generalized Hypergeometric-Mellin integral with binomial and logarithmic powers.
\begin{multline}\label{eq:thm_hypergeom}
\int_0^{\infty } \frac{\left(x^s-x^r\right)^p \, _2F_1\left(\alpha ,\beta ;\gamma ;x^{\epsilon } \delta \right)
   \log ^k(a x)}{\left(1+b x^v\right)^{n+1}} \, dx\\
=-\sum _{j=0}^n \sum _{l=0}^j \sum _{h=0}^p \sum _{g=0}^{\infty }
   \frac{(1+k-l)_l (-1)^{h-j} b^{-\frac{1+h r-h s+p s+g \epsilon }{v}} e^{\frac{i \pi  (1+h (r-s)+p s-n v+g \epsilon)}{v}} (2 \pi )^{1+k-l} }{g! n! (\gamma )_g}\\ \times
\left(-\frac{1}{v}\right)^l \left(\frac{i}{v}\right)^{1+k-l} \delta ^g \left(-\frac{1+h r-hs+p s-v+g \epsilon }{v}\right)^{j-l} \binom{j}{l} \binom{p}{h}\\
 \Phi \left(e^{\frac{2 i \pi  (1+h (r-s)+p s-n v+g\epsilon )}{v}},-k+l,\frac{\pi -i v \log \left(a b^{-1/v}\right)}{2 \pi }\right) (\alpha )_g (\beta )_g
   S_n^{(j)}
\end{multline}
where $|Re(s)|<1,|Re(r)|<1,|Re(m)|<1, Re(\alpha)>0, Re(\beta)>0, Re(\gamma)>0,Re(\epsilon)>0,Re(\delta)>0, Re(v)>1$
\end{theorem}
\begin{example}
A generalized Mellin-type integral with Modified Bessel function and logarithmic powers. In this example we use equation (\ref{eq:thm_hypergeom}) and set $\alpha \to \frac{1}{2},\beta \to \frac{1}{2},\gamma \to 1$ and simplify using equation [DLMF, \href{https://dlmf.nist.gov/15.9.E24}{15.9.24}] and simplify;
\begin{multline}
\int_0^{\infty } \frac{\left(x^s-x^r\right)^p K\left(x^{\epsilon } \delta \right) \log ^k(a x)}{\left(1+b
   x^v\right)^{n+1}} \, dx\\
=\sum _{j=0}^n \sum _{l=0}^j \sum _{h=0}^p \sum _{g=0}^{\infty } -\frac{(-1)^{h-j} 2^{k-l}
   b^{-\frac{1+h r-h s+p s+g \epsilon }{v}} e^{\frac{i \pi  (1+h (r-s)+p s-n v+g \epsilon )}{v}} \pi ^{2+k-l}
   \left(-\frac{1}{v}\right)^l \left(\frac{i}{v}\right)^{1+k-l} \delta ^g }{g! n! (1)_g}\\ \times
\left(-\frac{1+h r-h s+p s-v+g \epsilon
   }{v}\right)^{j-l} \binom{j}{l} \binom{p}{h}\\
 \Phi \left(e^{\frac{2 i \pi  (1+h (r-s)+p s-n v+g \epsilon
   )}{v}},-k+l,\frac{\pi -i v \log \left(a b^{-1/v}\right)}{2 \pi }\right) \left(\left(\frac{1}{2}\right)_g\right){}^2
   (1+k-l)_l S_n^{(j)}
\end{multline}
\end{example}
\begin{example}
A Hypergeometric-Mellin integral with logarithmic resonance denominator. In this example we use equation (\ref{eq:thm_hypergeom}) and set $k\to -1$ and simplify. Here we are looking at the reciprocal logarithm form.
\begin{multline}
\int_0^{\infty } \frac{\left(-x^r+x^s\right)^p \left(1+b x^v\right)^{-1-n} \, _2F_1\left(\alpha ,\beta ;\gamma
   ;x^{\epsilon } \delta \right)}{\log ^2(a)-\log ^2(x)} \, dx\\
=\sum _{j=0}^n \sum _{l=0}^j \sum _{h=0}^p \sum
   _{g=0}^{\infty } \frac{(-1)^{h-j+l} 2^{-1-l} b^{-\frac{1+h r-h s+p s+g \epsilon }{v}} e^{\frac{i \pi  (1+h (r-s)+p
   s-n v+g \epsilon )}{v}} \pi ^{-l} \left(-\frac{1}{v}\right)^l \left(\frac{i}{v}\right)^{-l} \delta ^g
   }{g! n! \log (a) (\gamma )_g}\\ \times
\left(-\frac{1+h r-h s+p s-v+g \epsilon }{v}\right)^{j-l} \binom{j}{l} \binom{p}{h} l! \\
\left(\Phi \left(e^{\frac{2
   i \pi  (1+h (r-s)+p s-n v+g \epsilon )}{v}},1+l,\frac{\pi -i v \log \left(\frac{b^{-1/v}}{a}\right)}{2 \pi
   }\right)\right. \\ \left.
-\Phi \left(e^{\frac{2 i \pi  (1+h (r-s)+p s-n v+g \epsilon )}{v}},1+l,\frac{\pi -i v \log \left(a
   b^{-1/v}\right)}{2 \pi }\right)\right) (\alpha )_g (\beta )_g S_n^{(j)}
\end{multline}
where $Re(a)< e^{\pi i/2}$ and there exists a singularity at $x=ie^{-ia}$.
\end{example}
\section{Ratio of Pochhammer and $q$-Pochhammer symbols}
In this section we look at the Mellin transform of quotient pochammer and $q$-Pochhammer symbols modulated by logarithm to a power. These integrals are expressed in terms of series involving special functions. The equations used to derive the contour integral representations in this section are equation (\ref{eq:mellin_poch}), [Wolfram MathWorld, \href{https://mathworld.wolfram.com/PochhammerSymbol.html}{(8)}], and [Wolfram MathWorld, \href{https://mathworld.wolfram.com/BinomialSeries.html}{(1)}].
\subsection{Ratio of Pochhammer symbols}
Contour integral representation involving the ratio of Pochhammer symbols is given by; 
\begin{multline}\label{eq:poch_contour}
\frac{1}{2\pi i}\int_{C}\int_{0}^{\infty}\frac{a^w w^{-k-1} x^{m+w-1} \left(b x^u+1\right){}_{p+1}}{\left(c x^v+1\right){}_{n+1}}dxdw\\
=\frac{1}{2\pi i}\int_{C}\sum _{j=0}^n \sum _{q=0}^{p+1} \sum _{l=0}^q\frac{\pi  a^w b^l
   w^{-k-1} \binom{n}{j} \binom{q}{l} (-1)^{j+p-q+1} S_{p+1}^{(q)} c^{-\frac{l u+m+w}{v}} (j+1)^{\frac{l u+m+w}{v}-1}}{v n!}\\ \times
\csc \left(\frac{\pi  (l u+m+w)}{v}\right)dw
\end{multline}
where $Re(u)>1,Re(v)>1$.
The theorem involving the ratio of Pochhammer symbols is derived using \cite{reyn4}, equations (\ref{eq:poch_contour}) and equation (5) in \cite{reyn5}.
\begin{theorem}
A generalized power–rational integral with logarithmic weight. Finite Pochhammer form of Ramanujan's Integral [DLMF,\href{https://dlmf.nist.gov/17.13.E3}{17.13}].
\begin{multline}
\int_0^{\infty }x^{-1+m} \frac{ \left(1+b x^u\right){}_{1+p} }{\left(1+c x^v\right){}_{1+n}}\log ^k(a x) \, dx\\
=\sum _{j=0}^n \sum _{q=0}^{p+1} \sum _{l=0}^q \frac{(-1)^{j+p-q} b^l
   c^{-\frac{m+l u}{v}} e^{\frac{i \pi  (m+l u)}{v}} (1+j)^{-1+\frac{m+l u}{v}} (2 \pi )^{1+k} \left(\frac{i}{v}\right)^{1+k} \binom{n}{j} \binom{q}{l} }{n!}\\ \times
\Phi \left(e^{\frac{2 i \pi  (m+l
   u)}{v}},-k,\frac{\pi -i v \log \left(a c^{-1/v} (1+j)^{1/v}\right)}{2 \pi }\right) S_{1+p}^{(q)}
\end{multline}
where $|Re(m)| < 1$. There exists a singularity at $x=1/a$ when $Re(b)>0$.
\end{theorem}
\subsection{Extended Luo-Luo integrals in terms of series involving the Hurwitz-Lerch zeta function}
In this section we derive Mellin transforms of quotient finite $q$-Pochhammer symbols modulated by logarithm to a power. We start by deriving the reciprocal finite $q$-Pochhammer symbol using Corollary (18) in \cite{luo} given by;
\begin{equation}\label{eq:q_poch_reciprocal_series}
\frac{1}{(x;q)_n}=\sum _{m=0}^{n-1} \frac{(-1)^m q^{\frac{1}{2} m (m+1)}}{(q;q)_m (q;q)_{n-1-m}
   \left(1-x q^m\right)}=\sum _{k=0}^{\infty } \binom{n+k-1}{k}_q x^k
\end{equation}
where $0< Re(x)<1,0< Re(q)<1$. The contour integral representation for the reciprocal finite $q$-Pochhammer symbol is given by; 
\begin{multline}\label{eq:contour_p_poch}
\frac{1}{2\pi i}\int_{C}\int_{0}^{\infty}\frac{a^w w^{-k-1} x^{m+w-1}}{\left(b x^v;q\right){}_n}dxdw\\
=\frac{1}{2\pi i}\int_{C}\sum_{j=0}^{n-1}\frac{\pi  (-1)^j a^w q^{\frac{1}{2} j (j+1)}
   w^{-k-1} \csc \left(\frac{\pi  (m+w)}{v}\right) \left(-b q^j\right)^{-\frac{m+w}{v}}}{v (q;q)_j
   (q;q)_{-j+n-1}}dw
\end{multline}
where $Re(v)>0$. Next we apply \cite{reyn4} to equation (\ref{eq:contour_p_poch}) to derive the Mellin transform of the reciprocal finite $q$-Pochhammer symbol modulated by logarithm to a power given by;
\begin{theorem}
A generalized q-Mellin integral with logarithmic powers and q-Pochhammer denominator.
\begin{multline}\label{eq:q_poch_thm}
\int_0^{\infty } \frac{x^{-1+m} \log ^k(a x)}{\left(b x^v;q\right){}_n} \, dx\\
=\sum _{j=0}^{n-1} \frac{(-1)^j
   e^{\frac{i m \pi }{v}} (2 \pi )^{1+k} q^{\frac{1}{2} j (1+j)} \left(-b q^j\right)^{-\frac{m}{v}}
   \left(\frac{i}{v}\right)^{-1+k} \Phi \left(e^{\frac{2 i m \pi }{v}},-k,\frac{\pi -i v \log \left(a \left(-b
   q^j\right)^{-1/v}\right)}{2 \pi }\right)}{v^2 (q;q)_j (q;q)_{-1-j+n}}
\end{multline}
where $0< Re(m)<1,Re(b)<0$ and there exists a singularity at $x=1/a$.
\end{theorem}
\begin{example}
Logarithmic moments of the reciprocal linear kernel. In this example we use equation (\ref{eq:q_poch_thm}) and set $v\to 1,n\to 1$ and simplify;
\begin{multline}
\int_0^{\infty } \frac{x^{-1+m} \log ^k(a x)}{1-b x} \, dx=i^{-1+k} (-b)^{-m} e^{i m \pi } (2 \pi )^{1+k}
   \Phi \left(e^{2 i m \pi },-k,\frac{\pi -i \log \left(-\frac{a}{b}\right)}{2 \pi }\right)
\end{multline}
where $Re(m)>0$.
\end{example}
\begin{example}
A Mellin-ype integral with $q$-Pochhammer modulation and logarithmic singularity. In this example we use equation (\ref{eq:q_poch_thm}) and set $k\to -1$ and simplify;
\begin{multline}\label{eq:q_poch_thm_2}
\int_0^{\infty } \frac{x^{-1+m}}{(\log (a)+\log (x)) \left(b x^v;q\right){}_n} \, dx\\
=\sum _{j=0}^{n-1}
   -\frac{(-1)^j e^{\frac{i m \pi }{v}} q^{\frac{1}{2} j (1+j)} \left(-b q^j\right)^{-\frac{m}{v}} \Phi
   \left(e^{\frac{2 i m \pi }{v}},1,\frac{\pi -i v \log \left(a \left(-b q^j\right)^{-1/v}\right)}{2 \pi
   }\right)}{(q;q)_j (q;q)_{-1-j+n}}
\end{multline}
where $Re(m)>0$ and there exists a singularity at $x=a i$.
\end{example}
\begin{example}
Integral with quadratic logarithmic denominator and $q$-Pochhammer kernel. In this example we use equation (\ref{eq:q_poch_thm_2}) and form a second equation replacing $a\to 1/a$ and take their difference and simplify; 
\begin{multline}\label{eq:q_poch_thm_3}
\int_0^{\infty } \frac{x^{-1+m}}{\left(\log ^2(a)-\log ^2(x)\right) \left(b x^v;q\right){}_n} \, dx\\
=\sum
   _{j=0}^{n-1} -\frac{(-1)^j e^{\frac{i m \pi }{v}} q^{\frac{1}{2} j (1+j)} \left(-b q^j\right)^{-\frac{m}{v}}
   }{2 \log (a) (q;q)_j (q;q)_{-1-j+n}}\\ \times
\left(-\Phi \left(e^{\frac{2 i m \pi }{v}},1,\frac{\pi -i v \log \left(\frac{\left(-b
   q^j\right)^{-1/v}}{a}\right)}{2 \pi }\right)+\Phi \left(e^{\frac{2 i m \pi }{v}},1,\frac{\pi -i v \log \left(a
   \left(-b q^j\right)^{-1/v}\right)}{2 \pi }\right)\right)
\end{multline}
where $Re(m)>0$ and there exists a singularity at $x=i/a$.
\end{example}
\begin{example}
Log-quadratic denominator integral modulated by a $q$-Pochhammer factor. In this example we use equation (\ref{eq:q_poch_thm_3}) and set $a\to e^{a}$ and simplify;
\begin{multline}\label{eq:q_poch_thm_4}
\int_0^{\infty } \frac{x^{-1+m}}{\left(a^2-\log ^2(x)\right) \left(b x^v;q\right){}_n} \, dx
=-\sum
   _{j=0}^{n-1} \frac{(-1)^j e^{\frac{i m \pi }{v}} q^{\frac{1}{2} j (1+j)} \left(-b q^j\right)^{-\frac{m}{v}}
  }{2 a (q;q)_j (q;q)_{-1-j+n}}\\ \times
 \left(-\Phi \left(e^{\frac{2 i m \pi }{v}},1,\frac{\pi -i v \log \left(\frac{\left(-b q^j\right)^{-1/v}}{\exp
   (a)}\right)}{2 \pi }\right)\right. \\ \left.
+\Phi \left(e^{\frac{2 i m \pi }{v}},1,\frac{\pi -i v \log \left(\exp (a) \left(-b
   q^j\right)^{-1/v}\right)}{2 \pi }\right)\right)
\end{multline}
where $Re(m)>0$ and there exists a singularity at $x=a$.
\end{example}
\begin{example}
Evaluation of a log-quadratic denominator integral involving a $q$-shifted product. In this example we use equation (\ref{eq:q_poch_thm_4}) and set $a\to e^{ai}$ and simplify;
\begin{multline}
\int_0^{\infty } \frac{x^{-1+m}}{\left(a^2+\log ^2(x)\right) \left(b x^v;q\right){}_n} \, dx
=\sum
   _{j=0}^{n-1} \frac{i (-1)^j e^{\frac{i m \pi }{v}} q^{\frac{1}{2} j (1+j)} \left(-b q^j\right)^{-\frac{m}{v}}
  }{2 a (q;q)_j (q;q)_{-1-j+n}}\\ \times
 \left(\Phi \left(e^{\frac{2 i m \pi }{v}},1,\frac{\pi -i v \log \left(e^{-i a} \left(-b
   q^j\right)^{-1/v}\right)}{2 \pi }\right)\right. \\ \left.
-\Phi \left(e^{\frac{2 i m \pi }{v}},1,\frac{\pi -i v \log \left(e^{i a}
   \left(-b q^j\right)^{-1/v}\right)}{2 \pi }\right)\right)
\end{multline}
where $Re(m)>0$ and there exists a singularity at $x=ai$.
\end{example}
\begin{example}
Malmsten $q$-Pochhammer integral form. In this example we use equation (\ref{eq:q_poch_thm}) and take the first partial derivative with respect to $k$ then set $k\to 0$ and simplify;
\begin{multline}
\int_0^{\infty } \frac{x^{-1+m} \log (\log (x))}{(b x;q)_n} \, dx
=-\sum _{j=0}^{n-1} \frac{e^{i (j+m) \pi }
   \pi  q^{\frac{1}{2} j (1+j)} \left(-b q^j\right)^{-m} }{\left(-1+e^{2 i m \pi }\right) (q;q)_j (q;q)_{-1-j+n}}\\ \times
\left(\pi -2 i \log (2 \pi )+4 e^{i m \pi } \sin (m \pi )
   \Phi'\left(e^{2 i m \pi },0,\frac{\pi -i \log \left(-\frac{q^{-j}}{b}\right)}{2 \pi
   }\right)\right)
\end{multline}
where $Re(m)>0$ and there exists a singularity at $x=1$.
\end{example}
\begin{example}
A $q$-Pochhammer weighted integral of logarithmic moments. In this example we use equation (\ref{eq:q_poch_thm}) and set $a\to 1,b\to -1,v\to 1,q\to \frac{1}{q},m\to 1$ then set $k\to 0$ and simplify;
\begin{equation}\label{eq:q_poch_hurwitz}
\int_0^{\infty } \frac{\log ^k(x)}{\left(-x;\frac{1}{q}\right)_n} \, dx=\sum _{j=0}^{n-1} \frac{i (-1)^j i^k
   (2 \pi )^{1+k} q^{-\frac{1}{2} (-1+j) j} \zeta \left(-k,\frac{\pi -i j \log (q)}{2 \pi
   }\right)}{\left(\frac{1}{q};\frac{1}{q}\right)_j \left(\frac{1}{q};\frac{1}{q}\right)_{-1-j+n}}
\end{equation}
where $k\in\mathbb{C}$.
\end{example}
\begin{example}
Malmsten $q$-Pochhammer integral form. In this example we use equation (\ref{eq:q_poch_hurwitz}) and take the first partial derivative with respect to $k$ and set $k\to 0$ and simplify using equation [DLMF, \href{https://dlmf.nist.gov/25.11.E18}{25.11.18}];
\begin{multline}
\int_0^{\infty } \frac{\log (\log (x))}{\left(-x;\frac{1}{q}\right)_n} \, dx\\
=\frac{i}{2}\sum _{j=0}^{n-1} \frac{(-1)^j q^{-\frac{1}{2} (j-1) j} \left(4 \pi  \log \left(\frac{2 \sqrt{2} \pi ^{3/2}}{\Gamma \left(-\frac{\pi +i j \log (q)}{2 \pi }\right) (\pi +i j \log (q))}\right)+2 i j \log (2 \pi  i) \log (q)\right)}{ \left(\frac{1}{q};\frac{1}{q}\right)_j
   \left(\frac{1}{q};\frac{1}{q}\right)_{n-1-j}}
\end{multline}
where $Re(q)>0$ and there exists a singularity at $x=1$.
\end{example}
%
%
%
\subsection{Mellin transform of finite quotient $q$-Pochhammer symbols modulated by logarithm to a power}
In this section we study integrals involving $q$-Pochhammer symbols. The $q$-Pochhammer symbol appears in the theory of basic hypergeometric series, $q$-analogues of special functions, and partition theory. Understanding their Mellin transforms may lead to deeper analysis of their asymptotic behaviour, connections with $q$-gamma and $q$-beta functions, and potential applications in quantum calculus and modular phenomena. The inclusion of a logarithmic modulation introduces analytic properties, often tied to derivative-like behaviour in Mellin space and the study of logarithmic moments. Mellin transform of quotient $q$-Pochhammer symbols is listed in \cite{berndt_ram3} on page 29.\\\\
The contour integral representation of quotient $q$-Pochhammer symbols is given by;
\begin{multline}\label{eq:contour_q_poch_quotient}
\frac{1}{2\pi i}\int_{C}\int_{0}^{\infty}\frac{a^w w^{-k-1} x^{m+w-1} \left(b x^u;p\right){}_l}{\left(c x^v;q\right){}_n}dxdw\\
=\frac{1}{2\pi i}\int_{C}\sum_{j=0}^l \sum _{h=0}^n\frac{\pi  (-1)^h a^w
   (-b)^j q^{\frac{1}{2} h (h+1)} p^{\frac{1}{2} (j-1) j} w^{-k-1} \binom{l}{j}_p}{v (q;q)_h (q;q)_{-h+n-1}}\\ \times
 \csc \left(\frac{\pi  (j
   u+m+w)}{v}\right) \left(-c q^h\right)^{-\frac{j u+m+w}{v}}dw
\end{multline}
where $k,w\in\mathbb{C}$. It's derivation is yielded by first applying the series in equation [Wolfram MathWorld, \href{https://mathworld.wolfram.com/q-PochhammerSymbol.html}{(4)}] to equation (\ref{eq:q_poch_reciprocal_series}) followed by taking the Mellin transform, then applying the method in \cite{reyn4}.
\begin{theorem}
Mellin transform involving quotient finite $q$-Pochhammer symbols. This is a finite case of the generalized form for equation (5.16) in \cite{berndt_ram4} involving the logarithm function. Similar integrals were studied by Ramanujan in \cite{ramanujan_messenger}.
\begin{multline}
\int_0^{\infty } \frac{x^{-1+m} \log ^k(a x) \left(b x^u;p\right){}_l}{\left(c x^v;q\right){}_n} \, dx\\
=\sum_{j=0}^l \sum _{h=0}^n \frac{(-1)^h (-b)^j e^{\frac{i \pi  (m+j u)}{v}} p^{\frac{1}{2} (-1+j) j} (2 \pi )^{1+k}q^{\frac{1}{2} h (1+h)} \left(-c q^h\right)^{-\frac{m+j u}{v}} \left(\frac{i}{v}\right)^{-1+k} }{v^2 (q;q)_h (q;q)_{-1-h+n}}\\ \times
\Phi
   \left(e^{\frac{2 i \pi  (m+j u)}{v}},-k,\frac{\pi -i v \log \left(a \left(-c q^h\right)^{-1/v}\right)}{2 \pi
   }\right) \binom{l}{j}_p
\end{multline}
where $n > 2 l, 0< Re(m) \leq 1/l, 0< Re(q)<1, 0< Re(p)<1,Re(c) < 0$ and there exists a singularity at $x=1/a$ when $Re(a)>0$. The derivation of this theorem is due to the method in \cite{reyn4}, applied to equation (\ref{eq:contour_q_poch_quotient}) and equation (5) in \cite{reyn5}.
\end{theorem}
\subsection{Extended Ramanujan integral involving the $q$-Pochhammer symbol}
In this section we look at a special case of equation (5.16) in \cite{berndt_ram4} and apply the contour integral method in \cite{reyn4} to derive some definite integrals. The process of derivation involves setting $a=1$ in equation (5.16) in \cite{berndt_ram4}, then applying the method in \cite{reyn4}. Next we use cite{reyn4} and equation (5) in \cite{reyn5} to write down the Mellin transform.  These integrals involve the Mellin transform of the product of quotient $q$-pochhamer symbols modulated by a logarithm raised to a power and is expressed in terms of the Hurwitz-Lerch zeta function. We will evaluate this integrals in terms of composite functions of the Hurwitz-Lerch zeta function and fundamental constants.
\begin{theorem}
A q-analog Mellin integral with bilateral q-Pochhammer ratio and logarithmic powers.
\begin{multline}\label{eq:ram_q_poch_inf}
\int_0^{\infty } \frac{x^{-1+m} \log ^k(a x) (-q x;q)_{\infty }}{(-x;q)_{\infty }} \, dx=-i i^k e^{i m \pi } (2
   \pi )^{1+k} \Phi \left(e^{2 i m \pi },-k,\frac{\pi -i \log (a)}{2 \pi }\right)
\end{multline}
where $0< Re(q)<1$.
\end{theorem}
\begin{example}
In this example we use equation (\ref{eq:ram_q_poch_inf}) and set $m\to 1/2$ and simplify in terms of the Hurwitz zeta function using equation [Wolfram Research Functions, \href{http://functions.wolfram.com/10.06.03.0072.01}{1}].
\begin{multline}\label{eq:ram_q_poch_inf0}
\int_0^{\infty } \frac{\log ^k(a x) (-q x;q)_{\infty }}{\sqrt{x} (-x;q)_{\infty }} \, dx\\
=i^k (2 \pi )^{1+k} \left(-2^k \zeta \left(-k,\frac{1}{2} \left(1+\frac{\pi -i \log (a)}{2 \pi
   }\right)\right)+2^k \zeta \left(-k,\frac{\pi -i \log (a)}{4 \pi }\right)\right)
\end{multline}
where $0< Re(q)<1$ and there exists a singularity at $x=1/a$.
\end{example}
\begin{example}
In this example we use equation (\ref{eq:ram_q_poch_inf0}) and take the first partial derivative with respect to $k$ and set $k\to 0$ and simplify using equation [DLMF, \href{https://dlmf.nist.gov/25.11.E18}{25.11.18}]. 
\begin{equation}\label{eq:ram_q_poch_inf1}
\int_0^{\infty } \frac{\log (\log (a x)) (-q x;q)_{\infty }}{\sqrt{x} (-x;q)_{\infty }} \, dx=\pi  \log
   \left(\frac{4 \pi  i \Gamma \left(\frac{3 \pi -i \log (a)}{4 \pi }\right)^2}{\Gamma \left(\frac{\pi -i \log (a)}{4
   \pi }\right)^2}\right)
\end{equation}
where $0< Re(q)<1$ and there exists a singularity at $x=1$.
\end{example}
\begin{example}
A Malmsten integral form involving quotient $q$-Pochhammer symbols. In this example we use equation (\ref{eq:ram_q_poch_inf1}) and set $a\to 1,q\to 1/2$ and simplify. We also provide a few plots showing the singularity at $x=1$; 
 \begin{equation}\label{eq:ram_q_poch_inf2}
\int_0^{\infty }\frac{\log (\log (x))}{\sqrt{x}}  \frac{\left(-\frac{x}{2};\frac{1}{2}\right)_{\infty }}{ \left(-x;\frac{1}{2}\right)_{\infty }} \, dx=\pi  \log \left(\frac{4 i \pi  \Gamma
   \left(\frac{3}{4}\right)^2}{\Gamma \left(\frac{1}{4}\right)^2}\right)
\end{equation}
\end{example}
\begin{figure}[H]
\includegraphics[scale=0.6]{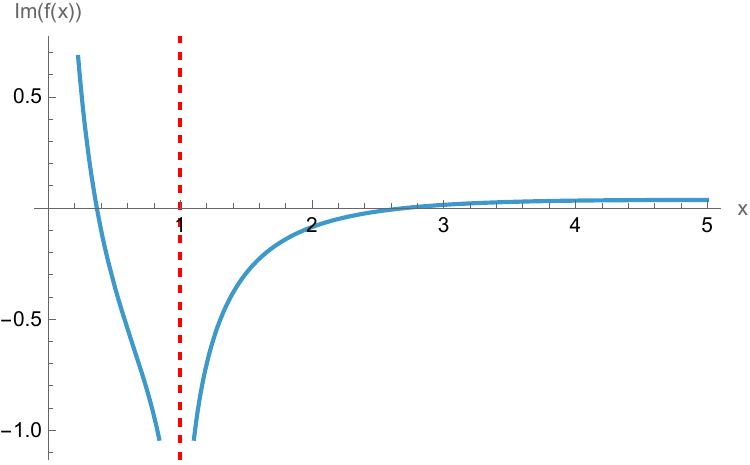}
\caption{ Plot of the $Re\left(\frac{\left(-\frac{x}{2};\frac{1}{2}\right)_{\infty } \log (\log (x))}{\sqrt{x} \left(-x;\frac{1}{2}\right)_{\infty }}\right)$}
   \label{fig:fig2}
\end{figure}
\begin{figure}[H]
\includegraphics[scale=0.6]{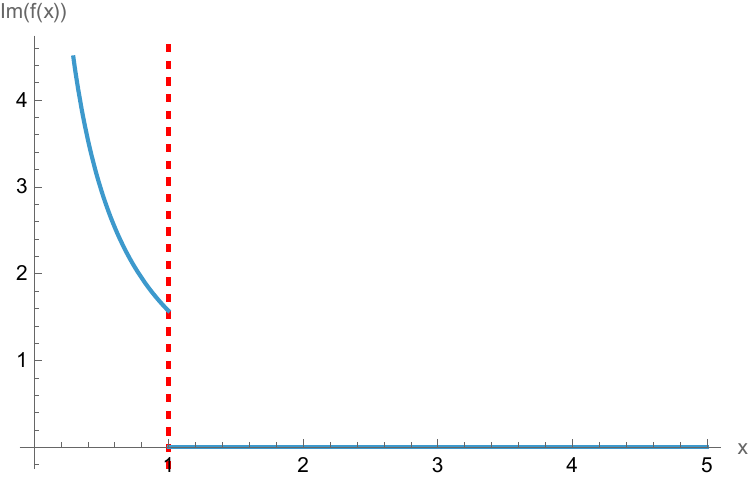}
\caption{ Plot of the $Im\left(\frac{\left(-\frac{x}{2};\frac{1}{2}\right)_{\infty } \log (\log (x))}{\sqrt{x} \left(-x;\frac{1}{2}\right)_{\infty }}\right)$}
   \label{fig:fig2}
\end{figure}
\begin{example}
In this example we use equation (\ref{eq:ram_q_poch_inf0}) and take the limit using l'Hopital's rule as $k\to -1$ and simplify using equation [Wolfram Research Functions, \href{http://functions.wolfram.com/10.02.25.0001.0}{1}].
\begin{equation}\label{eq:ram_q_poch_inf3}
\int_0^{\infty } \frac{(-q x;q)_{\infty }}{\sqrt{x} \left(a^2 \pi ^2+\log ^2(x)\right) (-x;q)_{\infty }} \,
   dx=\frac{\psi ^{(0)}\left(\frac{a+3}{4}\right)-\psi ^{(0)}\left(\frac{a+1}{4}\right)}{2 a \pi }
\end{equation}
where $0< Re(q)<1$ and there exists a singularity at $x=ie^a$.
\end{example}
\begin{example}
In this example we use equation (\ref{eq:ram_q_poch_inf}) and set $a\to -1$ and simplify using [DLMF, \href{https://dlmf.nist.gov/25.14.E2}{25.14.3}]. 
\begin{equation}
\int_0^{\infty } \frac{x^{-1+m} \log ^k(-x) (-q x;q)_{\infty }}{(-x;q)_{\infty }} \, dx=-i i^k e^{-i m \pi } (2 \pi )^{1+k} \text{Li}_{-k}\left(e^{2 i m \pi }\right)
\end{equation}
where $0< Re(q)<1$ and there exists a singularity at $x=1/a$.
\end{example}
\begin{example}
In this example we use equation (\ref{eq:ram_q_poch_inf}) and replace $m\to m+1$. Then we form a second equation by replacing $m\to s$ then take their difference and set $k\to -1, a\to 1$ and simplify;
\begin{equation}
\int_0^{\infty } \frac{\left(x^s-x^m\right) (-q x;q)_{\infty }}{\log (x) (-x;q)_{\infty }} \, dx=\log \left(\frac{\tan \left(\frac{\pi  m}{2}\right)}{\tan \left(\frac{s \pi
   }{2}\right)}\right)
\end{equation}
where $Re(m)<0,Re(s)<0,0< Re(q)<1$.
\end{example}
\begin{example}
In this example we use the method in equation [Wolfram MathWorld, \href{https://mathworld.wolfram.com/LerchTranscendent.html}{(8)}] and apply it to equation (\ref{eq:ram_q_poch_inf});
\begin{multline}
\int_0^{\infty } \frac{\log (-x) \log (\log (-x)) (-q x;q)_{\infty }}{\sqrt{x} (-x;q)_{\infty }} \, dx=-\frac{\pi ^3}{2}-4 i \pi ^2 \log \left(\frac{A^3}{\sqrt[3]{2} \sqrt[4]{e}}\right)+i \pi ^2 \log
   (2 \pi )
\end{multline}
where $0< Re(q)<1$.
\end{example}
\begin{example}
In this example we use the method in equation [Wolfram MathWorld, \href{https://mathworld.wolfram.com/LerchTranscendent.html}{(9)}] and apply it to equation (\ref{eq:ram_q_poch_inf});
\begin{equation}
\int_0^{\infty } \frac{\log ^2(-x) \log (\log (-x)) (-q x;q)_{\infty }}{\sqrt{x} (-x;q)_{\infty }} \, dx=14 \pi  \zeta (3)
\end{equation}
where $0< Re(q)<1$.
\end{example}
\begin{example}
In this example we use the method in equation [Wolfram MathWorld, \href{https://mathworld.wolfram.com/LerchTranscendent.html}{(10)}] and apply it to equation (\ref{eq:ram_q_poch_inf});
\begin{equation}
\int_0^{\infty } \frac{\log (x) \log (\log (x)) (-q x;q)_{\infty }}{\sqrt{x} (-x;q)_{\infty }} \, dx=-4 i C 
\end{equation}
where $0< Re(q)<1$ and $C$ is Catalan's constant given in [Wolfram MathWorld, \href{https://mathworld.wolfram.com/CatalansConstant.html}{(1)}].
\end{example}
\begin{example}
In this example we use equation [arXiv:2412.10395 [math.GM], \href{https://arxiv.org/pdf/2412.10395}{(1)}]  and apply it to equation (\ref{eq:ram_q_poch_inf});
\begin{multline}
\int_0^{\infty } \frac{\log (\log (a x)) (-q x;q)_{\infty }}{x^{3/4} (-x;q)_{\infty }} \, dx=\log \left(\frac{e^{\frac{i \pi ^2}{\sqrt{2}}} (8 \pi )^{\sqrt{2} \pi } \left(\frac{\Gamma \left(\frac{\pi
   -i \log (a)}{8 \pi }\right)}{\Gamma \left(\frac{5}{8}-\frac{i \log (a)}{8 \pi }\right)}\right)^{2 (-1)^{3/4} \pi }}{\left(\frac{\Gamma \left(\frac{3}{8}-\frac{i \log (a)}{8 \pi }\right)}{\Gamma
   \left(\frac{7}{8}-\frac{i \log (a)}{8 \pi }\right)}\right)^{2 \sqrt[4]{-1} \pi }}\right)
\end{multline}
where $0< Re(q)<1$ and there exists a singularity at $x=1/a$.
\end{example}
\section{Mellin transforms involving the product of Special functions}
In this section we extend the finite series listed in the \cite{prud2,prud3,brychkov}  by deriving a table of infinite integrals in terms of series. To achieve this table we applied the contour integral method in \cite{reyn4} to derive the propositions that follow. The derivations start by applying equation (\ref{eq:contour_gen_rational}) to the cited equations then applying the method in \cite{reyn4} to derive the Mellin transform involving special functions modulated by a logarithm raised to a power. These integral transforms are expressed in terms of series involving the product of special functions. The contour integral representation in equation (\ref{eq:contour_gen_rational}) is derived by applying the method in \cite{reyn4} to equation (3.194.4) in \cite{grad}, [Wolfram MathWorld, \href{https://mathworld.wolfram.com/PochhammerSymbol.html}{(8)}] and [Wolfram MathWorld, \href{https://mathworld.wolfram.com/BinomialTheorem.html}{(2)}]. Similar types of integrals over the real line were studied by Ramanujan in equations (3.14), (5.4-5.43) and (10.2) in \cite{ramanujan_qj}.
\begin{theorem}
Contour integral representation for the Mellin transform of a generalized rational polynomial. 
\begin{multline}\label{eq:contour_gen_rational}
\frac{1}{2\pi i}\int_{C}\int_{0}^{\infty}a^w w^{-k-1} \left(b x^v+1\right)^{-n-1} x^{m+n v+w}dxdw\\
=\frac{1}{2\pi i}\int_{C}\sum_{j=0}^{n}\sum_{l=0}^{j}\frac{\pi  (-1)^{-j} a^w \left(-\frac{1}{v}\right)^l
   \binom{j}{l} S_n^{(j)} w^{-k+l-1} \csc \left(\frac{\pi  (m+w+1)}{v}\right) b^{-\frac{m+n v+w+1}{v}}
  }{v n!}\\ \times
 \left(-\frac{m+(n-1) v+1}{v}\right)^{j-l}
\end{multline}
where $Re(v)>1$.
\end{theorem}
We take equation (\ref{eq:contour_gen_rational}) and multiply it by the left-hand side of equation (5.7.6.1) in \cite{prud2} given by $\frac{J_{p+\nu }(z)}{p!}$. Next we replace the variable $m\to m+p$. Then we replace the variables in equation (5.7.6.1) with $t\to x,k\to p,v\to \nu$. Next we take equation (5.7.6.1) and multiply it by $\left(b x^v+1\right)^{-n-1} x^{m+n v+w}$ and $a^w w^{-k-1}$ which converts equation (5.7.6.1) into the Cauchy contour integral representation. Since the left-hand sides of both transformed equations are equal we may equate the right-hand sides to yield the Cauchy contour integral representation given by;
\begin{multline}\label{eq:contour_gen_rational1}
\frac{1}{2\pi i}\int_{C}\int_{0}^{\infty}a^w w^{-k-1} z^{\nu /2} (z-2 x)^{-\nu /2} \left(b x^v+1\right)^{-n-1} J_{\nu }\left(\sqrt{z^2-2 x z}\right) x^{m+n v+w}dwdx\\
=\frac{1}{2\pi i}\int_{C}\sum _{j=0}^n \sum _{l=0}^j \sum _{p=0}^{\infty }\frac{\pi  (-1)^{-j} a^w \left(-\frac{1}{v}\right)^l
   \binom{j}{l} S_n^{(j)} w^{-k+l-1} J_{p+\nu }(z)  \left(-\frac{m+(n-1) v+p+1}{v}\right)^{j-l}}{v n! p! \sin \left(\frac{\pi  (m+p+w+1)}{v}\right)b^{\frac{m+n v+p+w+1}{v}} }dw
\end{multline}
where $Re(v)>0$.
\begin{proposition}
From equation (\ref{eq:contour_gen_rational1}), equation (5) in \cite{reyn5} and the method in \cite{reyn4}. Series evaluation of a complex-displaced Bessel integral with logarithmic powers and rational denominator.
\begin{multline}
\int_0^{\infty } \frac{x^{m+n v} z^{\nu /2} J_{\nu }\left(\sqrt{-2 x z+z^2}\right) \log ^k(a x)}{(-2
   x+z)^{\nu /2} \left(1+b x^v\right)^{n+1}} \, dx\\
=-\sum _{j=0}^n \sum _{l=0}^j \sum _{p=0}^{\infty }
   \frac{(1+k-l)_l (-1)^{-j} b^{-\frac{1+m+p+n v}{v}} e^{\frac{i (1+m+p) \pi }{v}} (2 \pi )^{1+k-l}
   \left(-\frac{1}{v}\right)^l }{n! p!}\\ \times
\left(\frac{i}{v}\right)^{1+k-l} \left(-\frac{1+m+p+(-1+n) v}{v}\right)^{j-l}
   J_{p+\nu }(z) \binom{j}{l}\\
 \Phi \left(e^{\frac{2 i (1+m+p) \pi }{v}},-k+l,\frac{\pi -i v \log \left(a
   b^{-1/v}\right)}{2 \pi }\right) S_n^{(j)}
\end{multline}
where $0< Re(z) <1,Re(m)<0,|Re(m)|>Re(v)>0,0< Re(\mu) <1$ and there exists a singularity at $x=1/a$ when $Re(a)>0$.
\end{proposition}
\begin{proposition}
From equation (5.7.6.2) in \cite{prud2}. Series evaluation of a displaced Bessel integral with logarithmic powers and rational denominator.
\begin{multline}
\int_0^{\infty } \frac{x^{m+n v} (2 x+z)^{\nu /2} J_{\nu }\left(\sqrt{2 x z+z^2}\right) \log ^k(a x)}{\left(1+b x^v\right)^{n+1} z^{\nu /2}} \, dx\\
=-\sum _{j=0}^n \sum _{l=0}^j \sum
   _{p=0}^{\infty } \frac{(1+k-l)_l i (-1)^{-j} b^{-\frac{1+m+p+n v}{v}} e^{\frac{i (1+m+p) \pi }{v}} (2 \pi )^{1+k-l} \left(-\frac{1}{v}\right)^l \left(\frac{i}{v}\right)^{k-l}
   }{v n! p!}\\ \times
\left(-\frac{1+m+p+(-1+n) v}{v}\right)^{j-l} J_{-p+\nu }(z) \binom{j}{l}\\
 \Phi \left(e^{\frac{2 i (1+m+p) \pi }{v}},-k+l,\frac{v \left(\frac{\pi }{v}-i \log \left(a b^{-1/v}\right)\right)}{2\pi }\right) S_n^{(j)}
\end{multline}
where $Im(a)\neq 0, Re(m)<0, Re(v)>3\pi, Re(z)>0$.
\end{proposition}
\begin{proposition}
From equation (5.7.6.3) in \cite{prud2}. Series evaluation of a branch-cut Bessel integral with logarithmic powers and rational denominator.
\begin{multline}
\int_0^{\infty } \frac{x^{1+m+n v-\nu -\sigma } \left(1-\frac{2 x}{z}\right)^{\frac{1}{2} (-1+\sigma )} J_{1-\sigma }\left(\sqrt{-2 x z+z^2}\right) \log ^k(a x)}{\left(1+b
   x^v\right)^{n+1}} \, dx\\
=\sum _{j=0}^n \sum _{l=0}^j \sum _{p=-\infty }^{\infty } \frac{(1+k-l)_l (-1)^{-j+1} b^{-\frac{1+m+p+n v}{v}} e^{\frac{i (1+m+p) \pi }{v}} (2 \pi )^{1+k-l}
   \left(-\frac{1}{v}\right)^l \left(\frac{i}{v}\right)^{1+k-l} }{n! \Gamma (p+\nu +\sigma )}\\ \times
\left(-\frac{1+m+p+(-1+n) v}{v}\right)^{j-l} J_{p+\nu }(z) \binom{j}{l}\\
 \Phi \left(e^{\frac{2 i (1+m+p) \pi }{v}},-k+l,\frac{\pi
   -i v \log \left(a b^{-1/v}\right)}{2 \pi }\right) S_n^{(j)}
\end{multline}
where $Re(v)>\pi, Re(z)>0$, and $\sigma=0 \vee 1$.
\end{proposition}
\begin{proposition}
From equation (5.7.15.1) in \cite{prud2}. Series evaluation of a Bessel–Log integral with reciprocal argument and rational denominator.
\begin{multline}
\int_0^{\infty } \frac{x^{m+n v} J_0\left(\left(\frac{1}{x}+x\right) z\right) \log ^k(a x)}{\left(1+b x^v\right)^{n+1}} \, dx\\
=\sum _{j=0}^n \sum _{l=0}^j \sum _{p=-\infty }^{\infty } \frac{i
   (-1)^{-j+p+1} b^{-\frac{1+m+2 p+n v}{v}} e^{\frac{i (1+m+2 p) \pi }{v}} (2 \pi )^{1+k-l} \left(-\frac{1}{v}\right)^l \left(\frac{i}{v}\right)^{k-l} }{v n!}\\ \times
\left(-\frac{1+m+2 p+(-1+n) v}{v}\right)^{j-l}
   J_p(z){}^2 \binom{j}{l} (1+k-l)_l\\
 \Phi \left(e^{\frac{2 i (1+m+2 p) \pi }{v}},-k+l,\frac{v \left(\frac{\pi }{v}-i \log \left(a b^{-1/v}\right)\right)}{2 \pi }\right) S_n^{(j)}
\end{multline}
where $Re(v)>\pi, Re(z)>0$.
\end{proposition}
\begin{proposition}
From equation (5.8.3.2) in \cite{prud2}. Series evaluation of a Modified Bessel-yype integral with logarithmic powers and rational generating functions.
\begin{multline}
\int_0^{\infty } \frac{e^{\frac{1}{2} \left(\frac{1}{x}+x\right) z} x^{m+n v} \log ^k(a x)}{\left(1+b x^v\right)^{n+1}} \, dx\\
=-\sum _{j=0}^n \sum _{l=0}^j \sum _{p=-\infty }^{\infty } \frac{i
   (-1)^{-j} b^{-\frac{1+m+p+n v}{v}} e^{\frac{i (1+m+p) \pi }{v}} (2 \pi )^{1+k-l} \left(-\frac{1}{v}\right)^l \left(\frac{i}{v}\right)^{k-l} (1+k-l)_l }{v n!}\\ \times
\left(-\frac{1+m+p+(-1+n) v}{v}\right)^{j-l}
   I_p(z) \binom{j}{l}\\
 \Phi \left(e^{\frac{2 i (1+m+p) \pi }{v}},-k+l,\frac{v \left(\frac{\pi }{v}-i \log \left(a b^{-1/v}\right)\right)}{2 \pi }\right) S_n^{(j)}
\end{multline}
where $Re(v)>\pi,Re(z)<0, Re(m)<0$.
\end{proposition}
\begin{proposition}
From equation (6.9.3.7) in \cite{prud3}. An infinite integral involving Gaussian exponent, Confluent hypergeometric function, logarithmic powers, and rational generating functions.
\begin{multline}
\int_0^{\infty } \frac{e^{-x^2+2 x \beta } x^{m+n v} \left(1+x^2 z\right)^{-\alpha } \, _1F_1\left(\alpha ;\frac{1}{2};\frac{x^2 z (x-\beta )^2}{1+x^2 z}\right) \log ^k(a x)}{\left(1+b
   x^v\right)^{n+1}} \, dx\\
=-\sum _{j=0}^n \sum _{l=0}^j \sum _{p=0}^{\infty } \frac{i (-1)^{-j} b^{-\frac{1+m+p+n v}{v}} e^{\frac{i (1+m+p) \pi }{v}} (2 \pi )^{1+k-l} \left(-\frac{1}{v}\right)^l
   \left(\frac{i}{v}\right)^{k-l} }{v n! p!}\\ \times
\left(-\frac{1+m+p+(-1+n) v}{v}\right)^{j-l} (1+k-l)_l \binom{j}{l} H_p(\beta )\\
 \Phi \left(e^{\frac{2 i (1+m+p) \pi }{v}},-k+l,\frac{v \left(\frac{\pi }{v}-i \log
   \left(a b^{-1/v}\right)\right)}{2 \pi }\right) \, _3F_1\left(\frac{1}{2}-\frac{p}{2},-\frac{p}{2},\alpha ;\frac{1}{2};z\right) S_n^{(j)}
\end{multline}
where $Re(v)>10\pi, Re(\beta)<0,|Re(m)|>1,-1< Re(z)<1$ in order for the series to converge. There exists a singularity at $x=1/a$ when $Re(a)>0$.
\end{proposition}
\begin{proposition}
From equation (5.5.1.1) in \cite{brychkov}. Integral representation involving logarithmic, power, and modified Bessel functions. 
\begin{multline}
\int_0^{\infty } \frac{x^{-\frac{1}{2}+m} K_{\frac{1}{2}+n}\left(\frac{1}{2} \left(z-\sqrt{\frac{z (-2+x
   z)}{x}}\right)\right) K_{\frac{1}{2}+n}\left(\frac{1}{2} \left(z+\sqrt{\frac{z (-2+x z)}{x}}\right)\right) \log
   ^k(a x)}{\left(1+b x^u\right)^{q+1}} \, dx\\
=\sum _{j=0}^q \sum _{p=0}^j \sum _{h=0}^p \sum _{l=0}^n
   \frac{(-1)^{1-j} 2^{1-h+k} b^{-\frac{1+l+m}{u}} e^{\frac{i \pi  (1+l+m+q u)}{u}} m^{-h+p} \pi ^{\frac{3}{2}-h+k}
   \left(-\frac{1}{u}\right)^p \left(\frac{i}{u}\right)^{1-h+k} }{l! (-l+n)! q!}\\ \times
\left(\frac{-1-l+u}{u}\right)^{j-p} (1-h+k)_h
   K_{\frac{1}{2}+l}(z) \binom{j}{p} \binom{p}{h} (l+n)! \\ 
\times \Phi \left(e^{\frac{2 i \pi  (1+l+m+q u)}{u}},h-k,\frac{\pi
   -i u \log \left(a b^{-1/u}\right)}{2 \pi }\right) S_q^{(j)}
\end{multline}
where $Re(u)>n+1, |Re(m)|<1$ and there exists a singularity at $x=1/a$ when $Re(b)>0$ and $x=-i/a$ when $Re(b)<0$.
\end{proposition}
\begin{proposition}
From pp. 181 Exercise number 20 in \cite{andrews}. A Mellin–Barnes representation of the integral involving the square of the Gauss Hypergeometric function and a generalized Beta-type denominator.
\begin{multline}\label{eq:andrews}
\int_0^{\infty } \frac{x^{-1+s} \, _2F_1(a,b;c;x){}^2}{\left(1+d x^u\right)^{p+1}} \, dx
=\sum _{n=0}^{\infty } \frac{(-1)^{-p} d^{-\frac{n+s}{u}} \pi  \binom{-\frac{n+s-p u}{u}}{p}
   \csc \left(\frac{\pi  (n+s-p u)}{u}\right) }{u n! (c)_n (-1+2 c)_n}\\ \times
\, _4F_3\left(\frac{1}{2},\frac{1}{2}+a+b-c,\frac{1}{2}-\frac{n}{2},-\frac{n}{2};\frac{1}{2}+a,\frac{1}{2}+b,\frac{3}{2}-c-n;1\right)\\
\times (2 a)_n
   (2 b)_n \left(-\frac{1}{2}+c\right)_n
\end{multline}
where $Re(p) > Re(u)$.
\end{proposition}
\begin{proposition}
From equation (5.7.1.3) in \cite{brychkov}. A Mellin-type integral involving associated Legendre functions, logarithmic powers, and rational generating function. 
\begin{multline}\label{eq:legendre}
\int_0^{\infty } \frac{x^{m+n v} \left(1+x^2+2 x z\right)^{q/2} P_q\left(\frac{1+x z}{\sqrt{1+x^2+2 x
   z}}\right) \log ^k(a x)}{\left(1+b x^v\right)^{n+1}} \, dx\\
=-\sum _{j=0}^n \sum _{l=0}^j \sum _{p=0}^q
   \frac{(1+k-l)_l (-1)^{-j} b^{-\frac{1+m+p+n v}{v}} e^{\frac{i (1+m+p) \pi }{v}} (2 \pi )^{1+k-l}
   \left(-\frac{1}{v}\right)^l }{n!}\\ \times
\left(\frac{i}{v}\right)^{1+k-l} \left(-\frac{1+m+p+(-1+n) v}{v}\right)^{j-l}
   \binom{j}{l} \binom{q}{p}\\
 \Phi \left(e^{\frac{2 i (1+m+p) \pi }{v}},-k+l,\frac{\pi -i v \log \left(a
   b^{-1/v}\right)}{2 \pi }\right) P_p(z) S_n^{(j)}
\end{multline}
where $Re(v)>0$ and there exists a singularity at $x=1/a$ when $Re(b)>0$.
\end{proposition}
\begin{proposition}
From equation (5.8.1.1) in \cite{brychkov}. A Mellin-type integral involving Legendre polynomials and logarithmic powers.
\begin{multline}
\int_0^{\infty } \frac{x^{m+n v} \left(1+x^2+2 x z\right)^{q/2} T_q\left(\frac{1+x z}{\sqrt{1+x^2+2 x
   z}}\right) \log ^k(a x)}{\left(1+b x^v\right)^{n+1}} \, dx\\
=-\sum _{j=0}^n \sum _{l=0}^j \sum _{p=0}^q
   \frac{(1+k-l)_l (-1)^{-j} b^{-\frac{1+m+p+n v}{v}} e^{\frac{i (1+m+p) \pi }{v}} (2 \pi )^{1+k-l}
   \left(-\frac{1}{v}\right)^l \left(\frac{i}{v}\right)^{1+k-l} }{n!}\\ \times
\left(-\frac{1+m+p+(-1+n) v}{v}\right)^{j-l}
   \binom{j}{l} \binom{q}{p} T_p(z)\\
 \Phi \left(e^{\frac{2 i (1+m+p) \pi }{v}},-k+l,\frac{\pi -i v \log \left(a
   b^{-1/v}\right)}{2 \pi }\right) S_n^{(j)}
\end{multline}
where $Re(v)>0$ and there exists a singularity at $x=1/a$ when $Re(b)>0$.
\end{proposition}
\begin{proposition}
From equation (5.8.4.1) in \cite{brychkov}. A generalized integral representation involving the logarithm of a linear argument, Ultra-Spherical Gegenbauer Polynomials $U_q$, and rational-hypergeometric weight functions.
\begin{multline}
\int_0^{\infty } \frac{x^{m+n v} \left(1+b x^v\right)^{-1-n} \left(1+x^2+2 x z\right)^{q/2} U_q\left(\frac{1+x
   z}{\sqrt{1+x^2+2 x z}}\right) \log ^k(a x)}{1+q} \, dx\\
=-\sum _{j=0}^n \sum _{l=0}^j \sum _{p=0}^q \frac{(1+k-l)_l
   (-1)^{-j} b^{-\frac{1+m+p+n v}{v}} e^{\frac{i (1+m+p) \pi }{v}} (2 \pi )^{1+k-l} \left(-\frac{1}{v}\right)^l
   \left(\frac{i}{v}\right)^{1+k-l} }{(1+p) n!}\\ \times
\left(-\frac{1+m+p+(-1+n) v}{v}\right)^{j-l} \binom{j}{l} \binom{q}{p} U_p(z)\\
 \Phi\left(e^{\frac{2 i (1+m+p) \pi }{v}},-k+l,\frac{\pi -i v \log \left(a b^{-1/v}\right)}{2 \pi }\right)
   S_n^{(j)}
\end{multline}
where $Re(v)>0$ and there exists a singularity at $x=1/a$ when $Re(b)>0$.
\end{proposition}
\begin{proposition}
From equation (5.9.1.1) in \cite{brychkov}. An iIntegral with Hermite polynomials $H_{2q}$, alternating binomial weights, and logarithmic powers in hypergeometric-rational kernels. 
\begin{multline}
\int_0^{\infty } \frac{(-1+x)^q x^{m+n v} \left(1+b x^v\right)^{-1-n} H_{2 q}\left(\sqrt{\frac{x}{-1+x}}
   z\right) \log ^k(a x)}{(2 q)!} \, dx\\
=-\sum _{j=0}^n \sum _{l=0}^j \sum _{p=0}^q \frac{(1+k-l)_l (-1)^{-j}
   b^{-\frac{1+m+p+n v}{v}} e^{\frac{i (1+m+p) \pi }{v}} (2 \pi )^{1+k-l} \left(-\frac{1}{v}\right)^l
   \left(\frac{i}{v}\right)^{1+k-l} }{n!
   (2 p)! (-p+q)!}\\ \times
\left(-\frac{1+m+p+(-1+n) v}{v}\right)^{j-l} \binom{j}{l} H_{2 p}(z) \\
\Phi
   \left(e^{\frac{2 i (1+m+p) \pi }{v}},-k+l,\frac{\pi -i v \log \left(a b^{-1/v}\right)}{2 \pi }\right) S_n^{(j)}
\end{multline}
where $Re(v)>0$ and there exists a singularity at $x=1/a$ when $Re(b)>0$.
\end{proposition}
\begin{proposition}
From equation (5.9.1.9) in \cite{brychkov}. An Integral with odd-degree Hermite $H_{1+2q}$, fractional binomial $(-1+x)^{\frac{1}{2}+q}$, and log-powered hypergeometric-rational weights. The finite series involves Stirling number of the first kind and the Hurwitz-Lerch zeta function.
\begin{multline}
\int_0^{\infty } \frac{(-1+x)^{\frac{1}{2}+q} x^{-\frac{1}{2}+m+n v} \left(1+b x^v\right)^{-1-n} H_{1+2
   q}\left(\sqrt{\frac{x}{-1+x}} z\right) \log ^k(a x)}{(1+2 q)!} \, dx\\
=-\sum _{j=0}^n \sum _{l=0}^j \sum _{p=0}^q
   \frac{(1+k-l)_l (-1)^{-j} b^{-\frac{1+m+p+n v}{v}} e^{\frac{i (1+m+p) \pi }{v}} (2 \pi )^{1+k-l}
   \left(-\frac{1}{v}\right)^l \left(\frac{i}{v}\right)^{1+k-l} }{n! (1+2 p)! (-p+q)!}\\ \times
\left(-\frac{1+m+p+(-1+n) v}{v}\right)^{j-l}
   \binom{j}{l} H_{1+2 p}(z)\\
 \Phi \left(e^{\frac{2 i (1+m+p) \pi }{v}},-k+l,\frac{\pi -i v \log \left(a
   b^{-1/v}\right)}{2 \pi }\right) S_n^{(j)}
\end{multline}
where $Re(v)>0,Re(k)>-1/2, 0< Re(b) < 1$ and there exists a singularity at $x=1/a$ when $Re(b)>0$.
\end{proposition}
\begin{proposition}
From equation (5.10.2.1) in \cite{brychkov}. A Mellin-transform evaluation of the generalized integral involving Laguerre polynomials, logarithmic powers, and Hypergeometric-type kernels. The finite series involves Stirling number of the first kind and the Hurwitz-Lerch zeta function.
\begin{multline}
\int_0^{\infty } \frac{x^{m+n v} (1+x)^q \left(1+b x^v\right)^{-1-n} L_q^{\lambda }\left(\frac{x z}{1+x}\right) \log ^k(a x)}{(1+\lambda )_q} \, dx\\
=-\sum _{j=0}^n \sum _{l=0}^j \sum_{p=0}^q \frac{(1+k-l)_l (-1)^{-j} b^{-\frac{1+m+p+n v}{v}} e^{\frac{i (1+m+p) \pi }{v}} (2 \pi )^{1+k-l} \left(-\frac{1}{v}\right)^l \left(\frac{i}{v}\right)^{1+k-l}
  }{n! (-p+q)! (1+\lambda )_p}\\ \times
 \left(-\frac{1+m+p+(-1+n) v}{v}\right)^{j-l} \binom{j}{l}\\
 \Phi \left(e^{\frac{2 i (1+m+p) \pi }{v}},-k+l,\frac{\pi -i v \log \left(a b^{-1/v}\right)}{2 \pi }\right) L_p^{\lambda }(z)
   S_n^{(j)}
\end{multline}
where $Re(v)>0,Re(k)>-1/2, 0< Re(b) < 1$ and there exists a singularity at $x=1/a$ when $Re(b)>0$.
\end{proposition}
\begin{proposition}
From equation (5.10.3.12). Integral representation of a generalized series Involving Laguerre polynomials, and logarithmic function in terms of the Stirling number of the first and second kind, Hurwitz-Lerch zeta function and Laguerre polynomials. 
\begin{multline}
\int_0^{\infty } \left(\sum _{p=0}^h \frac{(-1)^p (-x)^{-h} x^{m+n v} L_{h+q}^{-h+\lambda }(-p x+z) \log ^k(a
   x)}{(h-p)! p! \left(1+b x^v\right)^{n+1}}\right) \, dx\\
=-\sum _{j=0}^n \sum _{l=0}^j \sum _{p=0}^q \frac{(1+k-l)_l
   (-1)^{-j} b^{-\frac{1+m+p+n v}{v}} e^{\frac{i (1+m+p) \pi }{v}} (2 \pi )^{1+k-l} \left(-\frac{1}{v}\right)^l
   \left(\frac{i}{v}\right)^{1+k-l} }{n! (h+p)!}\\ \times
\left(-\frac{1+m+p+(-1+n) v}{v}\right)^{j-l} \binom{j}{l}\\
 \Phi \left(e^{\frac{2 i
   (1+m+p) \pi }{v}},-k+l,\frac{\pi -i v \log \left(a b^{-1/v}\right)}{2 \pi }\right) L_{-p+q}^{p+\lambda }(z)
   S_n^{(j)} \mathcal{S}_{h+p}^{(h)}
\end{multline}
where $Re(v)>2n>0,Re(k)>-1/2, 0< Re(b) < 1$ and there exists a singularity at $x=1/a$ when $Re(b)>0$.
\end{proposition}
\begin{proposition}
From equation (5.10.3.20) in \cite{brychkov}. Mellin integral of Laguerre polynomials with shifted argument, logarithmic powers, and rational decay on $\mathbb{R}^+$. The finite series involves the Stirling number of the first kind, Laguerre polynomials and Hurwitz-Lerch zeta function.
\begin{multline}
\int_0^{\infty } \frac{x^{m+q+n v} L_q^{-q+\lambda }\left(-\frac{1}{x}+z\right) \log ^k(a x)}{\left(1+b
   x^v\right)^{n+1}} \, dx\\
=-\sum _{j=0}^n \sum _{l=0}^j \sum _{p=0}^q \frac{(1+k-l)_l (-1)^{-j} b^{-\frac{1+m+p+n
   v}{v}} e^{\frac{i (1+m+p) \pi }{v}} (2 \pi )^{1+k-l} \left(-\frac{1}{v}\right)^l \left(\frac{i}{v}\right)^{1+k-l}
   }{n! (-p+q)!}\\ \times
\left(-\frac{1+m+p+(-1+n) v}{v}\right)^{j-l} \binom{j}{l}\\
 \Phi \left(e^{\frac{2 i (1+m+p) \pi }{v}},-k+l,\frac{\pi
   -i v \log \left(a b^{-1/v}\right)}{2 \pi }\right) L_p^{-p+\lambda }(z) S_n^{(j)}
\end{multline}
where $Re(v)>0,Re(k)>-1/2, 0< Re(b) < 1$ and there exists a singularity at $x=1/a$ when $Re(b)>0$.
\end{proposition}
\begin{proposition}
From equation (5.10.3.44) in \cite{brychkov}. Mellin integral of product of shifted-argument Laguerre polynomials with quadratic radical branches, logarithmic weights, and Hypergeometric decay. The finite series involves the Stirling number of the first kind, Laguerre polynomials and Hurwitz-Lerch zeta function.
\begin{multline}
\int_0^{\infty } \frac{(-x)^q x^{m+n v} q! L_q^{-2 q+\lambda }\left(\frac{\sqrt{x} z-\sqrt{4+x z^2}}{2
   \sqrt{x}}\right) L_q^{-2 q+\lambda }\left(\frac{\sqrt{x} z+\sqrt{4+x z^2}}{2 \sqrt{x}}\right) \log ^k(a
   x)}{\left(1+b x^v\right)^{n+1}} \, dx\\
=-\sum _{j=0}^n \sum _{l=0}^j \sum _{p=0}^q \frac{(1+k-l)_l (-1)^{-j}
   b^{-\frac{1+m+p+n v}{v}} e^{\frac{i (1+m+p) \pi }{v}} (2 \pi )^{1+k-l} \left(-\frac{1}{v}\right)^l
   \left(\frac{i}{v}\right)^{1+k-l} }{n! (-p+q)!}\\ \times
\left(-\frac{1+m+p+(-1+n) v}{v}\right)^{j-l} \binom{j}{l}\\
 \Phi \left(e^{\frac{2 i
   (1+m+p) \pi }{v}},-k+l,\frac{\pi -i v \log \left(a b^{-1/v}\right)}{2 \pi }\right) L_p^{-2 p+\lambda }(z)
   (q-\lambda )_p S_n^{(j)}
\end{multline}
where $Re(v)>0,Re(k)>-1/2, 0< Re(b) < 1$ and there exists a singularity at $x=1/a$ when $Re(b)>0$.
\end{proposition}
\begin{proposition}
From equation (5.11.3.5) in \cite{brychkov}. Definite integral representation of a generalized hypergeometric series involving binomial coefficients, Pochhammer symbols, Jacobi polynomials, and logarithmic powers. 
\begin{multline}
\int_0^{\infty } \left(\sum _{p=0}^h \frac{(-1)^p x^{-h+m+n v} \left(1+b x^v\right)^{-1-n} \binom{h}{p}
   C_{h+q}^{(-h+\lambda )}\left(\frac{p x}{2}+z\right) \log ^k(a x)}{h! (1-\lambda )_h}\right) \, dx\\
=-\sum _{j=0}^n\sum _{l=0}^j \sum _{p=0}^q \frac{(1+k-l)_l i (-1)^{-j} b^{-\frac{1+m+p+n v}{v}} e^{\frac{i (1+m+p) \pi }{v}} (2\pi )^{1+k-l} \left(-\frac{1}{v}\right)^l \left(\frac{i}{v}\right)^{k-l} }{v n!(h+p)!}\\ \times
\left(-\frac{1+m+p+(-1+n)v}{v}\right)^{j-l} \binom{j}{l} C_{-p+q}^{(p+\lambda )}(z)\\
 \Phi \left(e^{\frac{2 i (1+m+p) \pi }{v}},-k+l,\frac{\pi-i v \log \left(a b^{-1/v}\right)}{2 \pi }\right) (\lambda )_p S_n^{(j)} \mathcal{S}_{h+p}^{(h)}
\end{multline}
where $Re(v)>2n>0,Re(k)>-1/2, 0< Re(b) < 1$ and there exists a singularity at $x=1/a$ when $Re(b)>0$.
\end{proposition}
\begin{proposition}
From equation (5.11.3.10) in \cite{brychkov}. Definite integral involving Jacobi polynomials with argument $\sqrt{\frac{x}{1+x}}z$, logarithmic powers, and rational weight function.
\begin{multline}
\int_0^{\infty } \frac{x^{m+n v} (1+x)^q \left(1+b x^v\right)^{-1-n} C_{2 q}^{(-q+\lambda
   )}\left(\sqrt{\frac{x}{1+x}} z\right) \log ^k(a x)}{(1-\lambda )_q} \, dx\\
=-\sum _{j=0}^n \sum _{l=0}^j \sum _{p=0}^q
   \frac{(1+k-l)_l (-1)^{-j} b^{-\frac{1+m+p+n v}{v}} e^{\frac{i (1+m+p) \pi }{v}} (2 \pi )^{1+k-l}
   \left(-\frac{1}{v}\right)^l \left(\frac{i}{v}\right)^{1+k-l} }{n! (-p+q)! (1-\lambda )_p}\\ \times
\left(-\frac{1+m+p+(-1+n) v}{v}\right)^{j-l}\binom{j}{l} C_{2 p}^{(-p+\lambda )}(z)\\
 \Phi \left(e^{\frac{2 i (1+m+p) \pi }{v}},-k+l,\frac{\pi -i v \log \left(ab^{-1/v}\right)}{2 \pi }\right) S_n^{(j)}
\end{multline}
where $Re(v)>2n>0,Re(k)>-1/2, 0< Re(b) < 1$ and there exists a singularity at $x=1/a$ when $Re(b)>0$.
\end{proposition}
\begin{proposition}
From equation (5.11.3.23) in \cite{brychkov}.  Definite integral of alternating binomial sum with Gegenbauer polynomial $  C_{2h+2q}^{(-h+\lambda)}\left(\frac{z}{\sqrt{1-px}}\right)  $, logarithmic weight, and rational kernel.
\begin{multline}
\int_0^{\infty } \left(\sum _{p=0}^h \frac{(-1)^p x^{-h+m+n v} (1-p x)^{h+q} \left(1+b x^v\right)^{-1-n}
   \binom{h}{p} C_{2 h+2 q}^{(-h+\lambda )}\left(\frac{z}{\sqrt{1-p x}}\right) \log ^k(a x)}{h! (1-\lambda )_h}\right)
   \, dx\\
=-\sum _{j=0}^n \sum _{l=0}^j \sum _{p=0}^q \frac{(1+k-l)_l (-1)^{-j} b^{-\frac{1+m+p+n v}{v}} e^{\frac{i
   (1+m+p) \pi }{v}} (2 \pi )^{1+k-l} \left(-\frac{1}{v}\right)^l \left(\frac{i}{v}\right)^{1+k-l}
   }{n! (h+p)!}\\ \times
\left(-\frac{1+m+p+(-1+n) v}{v}\right)^{j-l} \binom{j}{l} C_{-2 p+2 q}^{(p+\lambda )}(z)\\
 \Phi \left(e^{\frac{2 i
   (1+m+p) \pi }{v}},-k+l,\frac{\pi -i v \log \left(a b^{-1/v}\right)}{2 \pi }\right) (\lambda )_p S_n^{(j)}
   \mathcal{S}_{h+p}^{(h)}
\end{multline}
where $Re(v)>2n>0,Re(k)>-1/2, 0< Re(b) < 1$ and there exists a singularity at $x=1/a$ when $Re(b)>0$.
\end{proposition}
\begin{proposition}
From equation (5.11.3.31) in \cite{brychkov}. Log-weighted definite integral involving Gegenbauer polynomial $  C_{1+2q}^{(-q+\lambda)}\left(\sqrt{\frac{x}{1+x}}\,z\right)  $ with rational and logarithmic factors.
\begin{multline}
\int_0^{\infty } \frac{x^{-\frac{1}{2}+m+n v} (1+x)^{\frac{1}{2}+q} \left(1+b x^v\right)^{-1-n} C_{1+2
   q}^{(-q+\lambda )}\left(\sqrt{\frac{x}{1+x}} z\right) \log ^k(a x)}{(1-\lambda )_q} \, dx\\
=-\sum _{j=0}^n \sum_{l=0}^j \sum _{p=0}^q \frac{(1+k-l)_l (-1)^{-j} b^{-\frac{1+m+p+n v}{v}} e^{\frac{i (1+m+p) \pi }{v}} (2 \pi)^{1+k-l} \left(-\frac{1}{v}\right)^l \left(\frac{i}{v}\right)^{1+k-l} }{n! (-p+q)! (1-\lambda )_p}\\ \times
\left(-\frac{1+m+p+(-1+n) v}{v}\right)^{j-l} \binom{j}{l} C_{1+2 p}^{(-p+\lambda )}(z)\\
 \Phi \left(e^{\frac{2 i (1+m+p) \pi }{v}},-k+l,\frac{\pi -i v \log\left(a b^{-1/v}\right)}{2 \pi }\right) S_n^{(j)}
\end{multline}
where $Re(v)>2n>0,Re(k)>-1/2, 0< Re(b) < 1$ and there exists a singularity at $x=1/a$ when $Re(b)>0$.
\end{proposition}
\begin{proposition}
From equation (5.11.3.48) in \cite{brychkov}. Definite integral representation of a finite sum involving Gegenbauer polynomials $C_{1+2h+2q}^{(-h+\lambda)}$ with argument $z/\sqrt{1-px}$, binomial expansion, and logarithmic weight.
\begin{multline}
\int_0^{\infty } \left(\sum _{p=0}^h \frac{(-1)^p x^{-h+m+n v} (1-p x)^{\frac{1}{2}+h+q} \binom{h}{p} C_{1+2 h+2 q}^{(-h+\lambda )}\left(\frac{z}{\sqrt{1-p x}}\right) \log ^k(a x)}{\left(1+b
   x^v\right)^{n+1} h!
   (1-\lambda )_h}\right) \, dx\\
=-\sum _{j=0}^n \sum _{l=0}^j \sum _{p=0}^q \frac{(1+k-l)_l (-1)^{-j} b^{-\frac{1+m+p+n
   v}{v}} e^{\frac{i (1+m+p) \pi }{v}} (2 \pi )^{1+k-l} \left(-\frac{1}{v}\right)^l \left(\frac{i}{v}\right)^{1+k-l}
  }{n! (h+p)!}\\ \times
 \left(-\frac{1+m+p+(-1+n) v}{v}\right)^{j-l} \binom{j}{l} C_{1-2 p+2 q}^{(p+\lambda )}(z)\\
 \Phi \left(e^{\frac{2 i
   (1+m+p) \pi }{v}},-k+l,\frac{\pi -i v \log \left(a b^{-1/v}\right)}{2 \pi }\right) (\lambda )_p S_n^{(j)}
   \mathcal{S}_{h+p}^{(h)}
\end{multline}
where $Re(v)>2n>0,Re(k)>-1/2, 0< Re(b) < 1$ and there exists a singularity at $x=1/a$ when $Re(b)>0$.
\end{proposition}
\begin{proposition}
From equation (5.12.2.21) in \cite{brychkov}. Definite integral of a finite alternating sum involving associated Jacobi polynomials $P_{h+q}^{(-h+\rho,-h+\sigma)}$, binomial coefficients, and logarithmic powers.
\begin{multline}
\int_0^{\infty } \left(\sum _{p=0}^h \frac{(-1)^p x^{-h+m+n v} \left(1+b x^v\right)^{-1-n} \binom{h}{p}
   P_{h+q}^{(-h+\rho ,-h+\sigma )}(2 p x+z) \log ^k(a x)}{h! (-q-\rho -\sigma )_h}\right) \, dx\\
=-\sum _{j=0}^n \sum_{l=0}^j \sum _{p=0}^q \frac{(1+k-l)_l (-1)^{-j} b^{-\frac{1+m+p+n v}{v}} e^{\frac{i (1+m+p) \pi }{v}} (2 \pi)^{1+k-l} \left(-\frac{1}{v}\right)^l \left(\frac{i}{v}\right)^{1+k-l} }{n!(h+p)!}\\ \times
\left(-\frac{1+m+p+(-1+n) v}{v}\right)^{j-l}\binom{j}{l}\\
 \Phi \left(e^{\frac{2 i (1+m+p) \pi }{v}},-k+l,\frac{\pi -i v \log \left(a b^{-1/v}\right)}{2 \pi}\right) P_{-p+q}^{(p+\rho ,p+\sigma )}(z) (1+q+\rho +\sigma )_p S_n^{(j)} \mathcal{S}_{h+p}^{(h)}
\end{multline}
where $Re(v)>2n>0,Re(k)>-1/2, 0< Re(b) < 1$ and there exists a singularity at $x=1/a$ when $Re(b)>0$.
\end{proposition}
\begin{proposition}
From equation (4.5.1.6 ) in \cite{prud2}. Errata. There should be a factor $(-1)^{\sigma}$ on the right-hand side. Mellin-type integral of Hermite-modulated hypergeometric kernel with logarithmic weight and generalized Beta-type decay. 
\begin{multline}
\int_0^{\infty } \frac{(-1+x)^q \left(\frac{-1+x}{x}\right)^{\sigma /2} x^{m+n v} \left(1+b x^v\right)^{-1-n} H_{2 q+\sigma }\left(\sqrt{\frac{x}{-1+x}} z\right) \log ^k(a x)}{(2
   q+\sigma )!} \, dx\\
=-\sum _{j=0}^n \sum _{l=0}^j \sum _{p=0}^q \frac{(-1)^{-j} b^{-\frac{1+m+p+n v}{v}} e^{\frac{i (1+m+p) \pi }{v}} (1+k-l)_l (2 \pi )^{1+k-l}
   \left(-\frac{1}{v}\right)^l }{n! (-p+q)! (2 p+\sigma )!}\\ \times
\left(\frac{i}{v}\right)^{1+k-l} \left(-\frac{1+m+p+(-1+n) v}{v}\right)^{j-l} \binom{j}{l}\\
 H_{2 p+\sigma }(z) \Phi \left(e^{\frac{2 i (1+m+p) \pi
   }{v}},-k+l,\frac{\pi -i v \log \left(a b^{-1/v}\right)}{2 \pi }\right) S_n^{(j)}
\end{multline}
where $Re(v)>2n>0,Re(k)>-1/2, 0< Re(b) < 1$ and there exists a singularity at $x=1/a$ when $Re(b)>0$. In the final infinite integral formula the variable $\sigma$ is replaced by $(-1)^{2\sigma}=1$ when $\sigma=0$ or $\sigma=1$.
\end{proposition}
\begin{proposition}
From equation (4.5.1.7) in \cite{prud2}. Definite integral involving power-law, Hypergeometric, Hermite polynomial, and logarithmic terms. 
\begin{multline}
\int_0^{\infty } \frac{x^{m+q+n v} H_q\left(\frac{1}{2 x}+z\right) \log ^k(a x)}{\left(1+b x^v\right)^{n+1}} \,
   dx\\
=-\sum _{j=0}^n \sum _{l=0}^j \sum _{p=0}^q \frac{(1+k-l)_l i (-1)^{-j} b^{-\frac{1+m+p+n v}{v}} e^{\frac{i
   (1+m+p) \pi }{v}} (2 \pi )^{1+k-l} \left(-\frac{1}{v}\right)^l \left(\frac{i}{v}\right)^{k-l}
  }{vn!}\\ \times
 \left(-\frac{1+m+p+(-1+n) v}{v}\right)^{j-l} \binom{j}{l} \binom{q}{p} H_p(z)\\
 \Phi \left(e^{\frac{2 i (1+m+p) \pi}{v}},-k+l,\frac{v \left(\frac{\pi }{v}-i \log \left(a b^{-1/v}\right)\right)}{2 \pi }\right) S_n^{(j)}
\end{multline}
where $Re(v)>2n>0,Re(k)>-1/2, 0< Re(b) < 1$ and there exists a singularity at $x=1/a$ when $Re(b)>0$.
\end{proposition}
\begin{proposition}
From equation (4.7.1.3 ) in \cite{prud2}. Definite integral of Jacobi polynomial product with Hypergeometric weight, logarithmic power, and quadratic radical arguments. 
\begin{multline}
\int_0^{\infty } \frac{(-1)^q x^{m+n v} q! P_q^{(\alpha ,\beta )}\left(-x-\sqrt{1+x^2+2 x z}\right) P_q^{(\alpha ,\beta )}\left(-x+\sqrt{1+x^2+2 x
   z}\right) \log ^k(a x)}{(1+\alpha )_q (1+\beta )_q \left(1+b x^v\right)^{n+1} } \, dx\\
=\sum _{j=0}^n \sum _{l=0}^j \sum _{p=0}^q \frac{(-1)^{-j} (1+k-l)_l b^{-\frac{1+m+p+n v}{v}} e^{\frac{i (1+m+p) \pi }{v}} (2
   \pi )^{1+k-l} \left(-\frac{1}{v}\right)^l \left(\frac{i}{v}\right)^{1+k-l} }{n! (-p+q)! (1+\alpha )_p (1+\beta )_p}\\ \times
\left(-\frac{1+m+p+(-1+n) v}{v}\right)^{j-l} \binom{j}{l} \\
\Phi \left(e^{\frac{2 i (1+m+p) \pi
   }{v}},-k+l,\frac{\pi -i v \log \left(a b^{-1/v}\right)}{2 \pi }\right) P_p^{(\alpha ,\beta )}(z) (1+q+\alpha +\beta )_p S_n^{(j)}
\end{multline}
where $Re(v)>2n>0,Re(k)>-1/2, 0< Re(b) < 1$ and there exists a singularity at $x=1/a$ when $Re(b)>0$.
\end{proposition}
\begin{proposition}
From equation (4.7.1.4) in \cite{prud2}. Integral involving Jacobi Polynomials and logarithmic functions over an infinite interval.
\begin{multline}
\int_0^{\infty } \frac{x^{m+n v} \left(1+b x^v\right)^{-1-n} \left(1+\frac{1}{2} (x+x z)\right)^q P_q^{(\alpha
   ,-1-q-\alpha -\beta )}\left(\frac{2+3 x-x z}{2+x+x z}\right) \log ^k(a x)}{(1+\alpha )_q} \, dx\\
=-\sum _{j=0}^n \sum
   _{l=0}^j \sum _{p=0}^q \frac{(-1)^{-j} (1+k-l)_l b^{-\frac{1+m+p+n v}{v}} e^{\frac{i (1+m+p) \pi }{v}} (2 \pi
   )^{1+k-l} \left(-\frac{1}{v}\right)^l \left(\frac{i}{v}\right)^{1+k-l} }{n! (-p+q)! (1+\alpha )_p}\\ \times
\left(-\frac{1+m+p+(-1+n) v}{v}\right)^{j-l}
   \binom{j}{l}\\
 \Phi \left(e^{\frac{2 i (1+m+p) \pi }{v}},-k+l,\frac{\pi -i v \log \left(a b^{-1/v}\right)}{2 \pi
   }\right) P_p^{(\alpha ,-p+\beta )}(z) S_n^{(j)}
\end{multline}
where $Re(v)>2n>0,Re(k)>-1/2, 0< Re(b) < 1$ and there exists a singularity at $x=1/a$ when $Re(b)>0$.
\end{proposition}
\begin{proposition}
From equation (5.1.1.1) in \cite{prud3}. Mellin-type Integral of Bernoulli polynomial kernels with power-logarithmic factors.
\begin{multline}
\int_0^{\infty } \frac{x^{h+m+n v} \left(B_h\left(\frac{1}{x}\right)+(-1)^q \left(\int_0^1
   B_{h-q}\left(\frac{1}{x}+\tau \right) B_q(\tau ) \, d\tau \right) \binom{h}{q}\right) \log ^k(a x)}{\left(1+b x^v\right)^{n+1}} \, dx\\
=-\sum _{j=0}^n \sum _{l=0}^j \sum _{p=0}^q \frac{(1+k-l)_l (-1)^{-j} b^{-\frac{1+m+p+n
   v}{v}} e^{\frac{i (1+m+p) \pi }{v}} (2 \pi )^{1+k-l} \left(-\frac{1}{v}\right)^l \left(\frac{i}{v}\right)^{1+k-l}}{n!}\\ \times
 \left(-\frac{1+m+p+(-1+n) v}{v}\right)^{j-l} B_p \binom{h}{p} \binom{j}{l}\\
 \Phi \left(e^{\frac{2 i (1+m+p) \pi }{v}},-k+l,\frac{\pi -i v \log \left(a b^{-1/v}\right)}{2 \pi }\right) S_n^{(j)}
\end{multline}
where $Re(v)>2n>0,Re(k)>-1/2, 0< Re(b) < 1$ and there exists a singularity at $x=1/a$ when $Re(b)>0,h>q$.
\end{proposition}
\begin{proposition}
From equation (5.1.1.2) in \cite{prud3}. A generalized Mellin-type integral with logarithmic powers and Bernoulli polynomial trigonometric expansion.
\begin{multline}
\int_0^{\infty } x^{m+n v} \left(1+b x^v\right)^{-1-n} \left(\frac{1}{2} x \cot
   \left(\frac{x}{2}\right)\right. \\ \left.
+\frac{(-1)^q x^{1+2 q} \left(\int_0^1 B_{2 q}(\tau ) \cos \left(x \left(-\frac{1}{2}+\tau
   \right)\right) \, d\tau \right) \csc \left(\frac{x}{2}\right)}{2 (2 q)!}\right) \log ^k(a x) \, dx\\
=-\sum _{j=0}^n
   \sum _{l=0}^j \sum _{p=0}^q \frac{(1+k-l)_l i (-1)^{-j+p} b^{-\frac{1+m+2 p+n v}{v}} e^{\frac{i (1+m+2 p) \pi}{v}} (2 \pi )^{1+k-l} \left(-\frac{1}{v}\right)^l \left(\frac{i}{v}\right)^{k-l} }{v n! (2 p)!}\\ \times
\left(-\frac{1+m+2 p+(-1+n)
   v}{v}\right)^{j-l} B_{2 p} \binom{j}{l}\\
 \Phi \left(e^{\frac{2 i (1+m+2 p) \pi }{v}},-k+l,\frac{v \left(\frac{\pi
   }{v}-i \log \left(a b^{-1/v}\right)\right)}{2 \pi }\right) S_n^{(j)}
\end{multline}
where $Re(v)>2n>0,Re(k)>-1/2, 0< Re(b) < 1$ and there exists a singularity at $x=1/a$ when $Re(b)>0$.
\end{proposition}
\begin{proposition}
From equation (5.1.2.7(i)) in \cite{prud3}. Errata. The Binomial coefficient on the right-hand side should be $\binom{m}{k}$. Infinite integral of alternating Bernoulli-Binomial sums with reciprocal argument, logarithmic powers, and rational decay.
\begin{multline}
\int_0^{\infty } \left(\sum _{g=0}^h \frac{(-1)^h (-x)^g x^{-h+m+q+n v} B_{g+q}\left(\frac{1}{x}+z\right)
   \binom{h}{g} \log ^k(a x)}{\left(1+b x^v\right)^{n+1}}\right) \, dx\\
=-\sum _{j=0}^n \sum _{l=0}^j \sum _{p=0}^q
   \frac{(1+k-l)_l (-1)^{-j} b^{-\frac{1+m+p+n v}{v}} e^{\frac{i (1+m+p) \pi }{v}} (2 \pi )^{1+k-l}
   \left(-\frac{1}{v}\right)^l \left(\frac{i}{v}\right)^{1+k-l} }{n!}\\ \times
\left(-\frac{1+m+p+(-1+n) v}{v}\right)^{j-l}
   B_{h+p}(z) \binom{j}{l} \binom{q}{p}\\
 \Phi \left(e^{\frac{2 i (1+m+p) \pi }{v}},-k+l,\frac{\pi -i v \log \left(a
   b^{-1/v}\right)}{2 \pi }\right) S_n^{(j)}
\end{multline}
where $Re(v)>2n>0,Re(k)>-1/2, 0< Re(b) < 1$ and there exists a singularity at $x=1/a$ when $Re(b)>0,q>h$.
\end{proposition}
\begin{proposition}
From equation (5.1.2.7(ii)) in \cite{prud3}. Errata. When the coefficient of $t$ is negative on the right-hand side of the series, the equation should be 
\begin{equation}
\sum _{k=0}^n \binom{n}{k} (-t)^k B_{k+m}(x)=t^{n-m} \sum _{k=0}^m t^k \binom{m}{k}
   B_{k+n}\left(x-\frac{1}{t}\right)
  \end{equation}. Generalized integral representation involving finite Bernoulli sums, binomials, logarithmic powers, and rational decay.
\begin{multline}
\int_0^{\infty } \left(\sum _{g=0}^h \frac{x^{g-h+m+q+n v} B_{g+q}\left(-\frac{1}{x}+z\right) \binom{h}{g} \log^k(a x)}{\left(1+b x^v\right)^{n+1}}\right) \, dx\\
=\sum _{j=0}^n \sum _{l=0}^j \sum _{p=0}^q \frac{(1+k-l)_l(-1)^{-j+p+q+1} b^{-\frac{1+m+p+n v}{v}} e^{\frac{i (1+m+p) \pi }{v}} (2 \pi )^{1+k-l} \left(-\frac{1}{v}\right)^l\left(\frac{i}{v}\right)^{1+k-l} }{n!}\\ \times
\left(-\frac{1+m+p+(-1+n) v}{v}\right)^{j-l} B_{h+p}(z) \binom{j}{l} \binom{q}{p}\\
\Phi \left(e^{\frac{2 i (1+m+p) \pi }{v}},-k+l,\frac{\pi -i v \log \left(a b^{-1/v}\right)}{2 \pi }\right)
   S_n^{(j)}
\end{multline}
where $Re(v)>2n>0,Re(k)>-1/2, 0< Re(b) < 1$ and there exists a singularity at $x=1/a$ when $Re(b)>0,q>h$.
\end{proposition}
\begin{proposition}
From equation (5.1.2.9) in \cite{prud3}. A closed-form integral involving Bernoulli polynomials, binomial expansion, and logarithmic powers over a generalized Beta kernel.
\begin{multline}
\int_0^{\infty } \left(\sum _{g=0}^h \frac{(-1)^q x^{g-h+m+q+n v}
   \left(B_{g+q}\left(-\frac{1}{x}+z\right)+(-1)^{g-h+q} B_{g+q}\left(\frac{1}{x}+z\right)\right) \binom{h}{g} \log
   ^k(a x)}{\left(1+b x^v\right)^{n+1}}\right) \, dx\\
=-\sum _{j=0}^n \sum _{l=0}^j \sum _{p=0}^q \frac{\left(2
   (1+k-l)_l\right) (-1)^{-j} b^{-\frac{1+m+2 p+n v}{v}} e^{\frac{i (1+m+2 p) \pi }{v}} (2 \pi )^{1+k-l}
   \left(-\frac{1}{v}\right)^l \left(\frac{i}{v}\right)^{1+k-l} }{n!}\\ \times
\left(-\frac{1+m+2 p+(-1+n) v}{v}\right)^{j-l} B_{h+2
   p}(z) \binom{j}{l} \binom{q}{2 p}\\
 \Phi \left(e^{\frac{2 i (1+m+2 p) \pi }{v}},-k+l,\frac{\pi -i v \log \left(a
   b^{-1/v}\right)}{2 \pi }\right) S_n^{(j)}
\end{multline}
where $Re(v)>2n>0,Re(k)>-1/2, 0< Re(b) < 1$ and there exists a singularity at $x=1/a$ when $Re(b)>0,q>h$.
\end{proposition}
%
%
\begin{proposition}
From equation (5.1.2.14) in \cite{prud3}. A closed-form integral representation involving Bernoulli polynomial convolutions, binomial coefficients, and logarithmic powers over a generalized Beta-type denominator.
\begin{multline}
\int_0^{\infty } \left(\sum _{p=0}^q \frac{x^{m+p+n v} B_p\left(z+\frac{s-\alpha }{x}\right) B_{-p+q}(\alpha )\binom{q}{p} \log ^k(a x)}{\left(1+b x^v\right)^{n+1}}\right) \, dx\\
=-\sum _{j=0}^n \sum _{l=0}^j \sum _{p=0}^q\frac{(1+k-l)_l (-1)^{-j} b^{-\frac{1+m+p+n v}{v}} e^{\frac{i (1+m+p) \pi }{v}} (2 \pi )^{1+k-l}
   \left(-\frac{1}{v}\right)^l \left(\frac{i}{v}\right)^{1+k-l} }{n!}\\ \times
\left(-\frac{1+m+p+(-1+n) v}{v}\right)^{j-l} B_p(z)
   B_{-p+q}(s) \binom{j}{l} \binom{q}{p} \\
\Phi \left(e^{\frac{2 i (1+m+p) \pi }{v}},-k+l,\frac{\pi -i v \log \left(a
   b^{-1/v}\right)}{2 \pi }\right) S_n^{(j)}
\end{multline}
where $Re(v)>2n>0,Re(k)>-1/2, 0< Re(b) < 1$ and there exists a singularity at $x=1/a$ when $Re(b)>0$.
\end{proposition}
\begin{proposition}
From equation (5.1.4.1) in \cite{prud3}. A closed-form integral representation with Euler Polynomial summation, alternating binomial coefficients, and logarithmic powers over a generalized Beta-Type denominator.
\begin{multline}
\int_0^{\infty } \left(\sum _{g=0}^h \frac{(-1)^{g+h} x^{g-h+m+q+n v} \binom{h}{g} E_{g+q}\left(\frac{1}{x}+z\right) \log ^k(a x)}{\left(1+b x^v\right)^{n+1}}\right) \, dx\\
=-\sum
   _{j=0}^n \sum _{l=0}^j \sum _{p=0}^q \frac{i (1+k-l)_l (-1)^{-j} b^{-\frac{1+m+p+n v}{v}} e^{\frac{i (1+m+p) \pi }{v}} (2 \pi )^{1+k-l} \left(-\frac{1}{v}\right)^l
   }{v n!}\\ \times
\left(\frac{i}{v}\right)^{k-l} \left(-\frac{1+m+p+(-1+n) v}{v}\right)^{j-l} \binom{j}{l} \binom{q}{p} E_{h+p}(z) \\
\Phi \left(e^{\frac{2 i (1+m+p) \pi }{v}},-k+l,\frac{v \left(\frac{\pi
   }{v}-i \log \left(a b^{-1/v}\right)\right)}{2 \pi }\right) S_n^{(j)}
\end{multline}
where $Re(v)>2n>0,Re(k)>-1/2, 0< Re(b) < 1$ and there exists a singularity at $x=1/a$ when $Re(b)>0$.
\end{proposition}
\begin{proposition}
From equation (5.1.4.6) in \cite{prud3}. A closed-form integral involving Euler Polynomial convolution, binomial expansion, and logarithmic powers over a generalized Beta kernel.
\begin{multline}
\int_0^{\infty } \left(\sum _{p=0}^q \frac{x^{m+p+n v} \binom{q}{p} E_p\left(z+\frac{s-\alpha }{x}\right) E_{-p+q}(\alpha ) \log ^k(a x)}{\left(1+b x^v\right)^{n+1}}\right) \,
   dx\\
=-\sum _{j=0}^n \sum _{l=0}^j \sum _{p=0}^q \frac{(1+k-l)_l \left(i (-1)^{-j} b^{-\frac{1+m+p+n v}{v}} e^{\frac{i (1+m+p) \pi }{v}} (2 \pi )^{1+k-l} \left(-\frac{1}{v}\right)^l
   \left(\frac{i}{v}\right)^{k-l} \right)}{v n!}\\ \times
\left(-\frac{1+m+p+(-1+n) v}{v}\right)^{j-l} \binom{j}{l} \binom{q}{p} E_p(z) E_{-p+q}(s) \\
\Phi \left(e^{\frac{2 i (1+m+p) \pi }{v}},-k+l,\frac{v
   \left(\frac{\pi }{v}-i \log \left(a b^{-1/v}\right)\right)}{2 \pi }\right) S_n^{(j)}
\end{multline}
where $Re(v)>2n>0,Re(k)>-1/2, 0< Re(b) < 1$ and there exists a singularity at $x=1/a$ when $Re(b)>0$.
\end{proposition}
%
%
%
\section{Special cases and evaluations of Propositions in terms of special functions and fundamental constants}
\begin{example}
A generalized integral involving Legendre functions of general order and logarithmic denominator. In this example we use equation (\ref{eq:legendre}) and replace $m\to m-n v$ then form a second equation replacing $m \to s$ and taking their difference. Next we set $k \to -1, a\to e^{a}$ then form another equation by replacing $a\to -a$ and taking their difference and simplify.
\begin{multline}\label{eq:legendre1}
\int_0^{\infty } \frac{\left(x^m-x^s\right) \left(1+b x^v\right)^{-1-n} \left(1+x^2+2 x z\right)^{q/2} P_q\left(\frac{1+x z}{\sqrt{1+x^2+2 x z}}\right)}{a^2-\log ^2(x)} \, dx\\
=\sum _{j=0}^n \sum _{l=0}^j
   \sum _{p=0}^q \frac{(-1)^{-j+l} 2^{-1-l} b^{-\frac{1+m+p+s}{v}} \pi ^{-l} \left(-\frac{1}{v}\right)^l \left(\frac{i}{v}\right)^{-l} }{a n!}\\ \times
 \left(-\frac{1+m+p-v}{v}\right)^{-l} \left(-\frac{1+p+s-v}{v}\right)^{-l}
   \binom{j}{l} \binom{q}{p} l!\\
   \left(b^{s/v} e^{\frac{i \pi  (1+m+p-n v)}{v}} \left(-\frac{1+m+p-v}{v}\right)^j \left(-\frac{1+p+s-v}{v}\right)^l \right. \\ \left.
\Phi \left(e^{\frac{2 i \pi  (1+m+p-n v)}{v}},1+l,\frac{\pi +i
   v \left(a-\log \left(b^{-1/v}\right)\right)}{2 \pi }\right)\right. \\ \left.
-b^{s/v} e^{\frac{i \pi  (1+m+p-n v)}{v}} \left(-\frac{1+m+p-v}{v}\right)^j \left(-\frac{1+p+s-v}{v}\right)^l \right. \\ \left.
\Phi \left(e^{\frac{2 i \pi  (1+m+p-nv)}{v}},1+l,\frac{\pi -i v \left(a+\log \left(b^{-1/v}\right)\right)}{2 \pi }\right)\right. \\ \left.
+b^{m/v} e^{\frac{i \pi  (1+p+s-n v)}{v}} \left(-\frac{1+m+p-v}{v}\right)^l \left(-\frac{1+p+s-v}{v}\right)^j \right. \\ \left.
\left(-\Phi\left(e^{\frac{2 i \pi  (1+p+s-n v)}{v}},1+l,\frac{\pi +i v \left(a-\log \left(b^{-1/v}\right)\right)}{2 \pi }\right) \right.\right. \\ \left.\left.
+\Phi \left(e^{\frac{2 i \pi  (1+p+s-n v)}{v}},1+l,\frac{\pi -i v \left(a+\log
   \left(b^{-1/v}\right)\right)}{2 \pi }\right)\right)\right) P_p(z) S_n^{(j)}
\end{multline}
where $Re(v)>2n>0,Re(a)>0$ and there exists a singularity at $x=ai$.
\end{example}
\begin{example}
Closed-form evaluation of a logarithmically resonant integral involving reciprocal $\pi^2-(2 \log(x))^2$. In this example we use equation (\ref{eq:legendre1}) and set $a\to \pi/2, b \to 1, v \to 2, n \to 1, q \to 1, z \to 1/2, m \to 1/2,s \to  -1/2$ and simplify.
\begin{multline}
\int_0^{\infty } \frac{(-1+x) (2+x)}{\sqrt{x} \left(1+x^2\right)^2 \left(\pi ^2-4 \log ^2(x)\right)} \, dx\\
=\frac{1}{8 \sqrt{2} \pi ^2}(-1-i) \left(\Phi \left(-i,2,\frac{1}{2}-\frac{i}{2}\right)-\Phi
   \left(-i,2,\frac{1}{2}+\frac{i}{2}\right)\right. \\ \left.
+3 i \left(\Phi \left(i,2,\frac{1}{2}-\frac{i}{2}\right)-\Phi \left(i,2,\frac{1}{2}+\frac{i}{2}\right)\right)\right)+\pi  \left(\,
   _2F_1\left(\frac{1}{2}-\frac{i}{2},1;\frac{3}{2}-\frac{i}{2};-i\right)\right. \\ \left.
+i \left(5 \, _2F_1\left(\frac{1}{2}-\frac{i}{2},1;\frac{3}{2}-\frac{i}{2};i\right)+\,
   _2F_1\left(\frac{1}{2}+\frac{i}{2},1;\frac{3}{2}+\frac{i}{2};-i\right)\right.\right. \\ \left.\left.
+5 i \, _2F_1\left(\frac{1}{2}+\frac{i}{2},1;\frac{3}{2}+\frac{i}{2};i\right)\right)\right)
\end{multline}
where there exists a singularity at $x=i e^{-\frac{\pi ^2}{4}}$.
\end{example}
\section{Discussion}
In this article we derived and evaluated several infinite integrals in terms of series involving special functions. Some of the work conducted involved evaluating integrals with singularities and Cauchy principal values. We also extended a few integral forms in terms of series which are listed in known book volumes.
\section{Conclusion}
Based on the work produced in this article we will use the contour integral method cited in this article to study other integral forms. We plan to evaluate other kernels involving special functions. The results in this work were checked using Mathematica by Wolfram. This software was also used to assist with formatting the mathematical equations into LaTeX.

\begin{thebibliography}{999}
%
\bibitem{schroder_dipert} Dipert, R.R. 
\emph{The life and work of Ernst Schröder}, . Modern Logic 1 (\textbf{1990/91}): 117–139. 
%
\bibitem{bdh}Bierens de Haan D. 
\emph{\textit{Nouvelles tables d’int\'{e}grales d'efinies}.}, (\textbf{1876}).
%

\bibitem{malmsten}Malmsten, C. J. (\textbf{1867}). 
 \emph{"Om definita integraler mellan imaginära gränsor"}. K. Vet. Akad. Handl. (in Swedish). 6 (3). Stockholm: P. A. Norstedt and Söner: 1–18.
 %
 
 \bibitem{grobner}W. Gröbner and N. Hofreiter (\textbf{1949}) 
 \emph{Integraltafel. Erster Teil. Unbestimmte Integrale.} Springer-Verlag, Vienna. 
 https://doi.org/10.1007/978-3-7091-7092-2
%

\bibitem{prud1}A. P. Prudnikov, Yu. A. Brychkov, and O. I. Marichev (\textbf{1986a}) 
\emph{Integrals and Series: Elementary Functions, Vol. 1}. Gordon and Breach Science Publishers, New York. 
https://doi.org/10.1201/9780203750643
%

\bibitem{erd_t1} A. Erdélyi, W. Magnus, F. Oberhettinger, and F. G. Tricomi (\textbf{1954a}) 
\emph{Tables of Integral Transforms. Vol. I.} McGraw-Hill Book Company, Inc., New York-Toronto-London.
%
\bibitem{erd_t2}A. Erdélyi, W. Magnus, F. Oberhettinger, and F. G. Tricomi (\textbf{1954b}) 
\emph{Tables of Integral Transforms. Vol. II}. McGraw-Hill Book Company, Inc., New York-Toronto-London.
%

\bibitem{ober_f}F. Oberhettinger (\textbf{1990}) 
\emph{Tables of Fourier Transforms and Fourier Transforms of Distributions}. Springer-Verlag, Berlin. 
https://doi.org/10.1007/978-3-642-74349-8
%
\bibitem{ober_m}F. Oberhettinger (\textbf{1974}) 
\emph{Tables of Mellin Transforms}. Springer-Verlag, Berlin-New York.
https://doi.org/10.1007/978-3-642-65975-1
%
\bibitem{brychkov_m}Brychkov, Y., Marichev, O., and Savischenko, N. (\textbf{2018}). 
\emph{Handbook of Mellin Transforms} (1st ed.). Chapman and Hall/CRC. https://doi.org/10.1201/9780429434259
%

\bibitem{moll_1}Moll, Victor Hugo (\textbf{2012}). 
\emph{Index of the papers in Revista Scientia with formulas from GR}. Retrieved 2016-02-17.
%
\bibitem{moll_2}Moll, Victor Hugo (\textbf{2014-10-01}). 
\emph{Special Integrals of Gradshteyn and Ryzhik: the Proofs – Volume I}. Monographs and Research Notes in Mathematics. Vol. I (1 ed.). Chapman and Hall/CRC Press/Taylor and Francis Group, LLC (published 2014-11-12). ISBN 978-1-4822-5651-2. Retrieved 2016-02-12.
https://doi.org/10.1201/b17674
%
\bibitem{moll_3}Moll, Victor Hugo (\textbf{2015-08-24}). 
\emph{Special Integrals of Gradshteyn and Ryzhik: the Proofs – Volume II}. Monographs and Research Notes in Mathematics. Vol. II (1 ed.). Chapman and Hall/CRC Press/Taylor and Francis Group, LLC (published 2015-10-27). ISBN 978-1-4822-5653-6. Retrieved 2016-02-12.
https://doi.org/10.1201/b19419
%

\bibitem{blagouchine}Blagouchine, I.V. 
\emph{Rediscovery of Malmsten’s integrals, their evaluation by contour integration methods and some related results}. Ramanujan J 35, 21–110 (\textbf{2014}). 
https://doi.org/10.1007/s11139-013-9528-5
%

\bibitem{ripon}Ripon, Sarowar Morshed. (\textbf{2014}). 
\emph{“Generalization of a Class of Logarithmic Integrals.”} Integral Transforms and Special Functions 26 (4): 229–45. 
doi:10.1080/10652469.2014.989390.
%

\bibitem{reyn5}Reynolds, Robert, and Allan Stauffer. (\textbf{2019}). 
\emph{"A Definite Integral Involving the Logarithmic Function in Terms of the Lerch Function"} Mathematics 7, no. 12: 1148. 
https://doi.org/10.3390/math7121148 
%

\bibitem{ahlfors}Ahlfors, Lars V. (\textbf{2000}). 
\emph{Complex Analysis}. McGraw-Hill series in Mathematics, McGraw-Hill. ISBN 0-07-000657-1
https://doi.org/10.1017/S0013091500012396
%

\bibitem{erd} Erd\'eyli, A.; Magnus, W.; Oberhettinger, F.; Tricomi, F.G.
\emph{Higher Transcendental Functions}; McGraw-Hill Book Company, Inc.: New York, NY, USA; Toronto, ON, Canada; London, UK, (\textbf{1953}); Volume I.
%

\bibitem{reyn4} Reynolds, R.; Stauffer, A.
\emph{A Method for Evaluating Definite Integrals in Terms of Special Functions with Examples}.  
\emph{Int. Math. Forum} (\textbf{2020}), \emph{15}, 235--244, 
doi:10.12988/imf.2020.91272 
%

\bibitem{grad} Gradshteyn, I.S.; Ryzhik, I.M.
\emph{Tables of Integrals, Series and Products}, 6th ed.; Academic Press: Cambridge, MA, USA, (\textbf{2006}).
 doi: 10.1016/C2010-0-64839-5
%

 \bibitem{reyn_aims}Robert Reynolds, Allan Stauffer. 
 \emph{Derivation of some integrals in Gradshteyn and Ryzhik}, [J]. AIMS Mathematics, 9, \textbf{2021}), 6(2): 1816-1821. 
doi: 10.3934/math.2021109
 %
 
 \bibitem{gelca} Gelca, R\u{a}zvan., Andreescu, Titu.,
\emph{Putnam and Beyond}. Germany: Springer International Publishing, (\textbf{2017}).
https://doi.org/10.1007/978-3-319-58988-6
 %
 
 \bibitem{kolbig}K.S. K\"{o}lbig., 
\emph{Explicit evaluation of certain definite integrals involving powers of logarithms}, Journal of Symbolic Computation,Volume 1, Issue 1,(\textbf{1985}),Pages 109-114,
https://doi.org/10.1016/S0747-7171(85)80032-8.
%

\bibitem{malmsten_specimen}Malmsten, C. Johan. (\textbf{1842}). 
\emph{Specimen analyticum: theoremata quaedam nova de integralibus definitis, summatione serierum earumque in alias series transformatione exhibens },  Upsaliae. 
%

\bibitem{mobius}M\"{o}bius, A. F. (\textbf{1834}). 
\emph{"Beweis der Gleichung $0^0 = 1$, nach J. F. Pfaff" [Proof of the equation $0^0 = 1$, according to J. F. Pfaff]}. Journal für die reine und angewandte Mathematik (in German). (\textbf{1834}) (12): 134–136. 
doi:10.1515/crll.1834.12.134. S2CID 199547186
%

\bibitem{schroder}Ernst Schr\"{o}der, Zeitschrift fur Mathematik und Physik, vol. 25, pp. 106–117 (\textbf{1880})
%

\bibitem{halmos}Halmos, P.R. (\textbf{1974}). 
\emph{The Schröder-Bernstein Theorem. In: Naive Set Theory}. Springer, New York, NY
https://doi.org/10.1007/978-1-4757-1645-0_22
%

\bibitem{schroder_fe}Schröder, E. (\textbf{1870}). Über iterierte Funktionen III. Mathematische Annalen, 3, 296–322.
%

\bibitem{schroder_logic}Schröder, Ernst. Vorlesungen über die Algebra der Logik. Vol. 1. Leipzig: B. G. Teubner, (\textbf{1890})
%

\bibitem{jordan}Jordan, C. (\textbf{1965}). 
\emph{"Gregory's Summation Formula."} §99 in Calculus of Finite Differences, 3rd ed. New York: Chelsea, pp. 284-287.
%

\bibitem{wood}Van E. Wood, 
\emph{Some integrals of Ramanujan and related contour integrals}, Math. Comp. 20 (\textbf{1966}), 424-429
https://doi.org/10.1090/S0025-5718-1983-0701632-1
%

\bibitem{nahin}Paul J. Nahin, 
\emph{Inside Interesting Integrals}, Springer Cham, Springer Nature Switzerland AG (\textbf{2020}), 2, XLVII, 503
https://doi.org/10.1007/978-3-030-43788-6
%

\bibitem{luo}Luo, ZQ., Luo, QM. 
\emph{A q-analogue for partial-fraction decomposition of a rational function and its application.} Adv Cont Discr Mod 2024, 18 (\textbf{2024}). 
https://doi.org/10.1186/s13662-024-03814-7
%

\bibitem{berndt_ram3}Bruce C. Berndt, 
\emph{Ramanujan's Notebooks, Part III}, Springer-Verlag, (\textbf{1991})
https://dn720005.ca.archive.org/0/items/ramanujan_notebooks_iii/ramanujan_notebooks_iii.pdf
%

\bibitem{berndt_ram4}Berndt, B. C., and Dixit, A. (\textbf{2021}). 
\emph{Ramanujan's Beautiful Integrals}. Hardy-Ramanujan Journal, 43, 69-82
https://doi.org/10.46298/hrj.2021.7429
%

\bibitem{ramanujan_messenger}Srinivasa Ramanujan, 
\emph{Some Definite Integrals}, Messenger of Mathematics, XLIV, (\textbf{1915}), 10 – 18
http://ramanujan.sirinudi.org/Volumes/published/ram11.pdf
%

\bibitem{prud2}A. P. Prudnikov, Yu. A. Brychkov, and O. I. Marichev (\textbf{1986b}) 
\emph{Integrals and Series: Special Functions, Vol. 2}. Gordon and Breach Science Publishers, New York. 
https://doi.org/10.2307/2007975
%

\bibitem{prud3}A. P. Prudnikov, Yu. A. Brychkov, and O. I. Marichev (\textbf{1990}) 
\emph{Integrals and Series: More Special Functions, Vol. 3.} Gordon and Breach Science Publishers, New York. 
https://doi.org/10.1201/9780203750643
%

\bibitem{andrews}Andrews, G.E., Askey, R. and Roy, R. (\textbf{2006)} 
\emph{Special Functions}. Cambridge University Press, Cambridge.
https://doi.org/10.1017/CBO9781107325937
%

\bibitem{brychkov}Brychkov, Y.A. (\textbf{2008}). 
\emph{Handbook of Special Functions: Derivatives, Integrals, Series and Other Formulas} (1st ed.). Chapman and Hall/CRC. 
https://doi.org/10.1201/9781584889571
%

\bibitem{ramanujan_qj}S. Ramanujan, 
\emph{A class of definite integrals}, Quart. J. Math. 48 (\textbf{1920}), 294-310.
http://ramanujan.sirinudi.org/Volumes/published/ram27.html
%
%
\end{thebibliography}
\end{document}